\documentclass[11pt]{article}

\usepackage{amsmath, amssymb, amscd, amsthm, amsfonts}
\usepackage{graphicx}
\usepackage{hyperref}

\usepackage{graphicx}

\usepackage{subfigure}
\usepackage[capitalize,noabbrev]{cleveref}
\usepackage[comma,square]{natbib}
\usepackage{booktabs}

\usepackage{algorithm}
\usepackage{algorithmic}

\setcitestyle{numbers}

\newcommand\keywords[1]{\textbf{Keywords}: #1}

\title{Convergence Analysis for Restarted Anderson Mixing and Beyond}
\author{Fuchao Wei $^1$, Chenglong Bao $^{3,4}$, Yang Liu $^{1,2}$, and Guangwen Yang $^1$ \\
 $^1$Department of Computer Science and Technology, Tsinghua University\\
 $^2$Institute for AI Industry Research (AIR), Tsinghua University\\
 $^3$Yau Mathematical Sciences Center, Tsinghua University\\
 $^4$Yanqi Lake Beijing Institute of Mathematical Sciences and Applications  \\
\texttt{wfc16@mails.tsinghua.edu.cn, \{clbao,liuyang2011,ygw\}@tsinghua.edu.cn}
}

\date{}

\newtheorem{theorem}{Theorem}[section]
\newtheorem{lemma}[theorem]{Lemma}

\newtheorem{proposition}[theorem]{Proposition}
\newtheorem{assumption}[theorem]{Assumption}
\newtheorem{corollary}[theorem]{Corollary}

\theoremstyle{definition}
\newtheorem{definition}[theorem]{Definition}

\theoremstyle{remark}
\newtheorem{remark}[theorem]{Remark}

\numberwithin{equation}{section}

\begin{document}

\maketitle

\begin{abstract}
Anderson mixing (AM) is a classical method that can accelerate fixed-point iterations by exploring historical information. Despite the successful application of AM in scientific computing, the theoretical properties of AM are still under exploration. In this paper, we study the restarted version of the Type-I and Type-II AM methods, i.e., restarted AM. With a multi-step analysis, we give a unified convergence analysis for the two types of restarted AM and justify that the restarted Type-II AM can locally improve the convergence rate of the fixed-point iteration. Furthermore, we propose an adaptive mixing strategy by estimating the spectrum of the Jacobian matrix. If the Jacobian matrix is symmetric, we develop the short-term recurrence forms of restarted AM to reduce the memory cost. Finally, experimental results on various problems validate our theoretical findings.
\end{abstract}


\section{Introduction} \label{sec:intro}
%
%
 Anderson mixing (AM) \cite{anderson1965iterative},  also known as Anderson acceleration \cite{walker2011anderson}, or Pulay mixing, 
 DIIS method in quantum chemistry~\cite{pulay1980convergence,pulay1982improved,rohwedder2011analysis},  is a classical extrapolation method  for accelerating  fixed-point iterations \cite{brezinski2018shanks} and has wide applications in scientific computing \cite{arora2017accelerating,Pratapa2016Anderson,Ho2017Acce,pollock2019anderson,yang2021anderson}. 
  Consider a fixed-point problem
 \begin{align}
  x = g(x),     \label{eq:fixed-point}
 \end{align}
 where $x\in \mathbb{R}^d$ and $g: \mathbb{R}^d\rightarrow \mathbb{R}^d$. 
 The conventional fixed-point iteration
 \begin{align}
  x_{k+1} = g(x_k), ~ k=0,1,\dots,   \label{eq:picard}
 \end{align}
 converges if $g$ is contractive. 
 To accelerate the convergence of \eqref{eq:picard}, AM generates each iterate by the extrapolation of historical steps. Specifically, let $m_k\geq 0$ be the size of the used historical sequences at the $k$-th iteration. AM obtains $x_{k+1}$ via
 \begin{align}
 	x_{k+1} = (1-\beta_k)\sum_{j=0}^{m_k}\alpha_k^{(j)} x_{k-m_k+j} +\beta_k\sum_{j=0}^{m_k}\alpha_k^{(j)}g(x_{k-m_k+j}),  \label{eq:extrapolation}
 \end{align}
 where $\beta_k>0$ is the mixing parameter, and the extrapolation coefficients $\lbrace \alpha_k^{(j)} \rbrace_{j=0}^{m_k}$ are determined by solving a constrained least squares problem:
 \begin{align}
 	\min_{\lbrace \alpha_k^{(j)} \rbrace_{j=0}^{m_k}} \left\| \sum_{j=0}^{m_k}\alpha_k^{(j)}\left(g(x_{k-m_k+j})-x_{k-m_k+j}\right) \right\|_2 ~~ \mbox{s.t.} ~~ \sum_{j=0}^{m_k}\alpha_k^{(j)} = 1.  \label{eq:optim}
 \end{align}
 There are several approaches to choosing $m_k$. For example, the full-memory AM chooses $m_k = k$, i.e., using the whole historical sequences for one extrapolation; the limited-memory AM sets $m_k = \min\lbrace m,k\rbrace$, where $m\geq 1$ is a constant integer.  
 

     
 For solving systems of equations where the fixed-point iterations are slow in convergence, AM is a practical alternative to Newton's method when handling the Jacobian matrices is difficult \cite{kelley2018numerical,brezinski2020shanks}. It has been recognized that AM is a multisecant quasi-Newton method that implicitly updates the approximation of the inverse Jacobian matrix to satisfy multisecant equations \cite{fang2009two}. Also, another type of AM called Type-I AM was introduced in \cite{fang2009two}. Different from the original AM (also called Type-II AM), the Type-I AM directly approximates the Jacobian matrix. Both types of AM have been adapted to solve various fixed-point problems \cite{Higham2016Anderson,mai2020anderson,zhang2020globally,fu2020anderson,sun2021damped}. 

  Motivated by the promising numerical performance in many applications, the theoretical analysis of AM methods has become an important topic.
 For solving linear systems, it turns out that both types of full-memory AM methods are closely related to Krylov subspace methods \cite{walker2011anderson}. However, for solving nonlinear problems, the theoretical properties of AM are still vague. 
 For the Type-II AM, the known results in \cite{toth2015convergence,Toth2017Local,chen2019convergence,bian2021anderson} show that the limited-memory version has a local linear convergence rate that is no worse than that of the  fixed-point iteration.  Recent works \cite{evans2020proof,pollock2021anderson}  further point out that the potential improvement of AM over fixed-point iterations depends on the quality of extrapolation, which is determined during iterations.  For the Type-I AM, whether similar results hold remains unclear.  
  It is worth noting that these theoretical results of the limited-memory AM follow the conventional one-step analysis, which may only have a partial assessment of the efficacy of AM, as also commented by Anderson in his review \cite{anderson2019comments}.
  A fixed-point analysis in \cite{de2022Linear} reveals the continuity and differentiability properties of the Type-II AM iterations, but the convergence still lacks theoretical quantification.
  Besides,  some new variants of AM have been developed and analyzed in different settings, e.g., see \cite{Scieur2020Reg,de2021asymptotic,WeiBL21,bian2022anderson,wei2022}. 

 In this paper, we apply a multi-step analysis to investigate the long-term convergence behaviour of AM for solving nonlinear fixed-point problems. We focus on the restarted version of AM, i.e., restarted AM, where the method clears the  historical information and restarts when some restarting condition holds. Restart is a common approach to improving the stability and robustness of AM \cite{fang2009two,zhang2020globally,he2022gdaam,Garstka2022safeguarded}. 
 Compared with the limited-memory AM, the restarted AM has the benefit that it is more amenable to extending the relationship between AM methods and Krylov subspace methods to nonlinear problems. Based on such a relationship, we establish the convergence properties of both types of restarted AM methods which explain the efficacy of AM in practice. Furthermore, by investigating the properties of restarted AM, we obtain an efficient procedure to estimate the eigenvalues of the Jacobian matrix that is beneficial for choosing the mixing parameters; for problems with symmetric Jacobian matrices, we derive the short-term recurrence forms of AM. We highlight our main contributions as follows.

 \begin{enumerate}
 \item We formulate the restarted Type-I and Type-II AM methods with certain restarting conditions and give a unified convergence analysis for both methods. Our multi-step analysis justifies that the restarted Type-II AM method can locally improve the convergence rate of the fixed-point iteration.
  \item  We propose an adaptive mixing strategy that adaptively chooses the mixing parameters by estimating the eigenvalues of the Jacobian matrix.  The eigenvalue estimation procedure originates from the projection method for eigenvalue problems and can be efficiently implemented using historical information. We also discuss the related theoretical properties. 
  \item  We show that the restarted AM methods can be simplified to have short-term recurrences if the Jacobian matrix is symmetric, which can reduce the memory cost. We give the convergence analysis of the short-term recurrence methods and
   develop the corresponding adaptive mixing strategy. 
 \end{enumerate}
 
 
 {\em Notations}. The operator $\Delta$ denotes the forward difference, e.g.,  $\Delta x_k = x_{k+1}-x_k$. $h'$ is the Jacobian of a function $h:\mathbb{R}^d\rightarrow\mathbb{R}^d$. For every matrix $A$, 
  ${\rm range}(A) $ is the subspace spanned by the columns of $A$; $\mathcal{K}_k(A,v):={\rm span}\lbrace v,Av,\dots,A^{k-1}v \rbrace$ is the $k$-th Krylov subspace generated by $A$ and a vector $v$; $\mathcal{S}(A):=(A+A^\mathrm{T})/2$ is the symmetric part of $A$; $\sigma(A)$ is the spectrum of $A$; $\|A\|_2$ is the spectral norm of $A$; $\|x\|_A:=(x^\mathrm{T}Ax)^{1/2}$ is the $A$-norm if $A$ is symmetric positive definite (SPD). 
$\mathcal{P}_k$ denotes the space of polynomials of degree not exceeding $k$.
 
\section{Two types of Anderson mixing methods} \label{sec:am}
 We re-interpret each iteration of the Type-I/Type-II AM method as a two-step procedure following \cite{WeiBL21}. 
 Define $r_k=g(x_k)-x_k$ to be the {\em residual} at $x_k$. The historical sequences are stored as two matrices $X_k, R_k \in \mathbb{R}^{d\times m_k}$ $(m_k\geq 1)$:
 \begin{equation}
 \begin{aligned}
   X_k &= (
		\Delta x_{k-m_k} , \Delta x_{k-m_k+1} , \dots , \Delta x_{k-1} ),\\
   R_k &= (
		\Delta r_{k-m_k} , \Delta r_{k-m_k+1} , \dots , \Delta r_{k-1} ). \label{Xk_Rk}
		\end{aligned}
 \end{equation}
 Both Type-I and Type-II AM obtain $x_{k+1}$ via a {\em projection step} and a {\em mixing step}:
\begin{equation}
\begin{aligned}
     \bar{x}_k &= x_k - X_k\Gamma_k,  \quad \quad \mbox{(Projection step)}\\
     \bar r_k &= r_k - R_k\Gamma_k,\\
     x_{k+1} &= \bar{x}_k + \beta_k\bar{r}_k, \quad \quad \mbox{ (Mixing step)}
\label{projection_mixing}
\end{aligned}   
\end{equation}
 where $\beta_k>0$ is the mixing parameter. For convenience, let $Z_k:=X_k$ for the Type-I AM and $Z_k:=R_k$ for the Type-II AM, then $\Gamma_k$ is determined by the condition 
 \begin{align}
  \bar{r}_k \perp {\rm range}(Z_k).		 \label{eq:cond}
 \end{align}
 Assume $Z_k^\mathrm{T}R_k$ is nonsingular. 
 From \eqref{projection_mixing}, $x_{k+1} = x_k + \beta_k r_k - \left( X_k + \beta_k R_k \right)\Gamma_k$.  With the solution $\Gamma_k$ from \eqref{eq:cond}, we obtain 
 \begin{align}
  x_{k+1} = x_k+G_kr_k, \ \mbox{where} \  G_k = \beta_k I -(X_k+\beta_kR_k)(Z_k^\mathrm{T}R_k)^{-1}Z_k^\mathrm{T}.   \label{eq:am_update}
 \end{align}
 For the Type-I AM, $G_k$ satisfies $G_k = J_k^{-1}$, where $J_k$ solves  $\min_{J}\|J-\beta_k^{-1}I\|_F$ s.t. $J X_k = -R_k$; For the Type-II AM, $G_k$ solves $\min_{G}\|G-\beta_kI\|_F$ s.t. $GR_k=-X_k$. Hence, both methods can be viewed as multisecant quasi-Newton methods \cite{fang2009two}. 
  \begin{remark}   
  For the Type-II method, the condition \eqref{eq:cond} is equivalent to $\Gamma_k = \arg\min_{\Gamma\in\mathbb{R}^{m_k}} \| r_k-R_k\Gamma\|_2 $. 
  Let $\Gamma_k = (\Gamma_k^{(1)},\dots,\Gamma_k^{(m_k)})^\mathrm{T}\in\mathbb{R}^{m_k}$.  The extrapolation coefficients $\lbrace\alpha_k^{(j)}\rbrace$ can be obtained from $\Gamma_k$: $\alpha_k^{(0)}=\Gamma_k^{(1)},  \alpha_k^{(j)}=\Gamma_k^{(j+1)}-\Gamma_k^{(j)} (j=1,\dots,m_k-1), \alpha_k^{(m_k)}=1-\Gamma_k^{(m_k)}$. 
  Then $r_k-R_k\Gamma_k = \sum_{j=0}^{m_k}\alpha_k^{(j)}r_{k-m_k+j}$.  
 The above formulation of Type-II AM is equivalent to that given by \eqref{eq:extrapolation} and \eqref{eq:optim}.
  \end{remark}

\section{Restarted Anderson mixing} \label{sec:restarted_am}
 Initialized with $m_0 = 0$,  the restarted AM sets $m_k = m_{k-1}+1$ if no restart occurs and sets $m_k=0$ if a restarting condition is satisfied, similar to the restarted GMRES \cite{saad1986gmres}. 
 Thus, the restarting conditions are critical for the method.
 To define the restarting conditions, we first construct modified historical sequences. Such modification does not alter the iterates but is essential for the following analysis.

 
\subsection{The AM update with modified historical sequences} \label{subsec:modified_seqs}
 Consider the nontrivial case that $m_k>0$. Note that the $G_k$ in \eqref{eq:am_update} does not change if we replace $X_k, R_k$ by $P_k := X_kS_k^{-1}, Q_k := R_kS_k^{-1}$, where $S_k\in\mathbb{R}^{m_k\times m_k}$ is nonsingular. So we can choose some suitable transformation $S_k$ to reformulate the AM update. 
 
 We construct the modified historical sequences  $ P_k = (p_{k-m_k+1},\dots,p_k), Q_k = (q_{k-m_k+1},\dots,q_k)$ in a recursive way. Let $V_k:=P_k$ for the Type-I AM and $V_k:=Q_k$ for the Type-II AM.  
 Assume that $\det(Z_j^\mathrm{T}R_j)\neq 0$ for $j=k-m_k+1,\dots,k$. 
 The AM update with modified historical sequences consists of the following two steps.
 
 {\em Step~1: Modified vector pair.} 
 If $m_k = 1$, then $p_k = \Delta x_{k-1}, q_k = \Delta r_{k-1}$.  
 If $m_k \geq 2$, we set the vector pair $p_k, q_k$ as
 \begin{align}
  p_k = \Delta x_{k-1}-P_{k-1}\zeta_k, ~~ q_k = \Delta r_{k-1}-Q_{k-1}\zeta_k,  \label{eq:q_k_update}
 \end{align}
 where $\zeta_k = (\zeta_k^{(1)},\dots,\zeta_k^{(m_k-1)})^\mathrm{T}$ is determined by $q_k \perp {\rm range}(V_{k-1})$. 

 {\em Step~2: AM update.}  We obtain $x_{k+1}$ via
 \begin{align}
  \bar{x}_k = x_k-P_k\Gamma_k, ~~ \bar{r}_k = r_k-Q_k\Gamma_k, ~~ x_{k+1} = \bar{x}_k+\beta_k\bar{r}_k,  \label{am_update}
 \end{align}
 where $\Gamma_k =(\Gamma_k^{(1)},\dots,\Gamma_k^{(m_k)})^\mathrm{T}$ is determined by $\bar{r}_k \perp {\rm range}(V_k)$. 
  
  It can be verified by induction that the above process produces the same iterates as \eqref{eq:am_update}.  
  To facilitate the analysis, we give explicit procedures to obtain  $\zeta_k$ and $\Gamma_k$.  
   
   Let $Z_k = (z_{k-m_k+1},\dots,z_k)$, $V_k = (v_{k-m_k+1},\dots,v_k)$. 
   We first describe the procedure to compute $\zeta_k$ and $q_k$. 
    Define $q_k^{0} = \Delta r_{k-1}$. For $j=1,2,\ldots,m_k-1$, the procedure computes $\zeta_k^{(j)}$ and the intermediate vector $q_k^j$ sequentially: 
  \begin{equation}
  \begin{aligned}
    & \zeta_k^{(j)} = \frac{v_{k-m_k+j}^\mathrm{T}q_k^{j-1}}{v_{k-m_k+j}^\mathrm{T}q_{k-m_k+j}},\quad q_k^{j} = q_k^{j-1}-q_{k-m_k+j}\zeta_k^{(j)}.  
 \end{aligned}
 \label{eq:q_k_type2}
  \end{equation}
 Then $q_k = q_k^{m_k-1}$. 
 Next, $\Gamma_k$ and $\bar{r}_k$ can be computed similarly. Define $r_k^{0} = r_k$. For $j=1,2,\ldots,m_k$, the $\Gamma_k^{(j)}$ and the intermediate vector $r_k^j$ are computed sequentially: 
 \begin{equation}
       \Gamma_k^{(j)} = \frac{v_{k-m_k+j}^\mathrm{T}r_k^{j-1}}{v_{k-m_k+j}^\mathrm{T}q_{k-m_k+j}},\quad r_k^j = r_k^{j-1}-q_{k-m_k+j}\Gamma_k^{(j)}. \label{eq:r_k_type2}    \\
 \end{equation}
 Then $\bar{r}_k = r_k^{m_k}$. 
 Procedures \eqref{eq:q_k_type2} and \eqref{eq:r_k_type2} are reminiscent of  the modified Gram-Schmidt orthogonalization process that is recommended for the implementation of Type-II AM \cite{anderson2019comments}. 
The next proposition shows the correctness of the above procedures. 
 \begin{proposition} \label{prop:modified_seqs}
  Suppose that $\det(Z_j^\mathrm{T}R_j) \neq 0$ for $j=k-m_k+1,\dots,k$. Then 
  the procedures \eqref{eq:q_k_type2} and \eqref{eq:r_k_type2} are well defined, and the following properties hold:
 \begin{enumerate}
   \item $X_k = P_kS_k, R_k = Q_kS_k$, where $S_k$ is unit upper triangular;  \label{am_update:property1}
   \item $V_k^\mathrm{T}Q_k$ is lower triangular; \label{am_update:property2}
   \item $\bar{r}_k \perp {\rm range}(V_k)$.  \label{am_update:property3}
 \end{enumerate}
 The scheme \eqref{am_update} produces the same $\lbrace x_j \rbrace_{j=k-m_k+1}^{k+1}$ as the original AM update \eqref{eq:am_update}. 
 \end{proposition}  
 
 The proof is given in \Cref{subsec:proof:prop:modified_seqs}.  
 It is worth noting that our formulation of the restarted AM focuses on theoretical analysis. Better implementations are needed in some specific scenarios, e.g., parallel computing. 

\subsection{Restarting conditions}  \label{subsec:restart_conds}
Let $\tau\in (0,1), \eta > 0$, and $ m\in(0,d]$ is an integer. 
 Following \cite{wei2022}, 
the restart criterion is related to  the following conditions: 
 \begin{align}
  & m_k \leq m,   \label{ineq:cond1}  \\
  & \vert v_k^\mathrm{T}q_k \vert \geq \tau \vert v_{k-m_k+1}^\mathrm{T}q_{k-m_k+1} \vert,  \label{ineq:cond2} \\
  & \|r_k\|_2 \leq \eta \|r_{k-m_k}\|_2. \label{ineq:cond3}
 \end{align}
 If any condition in~\eqref{ineq:cond1}-\eqref{ineq:cond3} is violated during the iteration,  set $m_k = 0$ and restart the method.  Details of the restarted AM are given in \Cref{alg:restartedAM}. Next, we explain the rationale behind the above three conditions.
 
  \begin{algorithm}[ht]
\caption{Restarted Anderson mixing for solving the fixed-point problem \eqref{eq:fixed-point}. The Type-I method: $v_j:=p_j$ $(j\geq 1)$; the Type-II method: $v_j:=q_j$ $(j\geq 1)$. }
\label{alg:restartedAM}
\textbf{Input}: $ x_0\in\mathbb{R}^d, \beta_k>0, m\in\mathbb{Z}_+, \tau\in(0,1), \eta >0 $\\
\textbf{Output}: $ x\in\mathbb{R}^d $
\begin{algorithmic}[1] 
\STATE $ m_0=0$
\FOR{$k=0,1,\dots, $ until convergence}
 \STATE $ r_k=g(x_k)-x_k $
 \IF { $m_k>m$ \textbf{or} $\|r_k\|_2 > \eta \|r_{k-m_k}\|_2$}
 \STATE $m_k = 0$
 \ENDIF
 \IF{$m_k>0$} 
  \STATE $p_k=x_k-x_{k-1}, ~~ q_k = r_k-r_{k-1}$  
  \FOR{$j=1,\dots,m_k-1$}
   \STATE $\zeta = \left(v_{k-m_k+j}^\mathrm{T}q_k\right)/\left(v_{k-m_k+j}^\mathrm{T}q_{k-m_k+j}\right)$
   \STATE $p_k=p_k-p_{k-m_k+j}\zeta, ~~ q_k = q_k-q_{k-m_k+j}\zeta$
  \ENDFOR  
  \IF{ $\vert v_k^\mathrm{T}q_k\vert < \tau \vert v_{k-m_k+1}^\mathrm{T}q_{k-m_k+1}\vert$ }
  \STATE $m_k = 0$
  \ENDIF
 \ENDIF
 \STATE $\bar{x}_k = x_k, ~~ \bar{r}_k = r_k$
 \FOR{$j=1,\dots,m_k$}
  \STATE $\gamma = \left(v_{k-m_k+j}^\mathrm{T}\bar{r}_k\right)/\left(v_{k-m_k+j}^\mathrm{T}q_{k-m_k+j}\right)$
  \STATE $\bar{x}_k = \bar{x}_k-p_{k-m_k+j}\gamma, ~~ \bar{r}_k = \bar{r}_k-q_{k-m_k+j}\gamma$
 \ENDFOR
 \STATE $x_{k+1} = \bar{x}_k+\beta_k\bar{r}_k$ 
  \STATE $m_{k+1} = m_k+1$ 
\ENDFOR
\STATE \textbf{return} $ x_k $
\end{algorithmic}
\end{algorithm}

 The first condition \eqref{ineq:cond1} limits the size of the historical sequences, which plays an important role in bounding the accumulated high-order errors in the convergence analysis. The second condition \eqref{ineq:cond2} ensures the nonsingularity of $V_k^\mathrm{T}Q_k$ as long as $v_{k-m_k+1}^\mathrm{T}q_{k-m_k+1} \neq 0$. This is because $V_k^\mathrm{T}Q_k$ is lower triangular and the diagonal elements
 $\lbrace v_j^\mathrm{T}q_j\rbrace_{j=k-m_k+1}^k$ are nonzero due to \eqref{ineq:cond2}. Also, \eqref{ineq:cond2} controls the condition number of $V_k^\mathrm{T}Q_k$ by the following lower bound:
 \begin{equation}
  \frac{\vert v_{k-m_k+1}^\mathrm{T}q_{k-m_k+1} \vert}{\vert v_k^\mathrm{T}q_k \vert} 
  = \frac{\vert e_1^\mathrm{T}V_k^\mathrm{T}Q_k e_1 \vert}{\vert e_{m_k}^\mathrm{T}V_k^\mathrm{T}Q_k e_{m_k} \vert} \leq \|V_k^\mathrm{T}Q_k\|_2 \|\left( V_k^\mathrm{T}Q_k \right)^{-1}\|_2, 
 \end{equation}
 where $e_j$ denotes the $j$-th column of the identity matrix $I_{m_k}$.  Thus, a too-small $\vert v_k^\mathrm{T}q_k\vert$ can cause numerical instability and we have to restart the AM method. The third condition \eqref{ineq:cond3} is to control the growth degree of the residuals, which avoids the problematic behaviour of AM and can be seen as a safeguard condition. Moreover, as shown in our proof, the conditions \eqref{ineq:cond1}-\eqref{ineq:cond3} can lead to the boundedness of the extrapolation coefficients, which is a critical assumption in  \cite{toth2015convergence}.


 \section{Convergence analysis}   \label{sec:convergence}
 In this section, we give a unified convergence analysis for the restarted AM methods described in \Cref{alg:restartedAM}. 
 We first recall the relationship between AM methods and the Krylov subspace methods for solving linear systems. 
 Let $x_k^{\rm A}$ and $x_k^{\rm G}$ denote the $k$-th iterate of Arnoldi's method  \cite{saad1981krylov} and the $k$-th iterate of GMRES \cite{saad1986gmres}, respectively. We summarize the results if \eqref{eq:fixed-point} is linear.
 \begin{proposition}  \label{prop:linear}
 Consider the fixed-point problem \eqref{eq:fixed-point} with $g(x) =(I-A)x+b$, where $A\in\mathbb{R}^{d\times d}$ is nonsingular and $b\in\mathbb{R}^{d}$. Let  $\lbrace x_k\rbrace$ be the sequence generated by the full-memory Type-I/Type-II AM method with nonzero mixing parameters. If $\det(Z_j^\mathrm{T}R_j) \neq 0$ for $j=1,\dots,k$, then the following relations hold:
  \begin{enumerate}
   \item $R_k = -AX_k, ~ {\rm range}(X_k) = \mathcal{K}_k(A,r_0)$;   \label{prop:linear:property1}
   \item for the Type-I AM method, $\bar{x}_k = x_k^{\rm A}$ provided that $x_0 = x_0^{\rm A}$;  \label{prop:linear:property2}
   \item for the Type-II AM method, $\bar{x}_k = x_k^{\rm G}$ provided that $x_0 = x_0^{\rm G}$.   \label{prop:linear:property3}
  \end{enumerate}
 Furthermore, if $A$ is positive definite and $r_j \neq 0$, $j=0,\dots,k$, then $\det(Z_j^\mathrm{T}R_j) \neq 0, $ $j=1,\dots,k$;  the constructions of the modified historical sequences $P_k$ and $Q_k$ are well-defined, and 
 \begin{equation}
   Q_k = -AP_k, ~~  {\rm range}(P_k) = {\rm range}(X_k) = \mathcal{K}_k(A,r_0).   \label{prop:linear:modified_seqs}
 \end{equation}
 \end{proposition} 
 
 We give the proof in \Cref{subsec:proof:prop:linear}.   Properties~\ref{prop:linear:property1}-\ref{prop:linear:property3} are known  results \cite{walker2011anderson}. 
 \Cref{prop:modified_seqs} and \Cref{prop:linear} establish the relationship between the restarted AM and Krylov subspace methods in the linear case. 
 
Now, we study the convergence properties of the restarted AM for solving nonlinear problems. Rewriting the fixed-point problem \eqref{eq:fixed-point} as $h(x):=x-g(x) = 0$, we make the following assumptions on $h$:
\begin{assumption}   \label{assum:h}
 (i) There exists $x^*$ such that $h(x^*)=0$; (ii) $h$ is Lipschitz continuously differentiable in a neighbourhood of $x^*$;
 (iii) The Jacobian $h'(x^*)$ is positive definite, i.e., all the eigenvalues of $\mathcal{S}(h'(x^*))$ are positive. 
\end{assumption}
 

From \cref{assum:h}, there exist positive constants $\hat{\rho}, \hat{\kappa}, \mu,$ and $L$ such that for all $x\in \mathcal{B}_{\hat{\rho}}(x^*):= \lbrace z\in\mathbb{R}^d\vert \|z-x^*\|_2\leq \hat{\rho}\rbrace$, the following relations hold:
 \begin{align}
  & \mu\|y\|_2 \leq \|h'(x)y\|_2 \leq L\|y\|_2,  \ \ \forall y\in\mathbb{R}^d;   \label{assum:jacobian} \\
  & \mu\|y\|_2^2 \leq y^\mathrm{T}h'(x)y \leq L\|y\|_2^2,  \ \ \ \forall y\in\mathbb{R}^d;   \label{assum:hermite_part}  \\
  & \|h(x)-h'(x^*)(x-x^*)\|_2 \leq \frac{1}{2}\hat{\kappa}\|x-x^*\|_2^2.  \label{assum:jacobian_continue} 
 \end{align} 
 
 Inspired by the proofs of the restarted conjugate gradient methods \cite{cohen1972rate,lenard1976convergence} and the cyclic Barzilai-Borwein method \cite{dai2006cyclic}, we establish the convergence properties of the restarted AM methods from their properties in the linear problems. To achieve this goal, 
 we first introduce the local linear model of $h$ around $x^*$: 
 \begin{equation}
  \hat{h}(x) = h'(x^*)(x-x^*), 
\end{equation}  
 which deviates from $h(x)$ by at most a second-order term $\frac{1}{2}\hat{\kappa}\|x-x^*\|_2^2$ in $\mathcal{B}_{\hat{\rho}}(x^*)$ from \eqref{assum:jacobian_continue}. 
 Then we construct two sequences of iterates  $\lbrace x_k\rbrace$ and $\lbrace \hat{x}_k \rbrace$, which are associated with solving $h(x)=0$ and $\hat{h}(x)=0$, respectively.
 \begin{definition}  \label{definition:two_processes}
 Let the mixing parameters $\{\beta_k\}$ satisfy $\beta\leq|\beta_k|\leq \beta'$ for positive constants $\beta$ and $\beta'$.  
 The sequences $\lbrace x_k\rbrace$ and $\lbrace \hat{x}_k \rbrace$ are generated by two processes:
 
 (i) Process~I: Solve the fixed-point problem \eqref{eq:fixed-point} with the restarted Type-I/Type-II AM method (see \cref{alg:restartedAM}), and the resulting sequence is $\lbrace x_k\rbrace$. 
 
 (ii) Process~II: In each interval between two successive restarts in Process~I, apply the full-memory Type-I/Type-II AM with modified historical sequences to solve the linear system $\hat{h}(x) = 0$.  Specifically, let $m_k$ and $\beta_k$ be the same ones in Process~I and define $\hat{r}_k = -\hat{h}(\hat{x}_k)$. The iterates are given as follows:
  \begin{equation}
 \begin{aligned}
   & \hat{x}_k = x_k,  ~~ \mbox{and} ~~ \hat{x}_{k+1} = \hat{x}_k+\beta_k\hat{r}_k, ~~ \mbox{if} ~~ m_k = 0;   \\
   & \hat{x}_{k+1} = \hat{x}_k + \beta_k\hat{r}_k-\left(\hat{P}_k+\beta_k\hat{Q}_k\right)\hat{\Gamma}_k,  ~~ \mbox{if} ~~ m_k>0,    \label{eq:process2}
 \end{aligned}
 \end{equation}
 where $\hat{\Gamma}_k$ is chosen such that $\hat{r}_k-\hat{Q}_k\hat{\Gamma}_k \perp {\rm range}(\hat{V}_k)$. Here $\hat{P}_k = (\hat{p}_{k-m_k+1},\dots,\hat{p}_k)$ and $\hat{Q}_k = (\hat{q}_{k-m_k+1},\dots,\hat{q}_k)$ are the modified historical sequences. Let $\hat{V}_k = \hat{P}_k$ if the Type-I method is used in Process~I, and $\hat{V}_k = \hat{Q}_k$ if the Type-II method is used in Process~I. Then, $\hat{p}_k = \Delta \hat{x}_{k-1}$, 
 $\hat{q}_k = \Delta \hat{r}_{k-1}$, if $m_k = 1$; $\hat{p}_k = \Delta \hat{x}_{k-1}-\hat{P}_{k-1}\hat{\zeta}_k$, $\hat{q}_k = \Delta \hat{r}_{k-1}-\hat{Q}_{k-1}\hat{\zeta}_k$, if $m_k\geq 2$, where $\hat{\zeta}_k$ is chosen such that $ \hat{q}_k \perp {\rm range}(\hat{V}_{k-1})$. 
 \end{definition}
 
  The next lemma compares the outputs of the above two processes. 
\begin{lemma}  \label{lemma:diff}
 Suppose that \cref{assum:h} holds for the fixed-point problem \eqref{eq:fixed-point}.  For the sequences $\lbrace x_k \rbrace$ and $\lbrace \hat{x}_k \rbrace$ in \cref{definition:two_processes}, if $x_0$ is sufficiently close to $x^*$ and $\|h(x_j)\|_2\leq \eta_0\|h(x_0)\|_2$, $j=0,\dots,k$, where $\eta_0>0$ is a constant, then 
  \begin{align}
  & \|r_k - \hat{r}_k\|_2 = \hat{\kappa}\cdot\mathcal{O}(\|x_{k-m_k}-x^*\|_2^2),  
  \label{lemma:diff:r_k}    \\
 & \| x_{k+1} - \hat{x}_{k+1}\|_2 = \hat{\kappa}\cdot\mathcal{O}(\|x_{k-m_k}-x^*\|_2^2). 
 \label{lemma:diff:x_k}
 \end{align}
\end{lemma}

 The proof is given in \Cref{subsec:proof:lemma:diff} due to space limitations. 
 Since Process~II is closely related to  Krylov subspace methods from \Cref{prop:linear},  \cref{lemma:diff} extends this relationship to the nonlinear case. 
  When certain assumptions hold, $\|x_k-\hat{x}_k\|_2$ is bounded by a second-order term. Intuitively, we can obtain the convergence of $\lbrace x_k \rbrace$ for nonlinear problems from the convergence of $\lbrace \hat{x}_k \rbrace$ for the corresponding linear problems. If $\lbrace \hat{x}_k \rbrace$ converges linearly (not quadratically), it is expected that $\lbrace x_k \rbrace$ has a similar convergence rate to $\lbrace \hat{x}_k \rbrace$ provided that $x_0$ is sufficiently close to $x^*$. 

 \begin{theorem}   \label{them:am}
  Suppose that \cref{assum:h} holds for the fixed-point problem \eqref{eq:fixed-point}. 
  Let $\lbrace x_k \rbrace$ and $\lbrace r_k \rbrace$ denote the iterates and residuals of the restarted AM, $A:=I-g'(x^*)$, $\theta_k := \|I-\beta_kA\|_2 $, and $\eta_0>0$ is a constant. We assume  $\beta_j\in[\beta,\beta']$ $(j\geq 0)$ for some positive constants $\beta$ and  $\beta'$.
  The following results hold.
  
 1. For the Type-I AM, let $\pi_k$ be the orthogonal projector onto $\mathcal{K}_{m_k}(A,r_{k-m_k})$ and $A_k := \pi_kA\pi_k$. $A_k\vert_{\mathcal{K}_{m_k}(A,r_{k-m_k})}$ denotes the restriction of $A_k$ to $\mathcal{K}_{m_k}(A,r_{k-m_k})$. If $\|r_j\|_2\leq \eta_0\|r_0\|_2$ $(0\leq j\leq k)$ and $x_0$ is sufficiently close to $x^*$, then  
  \begin{equation}
   \|x_{k+1}-x^*\|_2 \leq \theta_k\sqrt{1+\gamma_k^2\kappa_k^2}\min_{\mathop{}_{p(0)=1}^{p\in \mathcal{P}_{m_k}}}\|p(A)(x_{k-m_k}-x^*)\|_2 + \hat{\kappa}\mathcal{O}(\|x_{k-m_k}-x^*\|_2^2),  \label{them:type1}
  \end{equation}
  where $\gamma_k = \|\pi_kA(I-\pi_k)\|_2 \leq L$, and $\kappa_k = \|( A_k\vert_{\mathcal{K}_{m_k}(A,r_{k-m_k})})^{-1}\|_2 \leq 1/\mu$. 
  
 2. For the Type-II AM, if $\|r_j\|_2\leq \eta_0\|r_0\|_2$ $(0\leq j\leq k+1)$ and $x_0$ is sufficiently close to $x^*$, then 
 \begin{align}
  \|r_{k+1}\|_2 \leq \theta_k \min_{\mathop{}_{p(0)=1}^{p\in \mathcal{P}_{m_k}}}\|p(A)r_{k-m_k}\|_2 + \hat{\kappa}\mathcal{O}(\|x_{k-m_k}-x^*\|_2^2).  \label{them:type2}
 \end{align}  
Alternatively, letting $\theta\in \bigl[\bigl(1-\frac{\mu^2}{L^2}\bigr)^{1/2},1\bigr)$ be a constant, if $\theta_j=\|I-\beta_jA\|_2 \leq \theta$ $(j\geq 0)$ and $x_0$ is sufficiently close to $x^*$, then \eqref{them:type2} holds. 

 
 3. For either method, if the aforementioned assumptions hold and $m_k = d$, then  $\|x_{k+1}-x^*\|_2 = \hat{\kappa}\mathcal{O}(\|x_{k-m_k}-x^*\|_2^2)$, namely, $(d+1)$-step quadratic convergence. 
 \end{theorem}
 
 The proof is given in \Cref{subsec:proof:them:am}, which is based on  \Cref{lemma:diff}, \Cref{prop:linear}, and the convergence properties of Krylov subspace methods. 
 Results \eqref{them:type1} and \eqref{them:type2} characterize the long-term convergence behaviours of both restarted AM methods for solving nonlinear equations $h(x) = 0$, where $h$ satisfies \Cref{assum:h}. 
 
 \begin{remark}
 The assumption that $x_0$ is sufficiently close to $x^*$ is common for the local analysis of an iterative method \cite{toth2015convergence,kelley2018numerical,brezinski2020shanks}. 
 Similar to \cite{dai2006cyclic},  
 since an explicit bound for $\|x_0-x^*\|_2$ is rather cumbersome and not very useful in practice, we omit it here for conciseness. 
 Besides, we do not assume $g$ to be contractive here.  
 The critical point is the positive definiteness of the Jacobian $h'(x^*)$,  without which there is no convergence guarantee even for solving linear systems \cite{walker2011anderson,potra2013characterization,Embree2003the}. 
 \end{remark}
 
 \begin{remark}
 If $m_k$ is large and $x_0$ is sufficiently close to $x^*$, 
  the convergence rates of both restarted AM methods are dominated by the minimization problems in \eqref{them:type1} and \eqref{them:type2},  
  which have been extensively studied in the context of Krylov subspace methods   \cite{saad1981krylov,eisenstat1983variational,Greenbaum1994,saad2003iterative}.
 For $j\geq 0$, define $u_j=x_j-x^*$ for the Type-I method, and $u_j=r_j$ for the Type-II method. 
 Note that 
 $ \left\|I-\frac{\mu}{L^2}A\right\|_2 \leq \theta:=\bigl(1-\frac{\mu^2}{L^2}\bigr)^{1/2}$ (see \Cref{lemma:theta_k} in \Cref{subsec:proof:them:am}).  
Choosing  $ p(A) = \left( I-\frac{\mu}{L^2}A\right)^{m_k}$,
 it follows that 
 \begin{equation} 
  \min_{\mathop{}_{p(0)=1}^{p\in\mathcal{P}_{m_k}}}\|p(A)u_{k-m_k}\|_2 \leq \min_{\mathop{}_{p(0)=1}^{p\in\mathcal{P}_{m_k}}}\|p(A)\|_2\|u_{k-m_k}\|_2 
  \leq \theta^{m_k}\left\|u_{k-m_k}\right\|_2.  \label{ineq:minimal}
 \end{equation} 
 With more properties about $A$,  we may choose other polynomials to sharpen the upper bound in \eqref{ineq:minimal}. 
 We give a refined result in \Cref{remark:convergence_sym} when $A$ is symmetric. 
 \end{remark}

 Now, we consider the case that the fixed-point map $g$ is a contraction. Specifically, we make the following assumptions on $g$, which are similar to those in \cite{toth2015convergence,evans2020proof}. 
 \begin{assumption} \label{assum:g}
  The fixed-point map $g:\mathbb{R}^d\rightarrow\mathbb{R}^d$ has a fixed point $x^*$. In the local region $\mathcal{B}_{\hat{\rho}}(x^*):= \lbrace z\in\mathbb{R}^d\vert \|z-x^*\|_2\leq \hat{\rho}\rbrace $ for some constant $\hat{\rho}>0$, $g$ is Lipschitz continuously differentiable, and there are constants $\kappa\in (0,1)$ and $\hat{\kappa}>0$ such that
  \begin{itemize}
   \item $\|g(y)-g(x)\|_2 \leq \kappa \|y-x\|_2$ for every $x,y \in \mathcal{B}_{\hat{\rho}}(x^*)$;
   \item $\|g'(y)-g'(x)\|_2 \leq \hat{\kappa}\|y-x\|_2$ for every $x,y \in \mathcal{B}_{\hat{\rho}}(x^*)$. 
  \end{itemize}
 \end{assumption}
 
In fact, we show \cref{assum:g} is a sufficient condition for \cref{assum:h}.
 \begin{lemma}  \label{lemma:h}
  Suppose  \cref{assum:g} holds for the fixed-point problem \eqref{eq:fixed-point}. Let $h(x):=x-g(x)$. Then $h$ satisfies \cref{assum:h}. In $\mathcal{B}_{\hat{\rho}}(x^*)$, the Lipschitz constant of $h'$ is $\hat{\kappa}$; for \eqref{assum:jacobian} and \eqref{assum:hermite_part}, the constants are $\mu=1-\kappa, L=1+\kappa$. 
 \end{lemma}
 
 The proof is given in \Cref{subsec:proof:lemma:h}.  Based on \Cref{lemma:h} and \Cref{them:am}, we obtain the following corollary for the Type-II AM. 
 
 
 \begin{corollary}   \label{coro:am_type2}
  Suppose that \cref{assum:g} holds for the fixed-point problem \eqref{eq:fixed-point}. 
  Let $\lbrace x_k \rbrace$ and $\lbrace r_k \rbrace$ denote the iterates and residuals of the restarted Type-II AM with $\beta_k=1$ $(k\geq 0)$. If $x_0$ is sufficiently close to $x^*$, then   
  \begin{align}
  \|r_{k+1}\|_2 \leq \kappa \min_{\mathop{}_{p(0)=1}^{p\in \mathcal{P}_{m_k}}}\|p(A)r_{k-m_k}\|_2 + \hat{\kappa}\mathcal{O}(\|x_{k-m_k}-x^*\|_2^2),   \label{coro:type2}
 \end{align}
 where $A:=I-g'(x^*)$. If $m_k=d$, then $\|x_{k+1}-x^*\|_2 = \hat{\kappa}\mathcal{O}(\|x_{k-m_k}-x^*\|_2^2)$. 
 \end{corollary} 
 
 \begin{remark}   \label{remark:convergence}
 The $R$-linear convergence of the limited-memory Type-II AM has been established in  \cite{toth2015convergence}: Under \Cref{assum:g} and assuming that $\sum_{j=0}^{m_k}|\alpha_{k}^{(j)}|$ is bounded, it is proved that for  $\tilde{\kappa}\in(\kappa,1)$, if $x_0$ is sufficiently close to $x^*$, then
 \begin{equation}
  \|r_{k}\|_2 \leq \tilde{\kappa}^{k}\|r_0\|_2.   \label{convergence:toth}
 \end{equation}
 However,
 as noted by Anderson \cite{anderson2019comments}, \eqref{convergence:toth} does not show the advantage of AM over the fixed-point iteration \eqref{eq:picard} since the latter converges $Q$-linearly with $Q$-factor $\kappa$. 
 In \cite{evans2020proof}, an improved bound is obtained: 
 \begin{equation}
  \|r_{k+1}\|_2 \leq s_k(1-\beta_k+\kappa\beta_k)\|r_k\|_2
   + \sum_{j=0}^{m}\mathcal{O}(\|r_{k-j}\|_2^2),   \label{convergence:evans}
\end{equation}  
 where $k\geq m$ and $s_k:={\|\bar{r}_k\|_2}/{\|r_k\|_2}$. If $\beta_k = 1$, \eqref{convergence:evans} improves  \eqref{convergence:toth} since $s_k\leq 1$. 
 However, the quality of extrapolation, namely $s_k$, is difficult to estimate in advance.  The recent analysis in \cite{pollock2021anderson} refines the higher-order terms in \eqref{convergence:evans}, but leaves the issue about $s_k$ unaddressed. 
 For the restarted Type-II AM, 
  \Cref{coro:am_type2} shows that its convergence rate is dominated by the first term on the right-hand side of \eqref{coro:type2}. Using $p(A) = (I-A)^{m_k}$,  \eqref{coro:type2} leads to $\|r_{k+1}\|_2 \leq \kappa^{m_k+1}\|r_{k-m_k}\|_2+\hat{\kappa}\mathcal{O}(\|x_{k-m_k}-x^*\|_2^2)$ that is comparable to the fixed-point iteration \eqref{eq:picard}. Nonetheless, due to the optimality, 
  the polynomial that minimizes $\|p(A)r_{k-m_k}\|_2$ corresponds to the $m_k$-step GMRES iterations and can often provide a much better bound than $(I-A)^{m_k}$ \cite{saad1986gmres,Greenbaum1994}, which justifies the acceleration by Type-II AM in practice. 
  Therefore,  our multi-step analysis provides a better assessment of the efficacy of Type-II AM than previous works. 
 \end{remark}
\begin{remark}
  Though the numerical experiments in \cite{de2022Linear} suggest that the limited-memory AM can converge faster than the restarted AM with the same $m$,   
  the theoretical properties of the limited-memory AM are much more vague, even in the linear case \cite{desterck2023anderson}.  We leave the analysis for the limited-memory AM as our future work.
\end{remark}

\section{Adaptive mixing strategy} \label{sec:mix_params} 
  
  As shown in \Cref{them:am}, the choice of $\beta_k$ directly affects the factor $\theta_k$ in \eqref{them:type1} and \eqref{them:type2}.  If $g$ is not contractive,  a proper $\beta_k$ is required to ensure the numerical performance of AM \cite{fang2009two,anderson2019comments}.  However, tuning $\beta_k$ with a grid search can be costly in practice. In this section, we explore the properties of restarted AM to develop an efficient procedure to estimate the eigenvalues of $h'(x^*)$, based on which we can choose $\beta_k$ adaptively. 

 We start from the linear case to better explain how to estimate the eigenvalues. Let $g(x)=(I-A)x+b$ in the fixed-point problem \eqref{eq:fixed-point}, where $A\in\mathbb{R}^{d\times d}$ is positive definite and $b\in\mathbb{R}^d$. Then $h(x)=Ax-b$. Using the historical information in the  restarted AM, we apply a projection method \cite{saad2011} to estimate the spectrum of $A$:
 \begin{equation}
 \begin{aligned}
   u \in {\rm range}(Q_k), \quad (A-\lambda I)u \perp {\rm range}(V_k),  \label{eq:eigs_projection}
 \end{aligned}
 \end{equation}
 where $u\in\mathbb{R}^d$ is an approximate eigenvector of $A$ sought in ${\rm range}(Q_k)$, and $\lambda \in \mathbb{R}$ is an eigenvalue estimate.  The orthogonality condition in \eqref{eq:eigs_projection} is known as the Petrov-Galerkin condition. 
 Let $u = Q_ky, ~ y\in \mathbb{R}^{m_k}$. Then \eqref{eq:eigs_projection} leads to 
 \begin{equation}
  V_k^\mathrm{T}AQ_k y = \lambda V_k^\mathrm{T}Q_k y.   \label{eq:eigs_reduced}
 \end{equation}
 Next, we describe how to solve the generalized eigenvalue problem \eqref{eq:eigs_reduced} using the properties of restarted AM. 
 
 At the $(k+1)$-th iteration, suppose that $m_{k+1}\geq 2$. 
   As will be shown in \cref{prop:linear:hessenberg}, there is an upper Hessenberg matrix $H_k \in \mathbb{R}^{m_k\times m_k}$ such that
 \begin{equation}
  AQ_k = Q_k H_k + q_{k+1}\cdot h_{k}^{(m_k+1)}e_{m_k}^\mathrm{T},      \label{eq:AQ_k} 
 \end{equation}
 where  $ h_{k}^{(m_k+1)}\in \mathbb{R}$ and $e_{m_k}$ is the $m_k$-th column of $I_{m_k}$.  
 Since $q_{k+1}^\mathrm{T}V_k = 0$ from the construction of $q_{k+1}$ (cf. \Cref{prop:modified_seqs}), 
  it follows from \eqref{eq:AQ_k} that $V_k^\mathrm{T}AQ_k = V_k^\mathrm{T}Q_kH_k$, which together with \eqref{eq:eigs_reduced} yields that 
 $V_k^\mathrm{T}Q_kH_ky = \lambda V_k^\mathrm{T}Q_ky$. Noting that  $\det(V_k^\mathrm{T}Q_k) \neq 0$ if the restarted AM does not reach the exact solution, we find \eqref{eq:eigs_reduced} is reduced to
 \begin{equation}
  H_k y = \lambda y,    \label{eq:eigs_H_k}
 \end{equation}
which can be solved by efficient numerical algorithms \cite{golub2013matrix} using  $\mathcal{O}(m_k^3)$ flops. 

\begin{remark}
  From  \cref{prop:linear}, ${\rm range}(Q_k) = A\mathcal{K}_{m_k}(A,r_{k-m_k})$. 
  For the Type-I method,  \eqref{eq:eigs_projection} is an oblique projection method; for the Type-II method, \eqref{eq:eigs_projection} can be viewed as the Arnoldi's method \cite{saad1980variations} based on $A^\mathrm{T}A$-norm. It is expected that with larger $m_k$, the eigenvalue estimates are closer to the exact eigenvalues of $A$. 
\end{remark} 
 
 Now, we describe the construction of $H_k$ in \Cref{definition:H_k} and show the role of $H_k$ in \Cref{prop:linear:hessenberg}.
 \begin{definition}  \label{definition:H_k}
  Consider applying the restarted AM to solve the fixed-point problem \eqref{eq:fixed-point}. At the $(k+1)$-th iteration, suppose that $m_{k+1}\geq 2$. 
  Define the unreduced upper Hessenberg matrix $\bar{H}_k = (H_k^\mathrm{T},h_k^{(m_k+1)}e_{m_k})^\mathrm{T} \in \mathbb{R}^{(m_k+1)\times m_k}$, where  $h_k^{(m_k+1)}\in \mathbb{R}$, and $e_{m_k}$ is the $m_k$-th column of $I_{m_k}$.  The $H_k\in\mathbb{R}^{m_k\times m_k}$ is defined as 
 $H_k = h_k$ if $m_k=1$ and $H_k = \left(\bar{H}_{k-1},h_k\right)$ if $m_k\geq 2$. 
  Define $\phi_k = \Gamma_k+\zeta_{k+1}$, $\Gamma_k^{[m_k-1]} = (\Gamma_k^{(1)},\dots,\Gamma_k^{(m_k-1)})^\mathrm{T}$. The $h_k\in\mathbb{R}^{m_k}$ in $H_k$ is constructed as follows:
 \begin{equation}
 \begin{aligned}
  & h_k = \frac{1}{1-\Gamma_k}\left( \frac{1}{\beta_{k-1}} - \frac{1}{\beta_k}\phi_k\right),  ~ \mbox{if} ~ m_k = 1;  \\
  & h_k = \frac{1}{1-\Gamma_k^{(m_k)}}\left( \frac{1}{\beta_{k-1}}
  \begin{pmatrix}
   \phi_{k-1} \\
   1
  \end{pmatrix} -\frac{1}{\beta_k}\phi_k - \bar{H}_{k-1}\left( \phi_{k-1}-\Gamma_k^{[m_k-1]} \right) \right),  ~ \mbox{if} ~ m_k\geq 2.   \label{eq:H_k}
 \end{aligned}
 \end{equation}
 The construction of $h_k^{(m_k+1)}$ is
 \begin{align}
  h_k^{(m_k+1)} = -\frac{1}{\beta_k(1-\Gamma_k^{(m_k)})}, ~ \mbox{for} ~ m_k\geq 1.   \label{eq:h_k_2}
 \end{align}   
 \end{definition}
 
 \begin{proposition}   \label{prop:linear:hessenberg}
  Let $g(x)=(I-A)x+b$ in the fixed-point problem \eqref{eq:fixed-point}, where $A\in\mathbb{R}^{d\times d}$ is positive definite and $b\in\mathbb{R}^d$. For the restarted Type-I/Type-II AM method, if $m_{k+1}\geq 2$ at the $(k+1)$-th iteration, then  
  with the notations defined in \Cref{definition:H_k}, we have  
 \begin{align}
   AP_k = P_{k+1}\bar{H}_k = P_kH_k+p_{k+1}\cdot h_k^{(m_k+1)}e_{m_k}^\mathrm{T}.   \label{eq:AP_k}  
 \end{align}
 \end{proposition}
 
 The proof is given in \Cref{proof:prop:linear:hessenberg}. Since $Q_{k+1} = -AP_{k+1}$ from \Cref{prop:linear}, the relation \eqref{eq:AQ_k} holds as a result of \eqref{eq:AP_k}. 
\begin{remark}
 \Cref{definition:H_k} suggests that $H_k$ can be economically constructed by manipulating the coefficients in the restarted AM. Thus we can efficiently solve the problem  \eqref{eq:eigs_projection} without any additional matrix-vector product. 
\end{remark}

 Next, consider the nonlinear case.  We can still construct $H_k$ by \Cref{definition:H_k}. 
 Let $A:=h'(x^*)$.  Since $g$ is nonlinear, the relation \eqref{eq:AQ_k} does not exactly hold in general, which can make the eigenvalues of $H_k$ different from those computed by solving \eqref{eq:eigs_reduced}. Nonetheless, similar to the proof of  \Cref{lemma:diff}, we consider an auxiliary process using restarted AM to solve the linearized problem $\hat{h}(x) = 0$, where a Hessenberg matrix $\hat{H}_k$ can be constructed  as well.  By comparing $H_k$ and $\hat{H}_k$, we show that the eigenvalues of $H_k$ can still approximate the eigenvalues of $A$. 
 \begin{lemma}   \label{lemma:diff:Hessenberg}
  Suppose that \cref{assum:h} holds for the fixed-point problem \eqref{eq:fixed-point}. 
  For the Process~I in \cref{definition:two_processes}, assume that there are positive constants $\eta_0,\tau_0$ such that $\|h(x_j)\|_2\leq \eta_0\|h(x_0)\|_2$ $(0\leq j\leq k+1)$ and $\vert 1-\Gamma_j^{(m_j)} \vert \geq \tau_0$ $(1\leq j\leq k)$; 
  $H_k$ is  defined by \Cref{definition:H_k}. For the Process~II in \cref{definition:two_processes}, the upper Hessenberg matrix $\hat{H}_k\in\mathbb{R}^{m_k\times m_k}$ is defined correspondingly (by replacing $\Gamma_k, \zeta_{k+1}$ with $\hat{\Gamma}_k, \hat{\zeta}_{k+1}$ in \Cref{definition:H_k}). 
  Then for $x_0$ sufficiently close to $x^*$, we have
  \begin{equation}
   \|H_k\|_2 = \mathcal{O}(1), ~~ \|H_k-\hat{H}_k\|_2 = \hat{\kappa}\mathcal{O}(\|x_{k-m_k}-x^*\|_2).  \label{lemma:H_k}
  \end{equation}
 \end{lemma}
 
 The proof can be found in \Cref{subsec:proof:lemma:diff:Hessenberg}. 
 It suggests that $H_k$ is a perturbation of $\hat{H}_k$. Since $\hat{h}(x)$ is linear, the eigenvalues of $\hat{H}_k$ exactly solve $\hat{V}_k^\mathrm{T}A\hat{Q}_k y = \hat{\lambda} \hat{V}_k^\mathrm{T}\hat{Q}_k y$, thus approximating the eigenvalues of $A$. Next, 
  we compare $\sigma(H_k)$ and $\sigma(\hat{H}_k)$ using the perturbation theory. 
 
 \begin{theorem} \label{them:eigen_estimate}
  Under the same assumptions of \cref{lemma:diff:Hessenberg}, let $\lambda$ be an eigenvalue of $H_k$. Then for $x_0$ sufficiently close to $x^*$, we have
  \begin{equation}
   \min_{\hat{\lambda}\in\sigma(\hat{H}_k)}\vert \hat{\lambda}-\lambda \vert 
    = \hat{\kappa}^{1/m_k}\mathcal{O}(\|x_{k-m_k}-x^*\|_2^{1/m_k}).   \label{them:diff:lambda}
  \end{equation}
 If further assuming $\hat{H}_k$ is diagonalizable, i.e., there is a  nonsingular matrix $M_k\in\mathbb{R}^{m_k\times m_k}$ such that $\hat{H}_k = M_k\hat{D}_kM_k^{-1}$, where $\hat{D}_k$ is diagonal, then
 \begin{equation}
  \min_{\hat{\lambda}\in\sigma(\hat{H}_k)}\vert \hat{\lambda}-\lambda \vert
    = \|M_k\|_2\|M_k^{-1}\|_2\hat{\kappa}\mathcal{O}(\|x_{k-m_k}-x^*\|_2).  \label{them:diff:diagonal_lambda}
 \end{equation}
 \end{theorem}
 \begin{proof}
  \eqref{them:diff:lambda} follows from \Cref{lemma:diff:Hessenberg} and \cite[Theorem~VIII.1.1]{bhatia2013matrix}, and \eqref{them:diff:diagonal_lambda} is a consequence of \Cref{lemma:diff:Hessenberg} and Bauer-Fike theorem \cite[Theorem~7.2.2]{golub2013matrix}. 
 \end{proof}
  
  Since $\hat{H}_k$ is unavailable in practice, \Cref{them:eigen_estimate} suggests that we can use the eigenvalues of $H_k$ to roughly estimate the eigenvalues of $A$. 
  Then, suppose $m_k\geq 2$ and $\tilde\lambda$ is the eigenvalue of $H_{k-1}$ of the largest absolute value. We set the mixing parameter at the $k$-th iteration as
 \begin{align}
  \beta_k = \frac{2}{|\tilde{\lambda}|}.   \label{eq:beta_k}
 \end{align}
  We call such a way to choose $\beta_k$ as the adaptive mixing strategy since $\beta_k$ is chosen adaptively. 
  Usually, the extreme eigenvalues can be quickly estimated, so we only need to run the eigenvalue estimation procedure for a few steps. 

\section{Short-term recurrence methods} \label{sec:short-term}

  For solving high-dimensional problems, the memory cost of AM can be prohibitive when $m_k$ is large.  In this section, we show that if the Jacobian matrix is symmetric, the restarted AM methods can have short-term recurrence forms which  address the memory issue while maintaining fast convergence. 
  Since \Cref{assum:h} assumes that $h'(x^*)$ is positive definite, the symmetry of $h'(x^*)$ motivates us to consider solving SPD linear systems first.
 
 \begin{proposition}   \label{prop:spd_linear}
  Let $g(x)=(I-A)x+b$ in the fixed-point problem \eqref{eq:fixed-point}, where $A\in\mathbb{R}^{d\times d}$ is SPD and $b\in\mathbb{R}^d$. For the full-memory Type-I/Type-II AM with modified historical sequences, 
  we have 
  \begin{align}
   \zeta_k^{(j)} = 0, ~ \mbox{for} ~ j\leq k-3 ~ (k\geq 4),   ~ \mbox{and} ~~  
   \Gamma_k^{(j)} = 0, ~ \mbox{for} ~ j \leq k-2 ~ (k\geq 3),    
  \end{align}
  if the algorithm has not found the exact solution. 
 \end{proposition}
 \begin{proof}
 Since $A$ is SPD, by \Cref{prop:linear}, the procedures \eqref{eq:q_k_type2} and \eqref{eq:r_k_type2} are well defined during the iterations. 
  \cref{prop:modified_seqs} also suggests that 
  \begin{align}
   q_k = \Delta r_{k-1}-Q_{k-1}\zeta_k \perp {\rm range}(V_{k-1}), ~~
   \bar{r}_k = r_k - Q_k\Gamma_k \perp {\rm range}(V_k).   \label{eq:perp}
  \end{align}
  Hence, 
  \begin{align}
   \zeta_k = (V_{k-1}^\mathrm{T}Q_{k-1})^{-1}V_{k-1}^\mathrm{T}\Delta r_{k-1}, ~~
   \Gamma_k = (V_k^\mathrm{T}Q_k)^{-1}V_k^\mathrm{T}r_k. 
  \end{align}
  Since $A$ is symmetric, it follows that $V_k^\mathrm{T}Q_k$ is diagonal for either type of the AM methods. Note that $r_k = \bar{r}_{k-1}-\beta_{k-1}A\bar{r}_{k-1}$ and ${\rm range}(AV_{k-2}) \subseteq {\rm range}(V_{k-1})$ due to \eqref{prop:linear:modified_seqs} in  \cref{prop:linear}. Hence
  \begin{align}
   V_{k-2}^\mathrm{T}r_k = V_{k-2}^\mathrm{T}\bar{r}_{k-1}
      -\beta_{k-1}(AV_{k-2})^\mathrm{T}\bar{r}_{k-1} = 0,   \label{eq:r_k_perp}
  \end{align}
  as a consequence of \eqref{eq:perp}. So the first $(k-2)$ elements of $V_k^\mathrm{T}r_k$ are zeros. Thus, $\Gamma_k^{(j)} = 0$, $j\leq k-2$, for $k\geq 3$. Also, \eqref{eq:r_k_perp} yields that
  \begin{align}
   V_{k-3}^\mathrm{T}\Delta r_{k-1} = V_{k-3}^\mathrm{T}r_k - V_{k-3}^\mathrm{T}r_{k-1} = 0,
  \end{align}
  which infers that the first $(k-3)$ elements of $V_{k-1}^\mathrm{T}\Delta r_{k-1}$ are zeros. Thus, $\zeta_k^{(j)} = 0$, $j\leq k-3$, for $k\geq 4$. 
 \end{proof}
 
 \cref{prop:spd_linear} suggests that for solving SPD linear systems, 
 we only need to maintain the most recent two vector pairs, and there is no loss of historical information. 
 Specifically, suppose $k\geq 3$ and  define $\{v_j\}$ as that in \Cref{subsec:modified_seqs}. 
 The procedure has short-term recurrences and is described as follows. 
 
 {\em Step~1: Modified vector pair.} At the beginning of the $k$-th iteration, $p_k,q_k$ are obtained from $\Delta x_{k-1}, \Delta r_{k-1}$ and   $(p_{k-2},p_{k-1}), (q_{k-2},q_{k-1})$:
 \begin{align}
  p_k = \Delta x_{k-1}-(p_{k-2},p_{k-1})\zeta_k, ~~ q_k = \Delta r_{k-1}-(q_{k-2},q_{k-1})\zeta_k,
 \end{align}
 where $\zeta_k\in\mathbb{R}^2$ is chosen such that $q_k \perp {\rm span}\{v_{k-2},v_{k-1}\}$. 
  
 {\em Step~2: AM update.} The next step is the ordinary AM update: 
 \begin{align}
  \bar{x}_k = x_k-(p_{k-1},p_k)\Gamma_k, ~~ \bar{r}_k = r_k-(q_{k-1},q_k)\Gamma_k, ~~ x_{k+1} = \bar{x}_k+\beta_k\bar{r}_k, 
 \end{align}
 where $\Gamma_k\in\mathbb{R}^2$ is chosen such that $\bar{r}_k \perp {\rm span}\{v_{k-1},v_{k}\}$. 

 We call it short-term recurrence AM (ST-AM). The Type-II ST-AM has been proposed in \cite{wei2022a}. Combining the ST-AM update with the restarting conditions \eqref{ineq:cond1}-\eqref{ineq:cond3}, we obtain the restarted ST-AM methods, as shown in \Cref{alg:short-term_restartedAM}. 

 
 \begin{algorithm}[ht]
\caption{Restarted ST-AM. $g$ is the fixed-point map with symmetric Jacobian $g'$. The Type-I method: $v_j:=p_j$ $(j\geq 1)$; the Type-II method: $v_j:=q_j$ $(j\geq 1)$. }
\label{alg:short-term_restartedAM}
\textbf{Input}: $ x_0\in\mathbb{R}^d, \beta_k>0, m\in\mathbb{Z}_+, \tau\in(0,1), \eta >0 $\\
\textbf{Output}: $ x\in\mathbb{R}^d $
\begin{algorithmic}[1] 
\STATE $ m_0=0$
\FOR{$k=0,1,\dots, $ until convergence}
 \STATE $ r_k=g(x_k)-x_k $
 \IF { $m_k>m$ \textbf{or} $\|r_k\|_2 > \eta \|r_{k-m_k}\|_2$}
 \STATE $m_k = 0$
 \ENDIF
 \IF{$m_k>0$} 
  \STATE $p_k=x_k-x_{k-1}, ~~ q_k = r_k-r_{k-1}$  
  \FOR{$j=\max\{1,m_k-2\},\dots,m_k-1$}
   \STATE $\zeta = \left(v_{k-m_k+j}^\mathrm{T}q_k\right)/\left(v_{k-m_k+j}^\mathrm{T}q_{k-m_k+j}\right)$
   \STATE $p_k=p_k-p_{k-m_k+j}\zeta, ~~ q_k = q_k-q_{k-m_k+j}\zeta$
  \ENDFOR  
  \IF{ $\vert v_k^\mathrm{T}q_k\vert < \tau \vert v_{k-m_k+1}^\mathrm{T}q_{k-m_k+1}\vert$ }
  \STATE $m_k = 0$
  \ENDIF
 \ENDIF
 \STATE $\bar{x}_k = x_k, ~~ \bar{r}_k = r_k$
 \FOR{$j=\max\{1,m_k-1\},\dots,m_k$}
  \STATE $\gamma = \left(v_{k-m_k+j}^\mathrm{T}\bar{r}_k\right)/\left(v_{k-m_k+j}^\mathrm{T}q_{k-m_k+j}\right)$
  \STATE $\bar{x}_k = \bar{x}_k-p_{k-m_k+j}\gamma, ~~ \bar{r}_k = \bar{r}_k-q_{k-m_k+j}\gamma$
 \ENDFOR
 \STATE $x_{k+1} = \bar{x}_k+\beta_k\bar{r}_k$ 
  \STATE $m_{k+1} = m_k+1$ 
\ENDFOR
\STATE \textbf{return} $ x_k $
\end{algorithmic}
\end{algorithm}     
 
 We establish the convergence properties in the nonlinear case.
 \begin{theorem}   \label{them:short_term_convergence}
 For the fixed-point problem \eqref{eq:fixed-point}, suppose that $g'(x)$ is symmetric and \cref{assum:h} holds. 
  Let $\lbrace x_k \rbrace$ and $\lbrace r_k \rbrace$ denote the iterates and residuals of the restarted ST-AM, $A:=I-g'(x^*)$, $\theta_k := \|I-\beta_kA\|_2 $,  and $\theta\in \bigl[\frac{L-\mu}{L+\mu},1\bigr)$ is a constant. 
 For $ k = 0,1,\dots$,  $\beta_k$ is chosen such that $\theta_k \leq \theta$. 
  The following results hold. 
 
  
  1. For the Type-I method, if $x_0$ is sufficiently close to $x^*$, then
  \begin{equation}
   \|x_{k+1}-x^*\|_A \leq 2\theta_k\left( \frac{\sqrt{L/\mu}-1}{\sqrt{L/\mu}+1} \right)^{m_k} \|x_{k-m_k}-x^*\|_A+\hat{\kappa}\mathcal{O}(\|x_{k-m_k}-x^*\|_2^2).  \label{them:short_term_type1} 
  \end{equation}
  
  2. For the Type-II method, if $x_0$ is sufficiently close to $x^*$, then
  \begin{equation}
   \|r_{k+1}\|_2 \leq 2\theta_k \left( \frac{\sqrt{L/\mu}-1}{\sqrt{L/\mu}+1} \right)^{m_k}\|r_{k-m_k}\|_2+\hat{\kappa}\mathcal{O}(\|x_{k-m_k}-x^*\|_2^2).   \label{them:short_term_type2}
  \end{equation}
  
  3. For either method, if the aforementioned assumptions hold and $m_k=d$, then  $\|x_{k+1}-x^*\|_2=\hat{\kappa}\mathcal{O}(\|x_{k-m_k}-x^*\|_2^2)$, namely $(d+1)$-step quadratic convergence. 
 \end{theorem}
 
 We give the proof in \Cref{proof:them:short_term_convergence}. 
 \cref{them:short_term_convergence} shows that the asymptotic convergence rates of both types of restarted ST-AM methods are optimal with respect to the condition number (see \cite[Section~2.1.4]{Nesterov2018}), thus significantly improving the convergence rate of the fixed-point iteration.  The theorem also suggests that the restarted ST-AM methods are applicable for solving large-scale  unconstrained optimization problems since the Hessian matrices are naturally symmetric. 
 
 \begin{remark}   \label{remark:convergence_sym}
  When $A$ is symmetric, the convergence bound \eqref{them:type1} for the restarted Type-I AM can be refined to 
  \begin{equation}
   \|x_{k+1}-x^*\|_A \leq \theta_k\min_{\mathop{}_{p(0)=1}^{p\in \mathcal{P}_{m_k}}}\|p(A)(x_{k-m_k}-x^*)\|_A + \hat{\kappa}\mathcal{O}(\|x_{k-m_k}-x^*\|_2^2).   \label{them:type1_sym}
  \end{equation}
  Using Chebyshev polynomials for the minimization problems in \eqref{them:type1_sym} and \eqref{them:type2} \cite[Theorem~6.29]{saad2003iterative},  we can establish the same convergence rates as \eqref{them:short_term_type1} and \eqref{them:short_term_type2}  for the restarted Type-I AM and the restarted Type-II AM, respectively. On the other side, the convergence results in \cite{toth2015convergence,evans2020proof} shown by \eqref{convergence:toth} and  \eqref{convergence:evans} in \Cref{remark:convergence} cannot provide such refined results and underestimate the efficacy of AM. 
 \end{remark}
 
 In practice, we can also choose the mixing parameters $\lbrace \beta_k\rbrace$ with simplified computation by exploring the symmetry of the Jacobian.
 \begin{proposition}  \label{prop:linear:tridiagonal}
 Let $g(x)=(I-A)x+b$ in the fixed-point problem \eqref{eq:fixed-point}, where $A\in\mathbb{R}^{d\times d}$ is SPD and $b\in\mathbb{R}^d$. For the restarted Type-I/Type-II ST-AM method, if $m_{k+1}\geq 2$ at the $(k+1)$-th iteration, then 
 \begin{align}
  Ap_k = t_k^{(m_k-1)}p_{k-1}+t_k^{(m_k)}p_k+t_k^{(m_k+1)}p_{k+1},
 \end{align}
 where $p_{k-m_k}:=\mathbf{0}\in\mathbb{R}^d$, and the coefficients are given by
 \begin{equation}
 \begin{aligned}
  & t_k^{(m_k-1)} = \frac{\phi_{k-1}}{\beta_{k-1}(1-\Gamma_k^{(m_k)})},  \\
  & t_k^{(m_k)} = \frac{1}{1-\Gamma_k^{(m_k)}}\left( \frac{1}{\beta_{k-1}}-\frac{\phi_k}{\beta_k} \right),   \\
  & t_k^{(m_k+1)} = -\frac{1}{\beta_k(1-\Gamma_k^{(m_k)})},  
 \end{aligned}   \label{eq:t_k}
 \end{equation} 
 where $\phi_k:=0$ if $m_k = 0$, and $\phi_k := \Gamma_k^{(m_k)}+\zeta_{k+1}^{(m_k)} =  \Gamma_{k+1}^{(m_k)}= \frac{v_k^\mathrm{T}r_{k+1}}{v_k^\mathrm{T}q_k}$ if $m_k\geq 1$. Thus there exists a tridiagonal matrix $\bar{T}_k\in\mathbb{R}^{(m_k+1)\times m_k}$ such that 
 \begin{align}
  AP_k = P_{k+1}\bar{T}_k = P_kT_k+p_{k+1}\cdot t_k^{(m_k+1)}e_{m_k}^\mathrm{T},   \label{eq:T_k}
 \end{align}
 where $T_k$ is obtained from $\bar{T}_k$ by deleting its last row, and $e_{m_k}$ is the $m_k$-th column of $I_{m_k}$. 
 \end{proposition}
 \begin{proof}
  Note that 
  \begin{align*}
   \Gamma_k+\zeta_{k+1} = (V_k^\mathrm{T}Q_k)^{-1}V_k^\mathrm{T}r_k
   						 + (V_k^\mathrm{T}Q_k)^{-1}V_k^\mathrm{T}\Delta r_k 
   					    = (V_k^\mathrm{T}Q_k)^{-1}V_k^\mathrm{T}r_{k+1}.
  \end{align*}
  Since $V_k^\mathrm{T}Q_k$ is diagonal due to symmetry, and $V_{k-1}^\mathrm{T}r_{k+1} = 0$, it follows that 
  \begin{equation*}
   \Gamma_k+\zeta_{k+1} = (0,\dots,0,\frac{v_k^\mathrm{T}r_{k+1}}{v_k^\mathrm{T}q_k})^\mathrm{T} = \Gamma_{k+1}^{[m_k]},   
  \end{equation*}
 where $\Gamma_{k+1}^{[m_k]}$ is the subvector of the first $m_k$ elements of $\Gamma_{k+1}$. Then the formula \eqref{eq:t_k} follows from \Cref{definition:H_k} and \cref{prop:linear:hessenberg}. 
 \end{proof}
 
 Then, following the derivation in \cref{sec:mix_params}, the projection method \eqref{eq:eigs_projection} to estimate the eigenvalues of $h'(x^*)$ is reduced to solving the eigenvalues of $T_k$. 

 \begin{theorem}  \label{them:short-term:eigen_estimate}
 For the fixed-point problem \eqref{eq:fixed-point}, suppose that $g'(x)$ is symmetric and \cref{assum:h} holds. For the Process~I in \Cref{definition:two_processes}, replace the restarted AM by the restarted ST-AM, and assume that there are positive constants $\eta_0, \tau_0$ such that $\|h(x_j)\|_2 \leq \eta_0\|h(x_0)\|_2$ $(0\leq j\leq k+1)$, $|1-\Gamma_j^{(m_j)}|\geq \tau_0$ $(1\leq j\leq k)$; $T_k$ is the tridiagonal matrix as defined in \Cref{prop:linear:tridiagonal}, and $\lambda \in \sigma(T_k)$. 
 For the Process~II in \Cref{definition:two_processes}, the tridiagonal matrix $\hat{T}_k$ is defined correspondingly. 
  For $x_0$ sufficiently close to $x^*$, we have 
  \begin{align}
   \min_{\hat{\lambda}\in \sigma(\hat{T}_k)}\vert \hat{\lambda}-\lambda \vert = \hat{\kappa}\mathcal{O}(\|x_{k-m_k}-x^*\|_2).    \label{them:diff:short_term_lambda}
  \end{align}
 \end{theorem}
 
 The proof is given in \Cref{proof:them:short-term:eigen_estimate}. Since $\hat{h}(x)$ is linear and $h'(x^*)$ is symmetric, 
  using the eigenvalues of $\hat{T}_k$ to approximate the eigenvalues of $h'(x^*)$ is equivalent to a Lanczos method \cite{golub2013matrix}. For the Type-I method, the Lanczos method is $A$-norm based; for the Type-II method, the Lanczos method is $A^2$-norm based. Thus the eigenvalue $\hat{\lambda}\in\sigma(\hat{T}_k)$ is known as the generalized Ritz value \cite{manteuffel1994roots}. \Cref{them:short-term:eigen_estimate} indicates that  $\sigma(T_k)$ is close to $\sigma(\hat{T}_k)$ when $\|x_{k-m_k}-x^*\|_2$ is small.  
 
 At the $k$-th iteration, where $m_k\geq 2$, let $\tilde{\mu}$ be the eigenvalue of $T_{k-1}$ of the smallest absolute value, and let $\tilde{L}$ be the eigenvalue of $T_{k-1}$ of the largest absolute value. We use $\vert\tilde{\mu}\vert$ and $\vert\tilde{L}\vert$ 
  as the estimates of $\mu$ and $L$. Then we set 
 \begin{align}
  \beta_k = \frac{2}{|\tilde{\mu}|+|\tilde{L}|}    \label{eq:symmetric:beta_k}
 \end{align}
 as an estimate of the optimal value $2/(\mu+L)$.

\section{Numerical experiments}
\label{sec:experiments}
 
  In this section, we validate our theoretical findings by solving three nonlinear problems:  
   (I) the modified Bratu problem; (II) the Chandrasekhar H-equation; (III) the regularized logistic regression. 
   Additional experimental results can be found in \Cref{appendix:additional_expr}. 
 Let AM-I and AM-II denote the restarted Type-I and restarted Type-II AM.  AM-I($m$) and AM-II($m$) denote the Type-I AM and Type-II AM with $m_k = \min\{m,k\}$.  ST-AM-I and ST-AM-II are abbreviations of the restarted Type-I and restarted Type-II ST-AM. 
 Since this work focuses on the theoretical properties of AM, we used the iteration number  as the evaluation metric of convergence in the experiments.

 \subsection{Modified Bratu problem}  To verify \Cref{them:am} and \Cref{them:short_term_convergence}, we considered solving the modified Bratu problem introduced in \cite{fang2009two}:
 \begin{align}   \label{eq:bratu}
  u_{xx}+u_{yy}+\alpha u_x+\lambda e^{u} = 0, 
 \end{align}
 where $u$ is a function of $(x,y)\in \mathcal{D}=[0,1]^2$, and $\alpha, \lambda \in\mathbb{R}$ are constants. The boundary condition is $u(x,y)\equiv 0 $ for $(x,y)\in\partial\mathcal{D}$.  The equation was discretized using centered differences on a $200\times 200$ grid. The resulting problem is a  system of nonlinear equations: 
  $ F(U) = 0,$
 where $U\in\mathbb{R}^{200\times 200}$ and $F:\mathbb{R}^{200\times 200}\rightarrow\mathbb{R}^{200\times 200}$.  Following \cite{fang2009two},  we set $\lambda = 1$ and initialized $U$ with $0$.
 The Picard iteration is $U_{k+1} = U_k+\beta r_k$, where $r_k = F(U_k)$ is the residual. For the restarted AM and ST-AM, we set $\tau = 10^{-32}$, $m=1000$, and  $\eta = \infty$   
 since a large $m_k$ is beneficial for solving this problem. 
 
 \subsubsection{Nonsymmetric Jacobian} We set $\alpha = 20$ so that the Jacobian $F'(U)$ is not symmetric. For the Picard iteration, we tuned $\beta_k$ and set it as $6\times 10^{-6}$.  
 We applied the adaptive mixing strategy \eqref{eq:beta_k} for AM-I and AM-II with $\beta_0 = 1$. 
 
 \begin{figure}[t]
\centering 
\subfigure{
\includegraphics[width=0.46\textwidth]{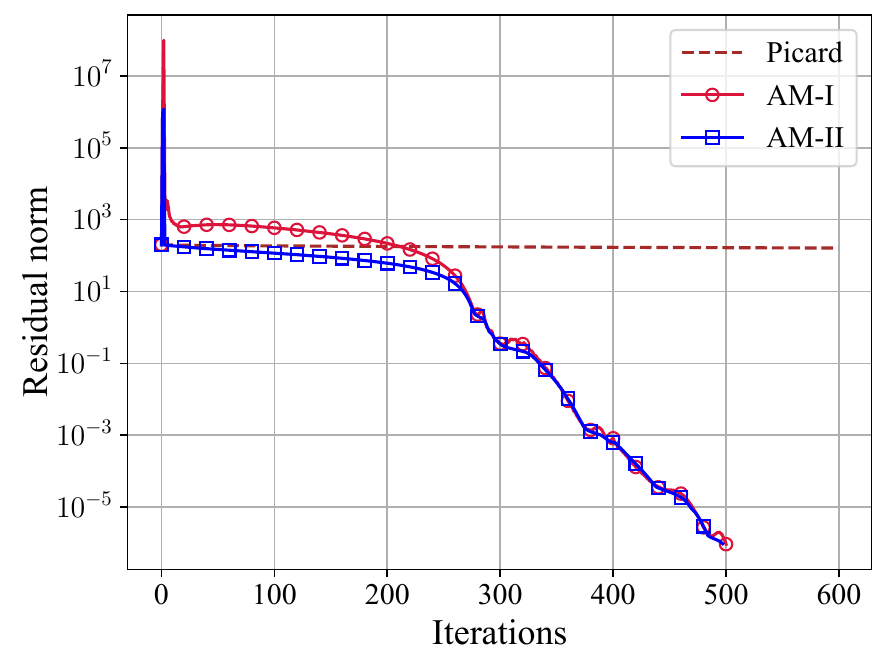}
}
\subfigure{
\includegraphics[width=0.46\textwidth]{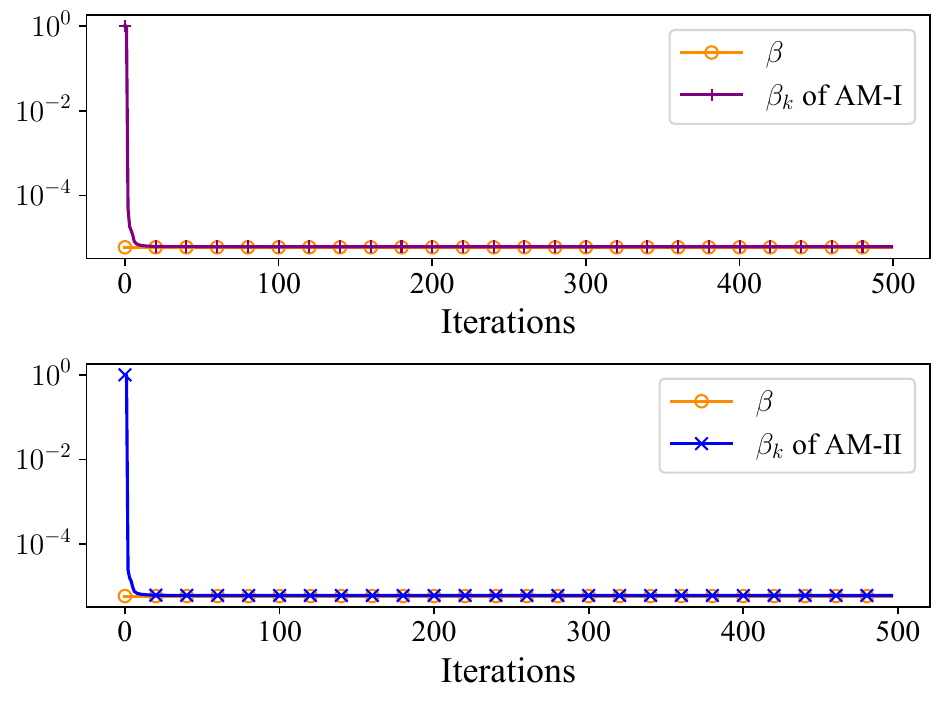}
}
\caption{The modified Bratu problem with $\alpha = 20$. Left: $\|F(U_k)\|_2$ of each method.  Right: $\beta$ of Picard iteration and $\beta_k$ of AM-I/AM-II during the iterations. }
\label{fig:bratu:nonsymmetric}
\end{figure} 

 As shown in \Cref{fig:bratu:nonsymmetric}, both AM-I and AM-II converge much faster than the Picard iteration. In fact, to achieve $\|F\|_2 \leq 10^{-6}$, AM-I uses  500 iterations, and AM-II uses 497 iterations, where no restart occurs in either method. Hence, the results verify \Cref{them:am} and suggest that AM methods significantly accelerate the Picard iteration when $m_k$ is large in solving this problem. Also, 
 observe that AM-I and AM-II diverge in the initial stage due to the inappropriate choice $\beta_0=1$. Nonetheless, from \Cref{fig:bratu:nonsymmetric}, we see the $\beta_k$ is quickly adjusted to the optimal value $\beta = 6\times 10^{-6}$ based on the eigenvalue estimates.  Thus we only need to compute the eigenvalue estimates within a few steps and keep $\beta_k$ unchanged in the later iterations.  

 \subsubsection{Symmetric Jacobian} We set $\alpha = 0$ so that the Jacobian $F'(U)$ is symmetric.  We compared the restarted ST-AM methods with the Picard iteration, the limited-memory AM, and the full-memory AM. 
 By a grid search in $\{1\times 10^{-6}, 2\times 10^{-6},\dots,1\times 10^{-5}\}$, we chose $\beta = 6\times 10^{-6}$  for the Picard iteration. Then, we set $\beta_k=6\times 10^{-6}$ for AM-I($2$),  AM-II($2$), AM-I($\infty$), and AM-II($\infty$).  We applied the adaptive mixing strategy \eqref{eq:symmetric:beta_k} for ST-AM-I and ST-AM-II with $\beta_0 = 1$. 
 
 \begin{figure}[t]
\centering 
\subfigure{
\includegraphics[width=0.46\textwidth]{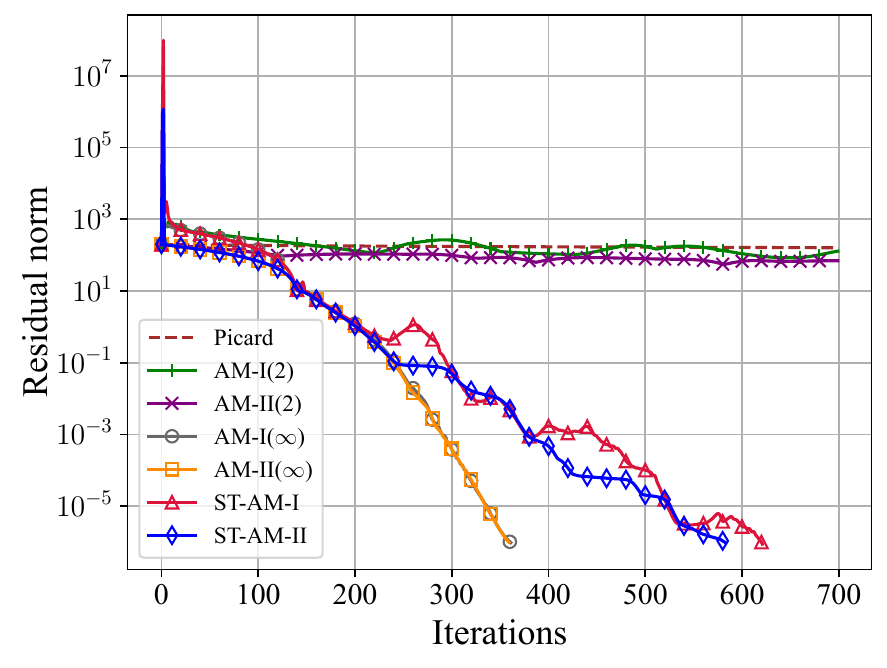}
}
\subfigure{
\includegraphics[width=0.46\textwidth]{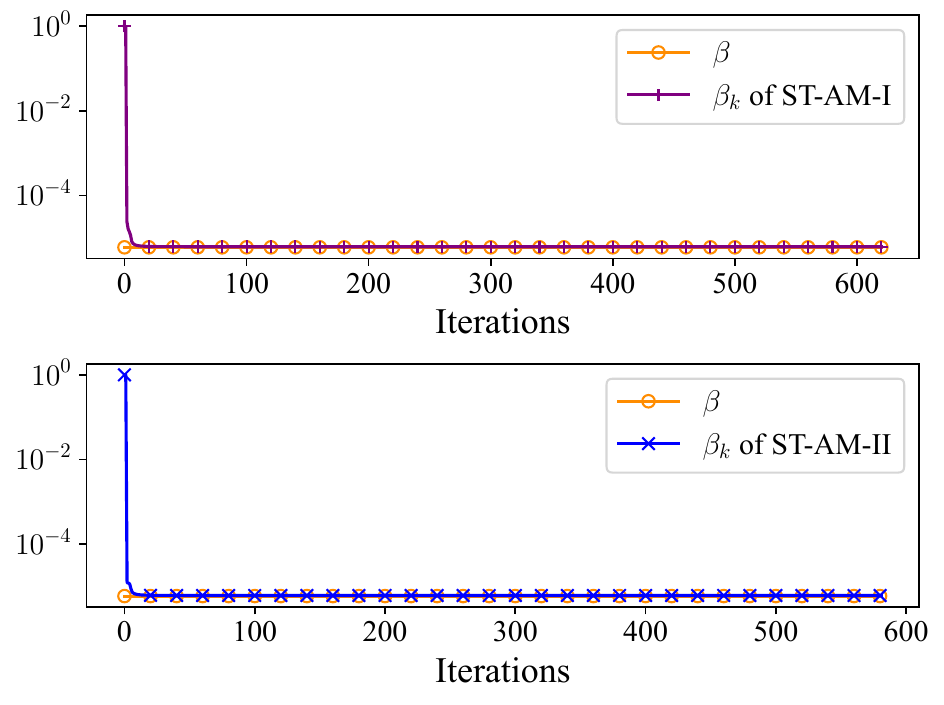}
}
\caption{The modified Bratu problem with $\alpha = 0$.  Left: $\|F(U_k)\|_2$ of each method. Right: $\beta$ of Picard iteration and  $\beta_k$ of ST-AM-I/ST-AM-II during the iterations. }
\label{fig:bratu:symmetric}
\end{figure} 

       
 
 The results in \Cref{fig:bratu:symmetric} show the convergence of each method and the choices of $\beta_k$ in ST-AM-I/ST-AM-II.  Observe that AM-I($2$) and AM-II($2$) perform similarly to the Picard iteration, which is reasonable since $m = 2$ is too small. On the other hand, ST-AM-I and ST-AM-II exhibit significantly faster convergence rates as predicted by \Cref{them:short_term_convergence} (no restart occurs in either method, i.e., $m_k = k$).  
 We also find that the $\beta_k$ of ST-AM-I/ST-AM-II quickly converges to $6.19\times 10^{-6}$. So the curve of ST-AM-I/ST-AM-II roughly coincides with that of the full-memory AM-I/AM-II in the early stage. However, due to the loss of orthogonality, ST-AM-I and ST-AM-II require more iterations to achieve $\|F\|_2\leq 10^{-6}$ than the full-memory methods.   

\subsection{Chandrasekhar H-equation}
 To check the effect of the restarting conditions \eqref{ineq:cond1}-\eqref{ineq:cond3}, we applied the restarted AM to solve  
 the Chandrasekhar H-equation considered in \cite{toth2015convergence,chen2019convergence}:  
 \begin{align}
  \mathcal{F}(H)(\mu) = H(\mu) - \left( 1-\frac{\omega}{2}\int_0^1 \frac{\mu H(\nu)d\nu}{\mu+\nu} \right)^{-1} = 0,
 \end{align}
 where $\omega\in [0,1]$ is a constant and the unknown is a continuously differentiable function $H$ defined in $[0,1]$.  Following \cite{toth2015convergence}, we discretized the equation with the composite midpoint rule. The resulting equation is 
 \begin{align}
  h^i = G(h)^i := \left( 1-\frac{\omega}{2N}\sum_{j=1}^N \frac{\mu_ih^j}{\mu_i+\mu_j} \right)^{-1}.
 \end{align}
 Here $h^i$ is the $i$-th component of $h\in\mathbb{R}^N$, $G(h)^i$ is the $i$-th component of $G(h) \in \mathbb{R}^N$, and $\mu_i = (i-1/2)/N$ for $1\leq i \leq N$. Define $F(h)=h-G(h) =0$, and $r_k = G(h_k)-h_k$ is the residual at $h_k$.   We set $N=500$ and considered $\omega=0.5, 0.99, 1$. The initial point was $h_0 = (1,1,\dots,1)^\mathrm{T}$. 
 Since the fixed-point operator $G$ is nonexpansive and the Picard iteration $h_{k+1} = G(h_k) = h_k+r_k$ converges in this case, we set $\beta_k = 1$ for both restarted AM methods. (The $h_k$ here has no relation with that in \Cref{sec:mix_params}.)

\begin{table}[t]
\caption{ Results of the restarted AM methods with different $\eta, m,\tau$. The table shows the iteration number $k$ to achieve $\|F(h_k)\|_2/\|F(h_0)\|_2 \leq 10^{-8}$ (``--'' means failure of the method).  }
\label{table:H_eq:restart_am}
\centering
\resizebox{0.8\textwidth}{!}{
\begin{tabular}{c c c c c c c c c}
\toprule
  \multicolumn{3}{c}{Hyperparameters} & \multicolumn{3}{c}{AM-I} & \multicolumn{3}{c}{AM-II}\\
 \cmidrule(lr){1-3}  \cmidrule(lr){4-6} \cmidrule(lr){7-9} 
 $\eta$ &  $m$ & $\tau$   & $\omega=0.50$ & $\omega=0.99$ & $\omega=1.00$ & $\omega=0.50$ & $\omega=0.99$ & $\omega=1.00$ \\  
  \cmidrule(lr){1-3} \cmidrule(lr){4-6}  \cmidrule(lr){7-9}
 $\infty$ & 4 & $10^{-15}$ & 5 & 11 & 40 & 5 & 10 & 30 \\
 $\infty$ & 4 & $10^{-32}$ & 5 & 11 & 40 & 5 & 10 & 30 \\
 $\infty$ & 100 & $10^{-15}$ & 5 & 12 & 34 & 5 & 11 & 27  \\
 $\infty$ & 100 & $10^{-32}$ & 5 & 10 & -- & 5 & 102 & 304 \\
 1 & 4 & $10^{-15}$ & 5 & 11 & 40 & 5 & 10 & 37 \\
 1 & 4 & $10^{-32}$ & 5 & 11 & 40 & 5 & 10 & 37 \\
 1 & 100 & $10^{-15}$ & 5 & 12 & 32 & 5 & 11 & 41  \\
 1 & 100 & $10^{-32}$ & 5 & 10 & 202 & 5 & 102 & 304 \\
\bottomrule
\end{tabular}}
\end{table}

  
  We studied the convergence of restarted AM with different settings of $m, \tau,$ and $\eta$. \Cref{table:H_eq:restart_am} tabulates the results.  This problem is hard to solve when $\omega$ approaches $1$.  
  For the easy case $\omega = 0.5$, the restarting conditions have neglectable effects on the convergence. However, for $\omega = 0.99$ and especially for $\omega=1$, the restarting conditions are critical, which help avoid the divergence of the iterations.  It is preferable to use a small $m$ in this problem.   
  By comparing the case $m=100, \tau=10^{-15}$ with the case $m=100, \tau=10^{-32}$, we find that using \eqref{ineq:cond2} to control the condition number of $V_k^\mathrm{T}Q_k$ is necessary. Also, setting $\eta=1$ in  \eqref{ineq:cond3} is helpful for AM-I.  
 
%
 
\subsection{Regularized logistic regression}
  To validate the effectiveness of ST-AM-I and ST-AM-II for solving unconstrained optimization problems, 
 we considered solving the regularized logistic regression:
 \begin{align}
  \min_{x\in\mathbb{R}^d}f(x):= \frac{1}{T}\sum_{i=1}^T \log (1+\exp(-y_ix^\mathrm{T}\xi_i))+\frac{w}{2}\|x\|_2^2,
 \end{align}
 where $\xi_i\in\mathbb{R}^d$ is the $i$-th input data sample and $y_i=\pm 1$ is the corresponding label. We used the ``madelon'' dataset from LIBSVM \cite{chang2011libsvm}, which contains 2000 data samples ($T=2000$) and 500 features ($d=500$).  We considered  $w=0.01$. The compared methods were gradient descent (GD), Nesterov's method \cite[Scheme~2.2.22]{Nesterov2018},  and the limited-memory AM methods with $m=2$ and $m=20$. For GD, we tuned the step size and set it as $1.6$. For AM methods, we also set $\beta_k = 1.6$. Let $x^*$ denote the minimizer. We used the smallest and the largest Ritz values of $\nabla^2 f(x^*)$ to approximate $\mu$ and $L$, which are required for Nesterov's method. 
  For the restarted ST-AM methods, we set $m=40$, $\tau=1\times 10^{-15}$, $\eta = \infty$ and applied the adaptive mixing strategy \eqref{eq:symmetric:beta_k} with $\beta_0 = 1$ to choose $\beta_k$. 
 
 \begin{figure}[t]
\centering 
\subfigure{
\includegraphics[width=0.46\textwidth]{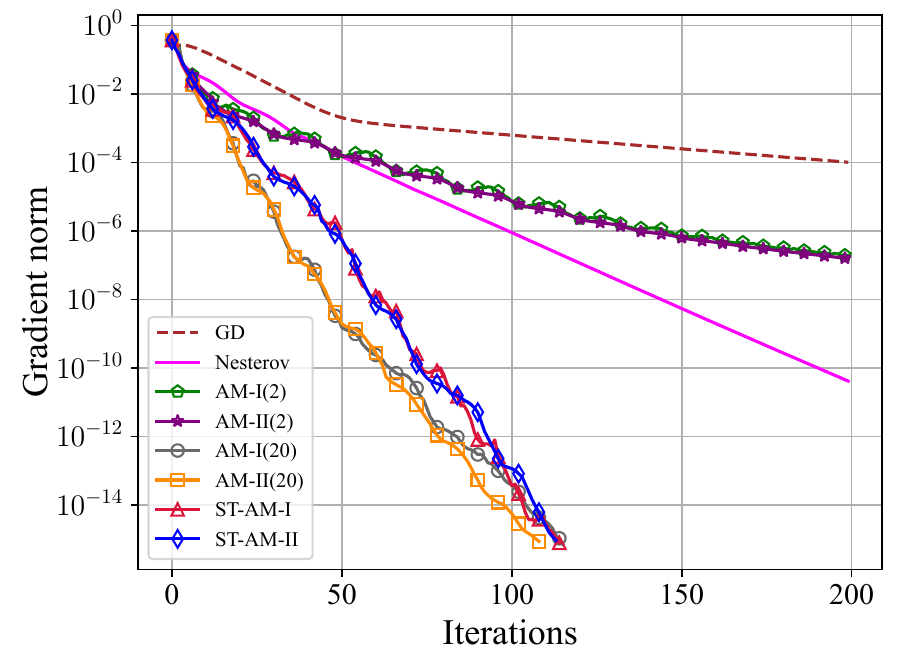}
}
\subfigure{
\includegraphics[width=0.46\textwidth]{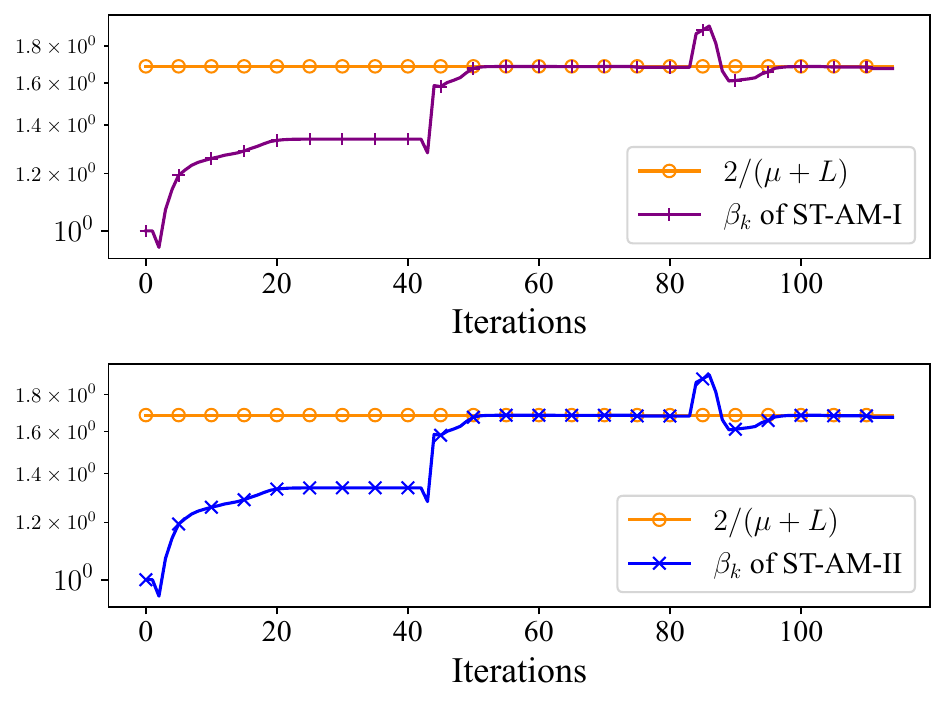}
}
\caption{The regularized logistic regression with $w = 0.01$. Left: $\|\nabla f(x_k)\|_2$ of each method. Right:  $2/(\mu+L)$, and the $\beta_k$ of ST-AM-I and ST-AM-II. }
\label{fig:logistic_regression}
\end{figure} 
 
  \Cref{fig:logistic_regression} shows the convergence of each method and the choices of $\beta_k$ in ST-AM-I and ST-AM-II.  Like AM-I(2) and AM-II(2), the ST-AM methods only use two vector pairs for the AM update. However, they have improved convergence, and the convergence rates are close to those of AM-I(20) and AM-II(20). Also,  the mixing parameters for restarted ST-AM methods do not need to be tuned manually. With the convergence of $x_k$ to the minimizer $x^*$, the $\beta_k$ is  adjusted to $2/(\mu+L)$.  
 
 
 
 

 \begin{table}[t]
  \caption{ Results of the restarted ST-AM methods with different $m, \tau$. The table shows the iteration number $k$ to achieve $\|\nabla f(x_k)\|_2\leq 10^{-15}$.}   \label{table:logistic_regression}
  \centering
  \resizebox{0.85\textwidth}{!}{
  \begin{tabular}{cccccc} 
    \toprule   	
    $(m,\tau)$  & $(28,10^{-15})$   & $(40,10^{-15})$  & $(1000,10^{-15})$  & $(1000,10^{-7})$  & $(1000,10^{-32})$    \\
    \midrule
      ST-AM-I   & 104   &  115   & 133 & 107   & 242       \\
      ST-AM-II  & 105   &  114  & 136 & 111   & 267    \\      
    \bottomrule
  \end{tabular}
  }
\end{table}
 
 In \Cref{table:logistic_regression}, we show the effects of the restarting conditions on ST-AM-I and ST-AM-II.  We only consider $m$ and $\tau$ in \eqref{ineq:cond1} and  \eqref{ineq:cond2} since both methods converge in solving this problem. The results suggest that well-chosen $m$ and $\tau$ can lead to improved convergence.

\section{Conclusions}
\label{sec:conclusions}
 
  In this paper, we study the restarted AM methods formulated with modified historical sequences and certain restarting conditions. Using a multi-step analysis, we extend the relationship between AM and Krylov subspace methods to nonlinear fixed-point problems. We prove that under reasonable assumptions, the long-term convergence behaviour of the restarted Type-I/Type-II AM is dominated by a minimization problem that also appears in the theoretical analysis of Krylov subspace methods. The convergence analysis provides a new assessment of the efficacy of AM in practice and justifies the potential local improvement of restarted Type-II AM over the fixed-point iteration. As a by-product of the restarted AM, the eigenvalues of the Jacobian can be efficiently estimated, based on which we can choose the mixing parameter adaptively. When the Jacobian is symmetric, we derive the short-term recurrence variants of restarted AM methods and the simplified eigenvalue estimation procedure. The short-term recurrence AM methods are memory-efficient and can significantly accelerate the fixed-point iterations. The experiments validate our theoretical results and the restarting conditions.

\appendix

\section{Proofs of Section~\ref{sec:restarted_am}} 
\subsection{Proof of Proposition~\ref{prop:modified_seqs}}  \label{subsec:proof:prop:modified_seqs}
 \begin{proof}  
  We prove the results by induction. 
  
  If $m_k=1$, then $p_k = \Delta x_{k-1}$, $q_k = \Delta r_{k-1}$. The Property~\ref{am_update:property1} and Property~\ref{am_update:property2} hold, and $v_k^\mathrm{T}q_k=z_k^\mathrm{T}\Delta r_{k-1}\neq 0$. Then by \eqref{eq:r_k_type2}, $\bar{r}_k \perp v_k$. Hence, \eqref{am_update} produces the same iterate as \eqref{eq:am_update}. 
  
  For $m_k>1$, suppose that from the $(k-m_k+1)$-th iteration to the $(k-1)$-th iteration,  Properties~\ref{am_update:property1}-\ref{am_update:property3} hold, and \eqref{am_update} produces the same iterates as \eqref{eq:am_update}. Then, at the $k$-th iteration, we first prove that $q_k^j \perp {\rm span}\{v_{k-m_k+1},\dots,v_{k-m_k+j}\}$, $j=1,\dots,m_k-1$, by induction. 
  
  For $j=1$, since $v_{k-m_k+1}^\mathrm{T}q_{k-m_k+1} \neq 0$, it follows that $q_k^1 \perp v_{k-m_k+1}$ due to \eqref{eq:q_k_type2}. 
  Consider $1<j\leq m_k-1$. Due to the inductive hypothesis,  $0\neq \det(Z_{k-1}^\mathrm{T}R_{k-1})=
  \det(S_{k-1}^\mathrm{T}V_{k-1}^\mathrm{T}Q_{k-1}S_{k-1})$. It follows that $\det(V_{k-1}^\mathrm{T}Q_{k-1}) \neq 0$. So the diagonal element $v_{k-m_k+j}^\mathrm{T}q_{k-m_k+j} \neq 0$, which together with \eqref{eq:q_k_type2}  implies $q_k^j \perp v_{k-m_k+j}$. Also,  both  $q_k^{j-1}$ and $q_{k-m_k+j}$ are orthogonal to $ {\rm span}\{v_{k-m_k+1},\dots,v_{k-m_k+j-1}\}$ by the inductive hypotheses. Thus, $q_k^j \perp {\rm span}\{v_{k-m_k+1},\dots,v_{k-m_k+j-1}\}$. We complete the induction. 
  
  Consequently, we have $q_k=q_k^{m_k-1} \perp {\rm span}\{v_{k-m_k+1},\dots,v_{k-1}\} = {\rm range}(V_{k-1})$. With the inductive hypothesis that $V_{k-1}^\mathrm{T}Q_{k-1}$ is lower triangular, it follows that $V_k^\mathrm{T}Q_k$ is lower triangular, namely the Property~\ref{am_update:property2}. 
  
  We prove that $q_k \neq 0$. Note that $q_k = \Delta r_{k-1}-Q_{k-1}\zeta_k$.  If $q_k = 0$, then $\Delta r_{k-1} \in {\rm range}(Q_{k-1}) = {\rm range}(R_{k-1})$, which is impossible since $R_k$ has full column rank due to $\det(Z_k^\mathrm{T}R_k) \neq 0$. Hence $q_k \neq 0$ and $R_k=Q_kS_k$, where $S_k$ is unit upper triangular. Since $p_k = \Delta x_{k-1}-P_{k-1}\zeta_k$, we also have $X_k = P_kS_k$. So the Property~\ref{am_update:property1} holds. 
  
  Next, we prove that $r_k^j \perp {\rm span}\{v_{k-m_k+1},\dots,v_{k-m_k+j}\}$, $j=1,\dots,m_k$, by induction. As  Properties~\ref{am_update:property1}-\ref{am_update:property2} hold at the $k$-th iteration, we have $0 \neq\det(Z_k^\mathrm{T}R_k) = \det(S_k^\mathrm{T}V_k^\mathrm{T}Q_kS_k)$, which implies that $\det(V_k^\mathrm{T}Q_k) \neq 0$. Hence $v_{k-m_k+j}^\mathrm{T}q_{k-m_k+j}\neq 0$ for $j=1,\dots,m_k$. Then we have $r_k^1 \perp v_{k-m_k+1}$ due to \eqref{eq:r_k_type2}. Consider $1< j\leq m_k$. Due to $v_{k-m_k+j}^\mathrm{T}q_{k-m_k+j} \neq 0$ and \eqref{eq:r_k_type2}, $r_k^j \perp v_{k-m_k+j}$. Also, by the inductive hypotheses, both $r_k^{j-1}$ and $q_{k-m_k+j}$ are orthogonal to  ${\rm span}\{v_{k-m_k+1},\dots,v_{k-m_k+j-1}\}$. It follows that $r_k^{j} \perp {\rm span}\{v_{k-m_k+1},\dots,v_{k-m_k+j-1}\}$. Thus, we complete the induction. It yields that $\bar{r}_k = r_k^{m_k}\perp {\rm range}(V_k)$, namely the Property~\ref{am_update:property3}. 
  
  Finally, the complete update of \eqref{am_update} is $x_{k+1}=x_k+G_kr_k$, where
  \begin{equation}
   G_k = \beta_kI - (P_k+\beta_kQ_k)(V_k^\mathrm{T}Q_k)^\mathrm{-1}V_k^\mathrm{T}.  \label{eq:G_k}
  \end{equation}
 Here, we use Property~\ref{am_update:property3} which implies $\Gamma_k = (V_k^\mathrm{T}Q_k)^{-1}V_k^\mathrm{T}r_k$.  Then with $P_k = X_kS_k^{-1}$ and $Q_k = R_kS_k^{-1}$, the equivalent form of \eqref{eq:G_k} is $G_k = \beta_kI-(X_k+\beta_kR_k)(Z_k^\mathrm{T}R_k)^{-1}Z_k^\mathrm{T}$, which is the original AM update \eqref{eq:am_update}. So \eqref{am_update} produces the same iterate as \eqref{eq:am_update}.  As a result, we complete the induction. 
 \end{proof}

\section{Proofs of Section~\ref{sec:convergence}}
 
 \subsection{Proof of Proposition~\ref{prop:linear}}  \label{subsec:proof:prop:linear}
\begin{proof}
 1. The Properties~\ref{prop:linear:property1}-\ref{prop:linear:property3} are known results \cite{walker2011anderson}. We give the proof here for completeness.
 
  The definition of $g$ suggests that the residual $r_k = g(x_k)-x_k=b-Ax_k$ for $k\geq 0$ and  $R_k = -AX_k$ for $k\geq 1$. Recall that each $\Gamma_j$ is determined by solving
  \begin{equation}
  \bar{r}_j = r_j - R_j\Gamma_j \perp {\rm range}(Z_j),   \label{eq:proj_cond}
 \end{equation}
 where $1\leq j\leq k$ and $k\geq 1$. 
 The condition $\det(Z_j^\mathrm{T}R_j)\neq 0$ ensures that $\Gamma_j$ is uniquely determined. Thus the AM updates are well defined. 
  
  Since $A$ is nonsingular and $R_k = -AX_k$, it follows that ${\rm rank}(X_k) = {\rm rank}(R_k)$. Then, due to $\det(Z_k^\mathrm{T}R_k)\neq 0$, we have ${\rm rank}(Z_k)={\rm rank}(R_k) = k$. So $ {\rm rank}(X_k)=k $. 
 We first prove $ {\rm range}(X_k)=\mathcal{K}_k(A,r_0) $
   by induction. 
 
 First, $\Delta x_0 = \beta_0r_0 $ since $x_1 = x_0+\beta_0r_0$. If $ k=1$, then the proof is complete. 
Suppose that $ k>1 $ and $ {\rm range}(X_{k-1})=\mathcal{K}_{k-1}(A,r_0) $. Define $e^k\in\mathbb{R}^k$ to be the vector with all elements being ones. 
 From the AM update \eqref{eq:am_update}, we have 
 \begin{align*}
\Delta x_{k-1}
&= \beta_{k-1}r_{k-1}-(X_{k-1}+\beta_{k-1}R_{k-1})\Gamma_{k-1} \nonumber \\
&= \beta_{k-1}(b-Ax_{k-1})-(X_{k-1}-\beta_{k-1}AX_{k-1})\Gamma_{k-1} \nonumber \\
&=\beta_{k-1}b-\beta_{k-1}A(x_0+\Delta x_0+\cdots+\Delta x_{k-2}) -(X_{k-1}-\beta_{k-1}AX_{k-1})\Gamma_{k-1}  \nonumber \\
&=\beta_{k-1}r_0-\beta_{k-1}AX_{k-1}e^{k-1}-(X_{k-1}-\beta_{k-1}AX_{k-1})\Gamma_{k-1}. 
\end{align*}
 Since $ {\rm range}(X_{k-1}) =  \mathcal{K}_{k-1}(A,r_0)$, we have $ {\rm range}(AX_{k-1})\subseteq \mathcal{K}_k(A,r_0)$. Also, noting that $r_0\in\mathcal{K}_{k-1}(A,r_0)$, we have $ \Delta x_{k-1}\in \mathcal{K}_k(A,r_0) $. Thus, $ {\rm range}(X_k) \subseteq \mathcal{K}_k(A,r_0) $. Since $ {\rm rank}(X_k) = k$, it follows that $ {\rm range}(X_k) = \mathcal{K}_k(A,r_0) $, thus completing the induction. 
 

 Since $ r_k = b-Ax_k=b-A(x_0+X_ke^k)=r_0-AX_ke^k $, it follows that 
$ r_k-R_k\Gamma = r_k + AX_k\Gamma = r_0-AX_ke^k + AX_k\Gamma = r_0 -AX_k\tilde{\Gamma}$, where $ \tilde{\Gamma}=e^k-\Gamma $, for $\forall ~ \Gamma\in\mathbb{R}^k$. So
$\Gamma_k$ solves (\ref{eq:proj_cond}) for $j=k$ if and only if $ \tilde{\Gamma}_k=e^k-\Gamma_k $  solves 
\begin{align}
  r_0 - AX_k\tilde{\Gamma}_k  \perp {\rm range}(Z_k).   \label{eq:krylov}
\end{align}
 Since ${\rm range}(X_k) = \mathcal{K}_k(A,r_0)$, the condition (\ref{eq:krylov}) is equivalent to 
\begin{align}
 r_0 - Az \perp {\rm range}(Z_k)  ~~ \mbox{s.t.} ~~ z\in \mathcal{K}_k(A,r_0).     \label{eq:galerkin}
\end{align}
 Here ${\rm range}(Z_k) = \mathcal{K}_k(A,r_0)$ for the Type-I method, and ${\rm range}(Z_k) = A\mathcal{K}_k(A,r_0)$ for the Type-II method. Since the initializations are identical, the conditions \eqref{eq:galerkin} for Type-I and Type-II methods are the Petrov-Galerkin conditions for the Arnoldi's method and GMRES, respectively. 
 Due to the nonsingularity of $Z_k^\mathrm{T}R_k$, the solution of (\ref{eq:krylov}) is also unique. Therefore, we have
\begin{align*}
\bar{x}_k =x_k-X_k\Gamma_k=x_k-X_k(e^k-\tilde{\Gamma}_k) =x_0+X_k\tilde{\Gamma}_k = x_k^{\text{\rm A}}, 
\end{align*}
 for the Type-I method, and $\bar{x}_k =x_k^{\text{\rm G}} $ for the Type-II method. 
   
 2. Consider the case that $A$ is positive definite, and the algorithm has not found the exact solution, i.e. $r_j \neq 0$ for $j = 0,\dots,k$. We prove the result by induction. 
  
  If $k=1$, then $\Delta x_0 = \beta_0 r_0$, and $\Delta r_0 = -\beta_0 Ar_0$. Hence $Z_1^\mathrm{T}R_1 = \Delta x_0^\mathrm{T}\Delta r_0 = -\beta_0^2 r_0^\mathrm{T}Ar_0$ for the Type-I method; $Z_1^\mathrm{T}R_1 = \Delta r_0^\mathrm{T}\Delta r_0 = \beta_0^2 r_0^\mathrm{T}A^\mathrm{T}Ar_0$ for the Type-II method. Since $r_0 \neq 0$ and $A$ is positive definite, it follows that $\det(Z_1^\mathrm{T}R_1) \neq 0$. 
  
  For $k>1$, suppose that $\det(Z_{k-1}^\mathrm{T}R_{k-1})\neq 0$. It indicates ${\rm rank}(R_{k-1}) = k-1$, thus ${\rm rank}(X_{k-1}) = k-1$. We prove $\det(Z_k^\mathrm{T}R_k) \neq 0$ by contradiction. 
  
  If $\det(Z_k^\mathrm{T}R_k) = 0$, then there exists a nonzero $y\in\mathbb{R}^k$ such that $Z_k^\mathrm{T}R_ky = 0$. Then $y^\mathrm{T}Z_k^\mathrm{T}R_ky = 0$. Note that $Z_k^\mathrm{T}R_k = X_k^\mathrm{T}R_k = -X_k^\mathrm{T}AX_k$ for the Type-I method, and $Z_k^\mathrm{T}R_k = R_k^\mathrm{T}R_k = X_k^\mathrm{T}A^\mathrm{T}AX_k$ for the Type-II method. Since $A$ is positive definite, we have $X_ky = 0$, which implies that $X_k$ is rank deficient. As $X_{k-1}$ has full column rank, it yields $\Delta x_{k-1} = -X_{k-1}\Gamma_{k-1}+\beta_{k-1}\bar{r}_{k-1} \in {\rm range}(X_{k-1})$. Hence $\bar{r}_{k-1}\in {\rm range}(X_{k-1})$. So $\bar{r}_{k-1}=X_{k-1}\xi$ for some $\xi\in\mathbb{R}^{k-1}$. Since $\det(Z_{k-1}^\mathrm{T}R_{k-1}) \neq 0$, the condition $\bar{r}_{k-1}=r_{k-1}-R_{k-1}\Gamma_{k-1} \perp Z_{k-1}$ has a unique solution. Thus
  \begin{align}
   0 = \bar{r}_{k-1}^\mathrm{T}Z_{k-1}\xi = \xi^\mathrm{T}X_{k-1}^\mathrm{T}Z_{k-1}\xi.    \label{eq:r_bar_Z_k}
  \end{align}
 For the Type-I method, $X_{k-1}^\mathrm{T}Z_{k-1} = X_{k-1}^\mathrm{T}X_{k-1}$; for the Type-II method, $X_{k-1}^\mathrm{T}Z_{k-1} = X_{k-1}^\mathrm{T}R_{k-1} = -X_{k-1}^\mathrm{T}AX_{k-1}$. Since $X_{k-1}$ has full column rank and $A$ is positive definite, it follows from \eqref{eq:r_bar_Z_k} that $\xi = 0$ for both cases, which yields $\bar{r}_{k-1} = 0$. However, it is impossible because when $\bar{r}_{k-1} = 0$, we have $x_k = \bar{x}_{k-1}$ and $r_k = \bar{r}_{k-1} = 0$, which contradicts the assumption that $r_k \neq 0$. Therefore, $\det(Z_k^\mathrm{T}R_k) \neq 0$. We complete the induction. 
 
 3. Since $\det(Z_j^\mathrm{T}R_j) \neq 0$, $j=1,\dots,k$, it follows from \Cref{prop:modified_seqs} that the constructions of the modified historical  sequences $P_k$ and $Q_k$ are well defined. The Property~\ref{am_update:property1} in \Cref{prop:modified_seqs} further yields the relation \eqref{prop:linear:modified_seqs}.
 \end{proof}

\subsection{Proof of Lemma~\ref{lemma:diff}}   \label{subsec:proof:lemma:diff}
\begin{proof}
 
 
 The proof follows the technique in \cite{wei2022}. 
 Besides \eqref{lemma:diff:r_k} and \eqref{lemma:diff:x_k}, we shall also prove the following relations.  
 \begin{align}  
  & x_k\in\mathcal{B}_{\hat{\rho}}(x^*),    \label{lemma:x_k}   \\
  & |\zeta_k^{(j)}| = \mathcal{O}(1),  &  |\hat{\zeta}_k^{(j)}-\zeta_k^{(j)}|=\hat{\kappa}\mathcal{O}(\|x_{k-m_k}-x^*\|_2),  \label{lemma:diff:zeta_k}  \\
  & \|p_k\|_2 = \mathcal{O}(\|x_{k-m_k}-x^*\|_2), &  \|q_k\|_2 = \mathcal{O}(\|x_{k-m_k}-x^*\|_2),   \label{lemma:p_k}  \\
  & \|p_k-\hat{p}_k\|_2=\hat{\kappa}\mathcal{O}(\|x_{k-m_k}-x^*\|_2^2), &
    \|q_k-\hat{q}_k\|_2=\hat{\kappa}\mathcal{O}(\|x_{k-m_k}-x^*\|_2^2),  \label{lemma:diff:p_k}  \\
  & |\Gamma_k^{(j)}| = \mathcal{O}(1), &  |\hat{\Gamma}_k^{(j)}-\Gamma_k^{(j)}|=\hat{\kappa}\mathcal{O}(\|x_{k-m_k}-x^*\|_2),   \label{lemma:diff:Gamma_k} \\
  & \|\bar{x}_k-\bar{\hat{x}}_k\|_2 = \hat{\kappa}\mathcal{O}(\|x_{k-m_k}-x^*\|_2^2),  & \|\bar{r}_k-\bar{\hat{r}}_k\|_2 = \hat{\kappa}\mathcal{O}(\|x_{k-m_k}-x^*\|_2^2),   \label{lemma:diff:bar_x_k}
 \end{align}
 where $j=0,\dots,m_k$. 
 Here, for convenience, 
 we define $\hat{\zeta}_k^{(0)}= \zeta_k^{(0)}=\hat{\zeta}_k^{(m_k)}= \zeta_k^{(m_k)}=0$, $\hat{\Gamma}_k^{(0)}=\Gamma_k^{(0)} = 0$; when $m_k = 0$, define $\hat{p}_{k}=\hat{q}_{k}= p_{k} = q_{k} = \mathbf{0}$, $\bar{x}_{k}=x_{k}$, $\bar{r}_{k}=r_{k}$, and $\bar{\hat{x}}_{k}=\hat{x}_{k}$, $\bar{\hat{r}}_{k}=\hat{r}_{k}$; when $m_k>0$, 
 $\bar{\hat{x}}_k = \hat{x}_k-\hat{P}_k\hat{\Gamma}_k$, $\bar{\hat{r}}_k = \hat{r}_k-\hat{Q}_k\hat{\Gamma}_k$.  Then, the two processes to generate $\{x_k\}$ and $\{\hat{x}_k\}$ are
 \begin{align*}
  x_{k+1} = \bar{x}_k+\beta_k\bar{r}_k,   ~~ \mbox{and}  ~~ 
  \hat{x}_{k+1} = \bar{\hat{x}}_k+\beta_k\bar{\hat{r}}_k. 
 \end{align*}
 

 We first prove \eqref{lemma:x_k}. 
 Due to \eqref{assum:jacobian}, we have the following relation:
 \begin{align}
  \mu\|x_k-x^*\|_2 \leq \|r_k\|_2=\|h(x_k)-h(x^*)\|_2 \leq L\|x_k-x^*\|_2.  \label{ineq:r_k}
\end{align}  
  Choose $\|x_0-x^*\|_2 \leq \frac{\mu\hat{\rho}}{\eta_0L}$. With the condition $\|r_k\|_2 \leq \eta_0\|r_0\|_2$, we obtain
 \begin{align}
  \|x_k-x^*\|_2 \leq \frac{1}{\mu}\|r_k\|_2 \leq \frac{\eta_0}{\mu}\|r_0\|_2
   \leq \frac{\eta_0 L}{\mu}\|x_0-x^*\|_2\leq \frac{\eta_0 L}{\mu}\cdot \frac{\mu\hat{\rho}}{\eta_0L} = \hat{\rho},   \label{ineq:bound:x_k_r_k}
 \end{align}
 namely \eqref{lemma:x_k}. The \eqref{ineq:bound:x_k_r_k} also implies we can choose sufficiently small $\|x_0-x^*\|_2$ to ensure $\|x_{k-m_k}-x^*\|_2 \leq \frac{\eta_0 L}{\mu}\|x_0-x^*\|_2$ is  sufficiently small. 
 Then, we prove \eqref{lemma:diff:r_k}, \eqref{lemma:diff:x_k}, and \eqref{lemma:diff:zeta_k}-\eqref{lemma:diff:bar_x_k} by induction.
 
 For $k=0$, the relations \eqref{lemma:diff:zeta_k}-\eqref{lemma:diff:Gamma_k} clearly hold. Besides, due to \eqref{assum:jacobian_continue}, we have
  $\|r_0-\hat{r}_0\|_2 \leq\frac{1}{2}\hat{\kappa}\|x_0-x^*\|_2^2$, namely \eqref{lemma:diff:r_k}. Since $x_0 = \hat{x}_0$, the \eqref{lemma:diff:bar_x_k} also holds. Then \eqref{lemma:diff:x_k} follows from 
 \begin{align*}
  \|x_1-\hat{x}_1\|_2 = \|x_0+\beta_0 r_0 - (\hat{x}_0+\beta_0 \hat{r}_0)\|_2
    = \beta_0 \|r_0-\hat{r}_0\|_2 
    \leq \frac{\beta_0\hat{\kappa}}{2}\|x_0-x^*\|_2^2. 
 \end{align*}
  
 Suppose that $k\geq 1$, and as an inductive hypothesis, the relations  \eqref{lemma:diff:r_k}, \eqref{lemma:diff:x_k}, and \eqref{lemma:diff:zeta_k}-\eqref{lemma:diff:bar_x_k} hold for $i=0,\dots,k-1.$ Consider the $k$-th  iteration. 
  
 If $m_k=0$, i.e., a restarting condition is met at the beginning of the $k$-th iteration, then $\hat{x}_k = x_k$. The same as the case that $k=0$, \eqref{lemma:diff:r_k}, \eqref{lemma:diff:x_k}, and \eqref{lemma:diff:zeta_k}-\eqref{lemma:diff:bar_x_k} hold. 
 
 Consider the nontrivial case that $m_k>0$.  Due to \eqref{ineq:cond3}, we have  
 \begin{equation*}
  \|x_j-x^*\|_2\leq \frac{1}{\mu}\|r_j\|_2 \leq \frac{\eta}{\mu}\|r_{k-m_k}\|_2
    \leq \frac{\eta L}{\mu}\|x_{k-m_k}-x^*\|_2,  ~~ j=k-m_k+1,\dots,k.   
 \end{equation*}
 Therefore, 
 \begin{align}
 \|x_j-x^*\|_2 = \mathcal{O}(\|x_{k-m_k}-x^*\|_2),  ~~ j=k-m_k,\dots,k.   \label{ineq:bound:relative_x_j} 
 \end{align}
 

 Since $x_k\in\mathcal{B}_{\hat{\rho}}(x^*)$, it follows that 
 \begin{align}
  \|r_k - \hat{r}_k\|_2 &= \|h(x_k)-\hat{h}(\hat{x}_k)\|_2
  	\leq \|h(x_k)-\hat{h}(x_k)\|_2 + \|\hat{h}(x_k)-\hat{h}(\hat{x}_k)\|_2   \nonumber \\
  	& = \|h(x_k)-h'(x^*)(x_k-x^*)\|_2+ \|h'(x^*)(x_k-\hat{x}_k)\|_2   \nonumber \\
  	& \leq \frac{1}{2}\hat{\kappa}\|x_k-x^*\|_2^2+L\|x_k-\hat{x}_k\|_2
  	  = \hat{\kappa}\mathcal{O}(\|x_{k-m_k}-x^*\|_2^2),   \label{proof:lemma:diff:r_k}
 \end{align}
 where the second inequality is due to \eqref{assum:jacobian_continue} and  \eqref{assum:jacobian}, and the last equality is due to \eqref{ineq:bound:relative_x_j} and the inductive hypothesis  \eqref{lemma:diff:x_k}. 
 Thus, the relation \eqref{lemma:diff:r_k} holds. 
 
 Since the condition \eqref{ineq:cond2} holds, we have
   $|v_k^\mathrm{T}q_k| \geq \tau |v_{k-m_k+1}^\mathrm{T}q_{k-m_k+1}| $. We discuss the Type-I method and the Type-II method separately. Using the fact that the $(k-m_k)$-th iteration is $x_{k-m_k+1} = x_{k-m_k}+\beta_{k-m_k}r_{k-m_k}$, 
  we have that 
 \begin{align*}
  |v_k^\mathrm{T}q_k| &\geq \tau |p_{k-m_k+1}^\mathrm{T}q_{k-m_k+1}|
         = \tau |\Delta x_{k-m_k}^\mathrm{T}\Delta r_{k-m_k}|       \nonumber \\
        & = \tau \left|\Delta x_{k-m_k} ^\mathrm{T} \int_0^1 h'(x_{k-m_k}+t\Delta x_{k-m_k})\Delta x_{k-m_k}dt\right|    \nonumber \\
        & \geq \tau\mu \|\Delta x_{k-m_k}\|_2^2
         = \tau\mu \beta_{k-m_k}^2\|r_{k-m_k}\|_2^2  
         \geq \tau \mu^3 \beta_{k-m_k}^2 \|x_{k-m_k}-x^*\|_2^2,  
 \end{align*}
 for the Type-I method, where the second inequality is due to \eqref{assum:hermite_part} and the third inequality is due to \eqref{ineq:r_k}. For the Type-II method, 
 \begin{align*}
  |v_k^\mathrm{T}q_k| &\geq \tau |q_{k-m_k+1}^\mathrm{T} q_{k-m_k+1}|
         = \tau \| \Delta r_{k-m_k} \|_2^2 \geq \tau \mu^2 \|\Delta x_{k-m_k}\|_2^2   \nonumber \\
    & = \tau\mu^2\beta_{k-m_k}^2\|r_{k-m_k}\|_2^2 
     \geq \tau\mu^4\beta_{k-m_k}^2\|x_{k-m_k}-x^*\|_2^2. 
 \end{align*}
 Then, define $\underline{\kappa} = \tau\mu^3\beta^2$ for the Type-I method, and $\underline{\kappa} = \tau\mu^4\beta^2$ for the Type-II method. Since no restart has occurred in the last $m_k$ iterations, we have
 \begin{equation}
  |v_i^\mathrm{T}q_i|_2 \geq \underline{\kappa}\|x_{k-m_k}-x^*\|_2^2, ~ \mbox{for} ~ i=k-m_k+1,\dots,k.   \label{ineq:vq_lower}
 \end{equation}
 
 Now, we prove \eqref{lemma:diff:zeta_k}. 
 We shall prove an auxiliary relation:
 \begin{align}
  \|q_k^j\|_2 = \mathcal{O}(\|x_{k-m_k}-x^*\|_2), \qquad 
  \|q_k^j-\hat{q}_k^j\|_2 = \hat{\kappa}\mathcal{O}(\|x_{k-m_k}-x^*\|_2^2),    \label{lemma:diff:q_k_j}
 \end{align}
 for $j=0,\dots,m_k-1$. We conduct the proof by induction. 
 
 For $j=0$, \eqref{lemma:diff:zeta_k} holds due to $\zeta_k^{(0)} = \hat{\zeta}_k^{(0)} = 0$. Since $q_k^0=\Delta r_{k-1}, \hat{q}_k^{0}=\Delta \hat{r}_{k-1}$, it follows that
 \begin{align*}
  \|q_k^0\|_2 \leq \|r_k\|_2+\|r_{k-1}\|_2 \leq 2\eta\|r_{k-m_k}\|_2
     =\mathcal{O}(\|x_{k-m_k}-x^*\|_2),
 \end{align*}
 which is due to \eqref{ineq:cond3} and \eqref{ineq:r_k}. Also, from \eqref{proof:lemma:diff:r_k} and \eqref{lemma:diff:r_k}, we have 
 \begin{align*}
  \|q_k^0-\hat{q}_k^0\|_2 \leq \|r_k-\hat{r}_k\|_2+\|r_{k-1}-\hat{r}_{k-1}\|_2
     =\hat{\kappa}\mathcal{O}(\|x_{k-m_k}-x^*\|_2^2).          
 \end{align*}
 Hence, the \eqref{lemma:diff:zeta_k} and \eqref{lemma:diff:q_k_j} hold when $j=0$. 
 
 Suppose that $j\geq 1$, and \eqref{lemma:diff:zeta_k} and \eqref{lemma:diff:q_k_j} hold for $\ell =0,\dots,j-1$. Consider the $j$-th step  in \eqref{eq:q_k_type2}. Due to  \eqref{ineq:vq_lower} and the inductive hypotheses \eqref{lemma:p_k} and  \eqref{lemma:diff:q_k_j}, we obtain
 \begin{align}
  |\zeta_k^{(j)}| 
  \leq \frac{\|v_{k-m_k+j}\|_2\|q_k^{j-1}\|_2}{\underline{\kappa}\|x_{k-m_k}-x^*\|_2^2} = \frac{\mathcal{O}(\|x_{k-m_k}-x^*\|_2^2)}{\underline{\kappa}\|x_{k-m_k}-x^*\|_2^2} = \mathcal{O}(1).    \label{ineq:zeta_k_j}
\end{align}  
 Next, if $v_{k-m_k+j}^\mathrm{T}q_k^{j-1}\neq 0$, then 
 \begin{align}
   & \quad |\zeta_k^{(j)}-\hat{\zeta}_k^{(j)}|    \nonumber \\
   &= |\zeta_k^{(j)}|\cdot \left| 1- \frac{\hat{\zeta}_k^{(j)}}{\zeta_k^{(j)}}\right| 
   = |\zeta_k^{(j)}|\cdot \left| 1- \frac{\hat{v}_{k-m_k+j}^\mathrm{T}\hat{q}_k^{j-1}}{v_{k-m_k+j}^\mathrm{T}q_k^{j-1}}\cdot \frac{v_{k-m_k+j}^\mathrm{T}q_{k-m_k+j}}{\hat{v}_{k-m_k+j}^\mathrm{T}\hat{q}_{k-m_k+j}} \right|   \nonumber \\
   &= |\zeta_k^{(j)}|\cdot |a(1-b)+b| \leq |\zeta_k^{(j)}|\cdot (|a|+|b|+|ab|),   \label{ineq:diff:zeta_k} 
 \end{align}
 where $a:=1-\frac{\hat{v}_{k-m_k+j}^\mathrm{T}\hat{q}_k^{j-1}}{v_{k-m_k+j}^\mathrm{T}q_k^{j-1}}$ and $ b:=1-\frac{v_{k-m_k+j}^\mathrm{T}q_{k-m_k+j}}{\hat{v}_{k-m_k+j}^\mathrm{T}\hat{q}_{k-m_k+j}}$. We have
\begin{align}
  |\zeta_k^{(j)}|\cdot |a| =  
  \left| \frac{v_{k-m_k+j}^\mathrm{T}q_k^{j-1}-\hat{v}_{k-m_k+j}^\mathrm{T}\hat{q}_k^{j-1}}{v_{k-m_k+j}^\mathrm{T}q_{k-m_k+j}} \right|.  \label{eq:zeta_k_a}
 \end{align}
 From \eqref{lemma:p_k}, \eqref{lemma:diff:p_k}, and \eqref{lemma:diff:q_k_j}, we obtain
\begin{equation*}
   |v_{k-m_k+j}^\mathrm{T}(q_k^{j-1}-\hat{q}_k^{j-1})| \leq \|v_{k-m_k+j}\|_2 \|q_k^{j-1}-\hat{q}_k^{j-1}\|_2 = \hat{\kappa}\mathcal{O}(\|x_{k-m_k}-x^*\|_2^3), 
  \end{equation*}
 and 
 \begin{align*}
  & \quad |(v_{k-m_k+j}-\hat{v}_{k-m_k+j})^\mathrm{T}\hat{q}_k^{j-1}|   \nonumber \\
  & \leq |(v_{k-m_k+j}-\hat{v}_{k-m_k+j})^\mathrm{T}q_k^{j-1}|
           + |(v_{k-m_k+j}-\hat{v}_{k-m_k+j})^\mathrm{T}(q_k^{j-1}-\hat{q}_k^{j-1})|   \nonumber \\
  & \leq \hat{\kappa}\mathcal{O}(\|x_{k-m_k}-x^*\|_2^3) + \hat{\kappa}^2\mathcal{O}(\|x_{k-m_k}-x^*\|_2^4) = \hat{\kappa}\mathcal{O}(\|x_{k-m_k}-x^*\|_2^3).
 \end{align*}
 Then, it follows that
 \begin{align}
   &| v_{k-m_k+j}^\mathrm{T}q_k^{j-1}-\hat{v}_{k-m_k+j}^\mathrm{T}\hat{q}_k^{j-1} |  
    \leq |v_{k-m_k+j}^\mathrm{T}(q_k^{j-1}-\hat{q}_k^{j-1})|   \nonumber \\
    & \qquad \qquad\qquad +|(v_{k-m_k+j}-\hat{v}_{k-m_k+j})^\mathrm{T}\hat{q}_k^{j-1}|
  = \hat{\kappa}\mathcal{O}(\|x_{k-m_k}-x^*\|_2^3).   \label{ineq:diff:v_q_1}
 \end{align}
 Combining \eqref{eq:zeta_k_a}, \eqref{ineq:diff:v_q_1}, and \eqref{ineq:vq_lower} yields 
 \begin{equation}
   |\zeta_k^{(j)}|\cdot |a| \leq \frac{\hat{\kappa}\mathcal{O}(\|x_{k-m_k}-x^*\|_2^3)}{\underline{\kappa}\|x_{k-m_k}-x^*\|_2^2} = \hat{\kappa}\mathcal{O}(\|x_{k-m_k}-x^*\|_2).    \label{ineq:zeta_k_a}
 \end{equation}
 Similar to \eqref{ineq:diff:v_q_1}, the following bound holds:
 \begin{align}
  |v_{k-m_k+j}^\mathrm{T}q_{k-m_k+j}-\hat{v}_{k-m_k+j}^\mathrm{T}\hat{q}_{k-m_k+j}| = \hat{\kappa}\mathcal{O}(\|x_{k-m_k}-x^*\|_2^3).   \label{ineq:diff:v_q_2}
 \end{align}
 Besides, 
 \begin{align}
  & \quad |\hat{v}_{k-m_k+j}^\mathrm{T}\hat{q}_{k-m_k+j}|   \nonumber \\
  &\geq |v_{k-m_k+j}^\mathrm{T}q_{k-m_k+j}|-
  |v_{k-m_k+j}^\mathrm{T}q_{k-m_k+j}-\hat{v}_{k-m_k+j}^\mathrm{T}\hat{q}_{k-m_k+j}|    \nonumber \\
  & \geq \underline{\kappa}\|x_{k-m_k}-x^*\|_2^2-\hat{\kappa}c_1\|x_{k-m_k}-x^*\|_2^3 \geq \frac{1}{2}\underline{\kappa}\|x_{k-m_k}-x^*\|_2^2,    \label{ineq:vq_lower_2}
 \end{align}
 where the existence of $c_1$ is guaranteed by \eqref{ineq:diff:v_q_2}, and the last inequality holds if $\|x_{k-m_k}-x^*\|_2\leq \frac{\underline{\kappa}}{2\hat{\kappa}c_1}$, which can be obtained by choosing $\|x_0-x^*\|_2\leq \frac{\mu\underline{\kappa}}{2\hat{\kappa}\eta_0Lc_1}$ since 
 $\|x_{k-m_k}-x^*\|_2 \leq \frac{\eta_0L}{\mu}\|x_0-x^*\|_2 $ by \eqref{ineq:bound:x_k_r_k}. 
 From \eqref{ineq:diff:v_q_2} and \eqref{ineq:vq_lower_2}, it follows that 
 \begin{align}
  |b| = \left| \frac{\hat{v}_{k-m_k+j}^\mathrm{T}\hat{q}_{k-m_k+j}-v_{k-m_k+j}^\mathrm{T}q_{k-m_k+j}}{\hat{v}_{k-m_k+j}^\mathrm{T}\hat{q}_{k-m_k+j}} \right| = \hat{\kappa}\mathcal{O}(\|x_{k-m_k}-x^*\|_2).   \label{ineq:b}
 \end{align}
 As a result, by \eqref{ineq:zeta_k_a}, \eqref{ineq:b}, \eqref{ineq:zeta_k_j}, and \eqref{ineq:diff:zeta_k}, we obtain
 \begin{equation*}
   |\zeta_k^{(j)}-\hat{\zeta}_k^{(j)}| = \hat{\kappa}\mathcal{O}(\|x_{k-m_k}-x^*\|_2). 
 \end{equation*}
 Now consider the case that $v_{k-m_k+j}^\mathrm{T}q_k^{j-1} = 0$. It is clear that $\zeta_k^{(j)} = 0$. Then
 \begin{align*}
  |\zeta_k^{(j)}-\hat{\zeta}_k^{(j)}|= \left|\frac{\hat{v}_{k-m_k+j}^\mathrm{T}\hat{q}_k^{j-1}}{\hat{v}_{k-m_k+j}^\mathrm{T}\hat{q}_{k-m_k+j}}\right| \leq \frac{\hat{\kappa}\mathcal{O}(\|x_{k-m_k}-x^*\|_2^3)}{\frac{1}{2}\underline{\kappa}\|x_{k-m_k}-x^*\|_2^2} = \hat{\kappa}\mathcal{O}(\|x_{k-m_k}-x^*\|_2). 
 \end{align*}
 Therefore, \eqref{lemma:diff:zeta_k} holds for $\ell = j$. Next, we obtain
 \begin{align}
  \|q_k^j\|_2 \leq \|q_k^{j-1}\|_2+\|q_{k-m_k+j}\|_2 |\zeta_k^{(j)}| = \mathcal{O}(\|x_{k-m_k}-x^*\|_2), 
 \end{align}
 which is due to \eqref{lemma:diff:q_k_j}, \eqref{lemma:p_k},  \eqref{ineq:zeta_k_j}, and $j<m_k\leq m$.  Also, from \eqref{lemma:diff:p_k}, \eqref{lemma:diff:zeta_k}, \eqref{lemma:p_k}, and $j\leq m_k-1$, it follows that
 \begin{align}
  & \|q_{k-m_k+j}\zeta_k^{(j)}-\hat{q}_{k-m_k+j}\hat{\zeta}_k^{(j)}\|_2
  \leq \|(q_{k-m_k+j}-\hat{q}_{k-m_k+j})\zeta_k^{(j)}\|_2    \nonumber \\
      & \qquad + \|(\hat{q}_{k-m_k+j}-q_{k-m_k+j})(\zeta_k^{(j)}-\hat{\zeta}_k^{(j)})\|_2 
      + \|q_{k-m_k+j}(\zeta_k^{(j)}-\hat{\zeta}_k^{(j)})\|_2    \nonumber \\
      & \qquad = \hat{\kappa}\mathcal{O}(\|x_{k-m_k}-x^*\|_2^2),   \label{ineq:diff:q_zeta}
 \end{align}
 which together with \eqref{eq:q_k_type2} further yields that
 \begin{equation*}
  \|q_k^j-\hat{q}_k^j\|_2 \leq \|q_k^{j-1}-\hat{q}_k^{j-1}\|_2
               +\|q_{k-m_k+j}\zeta_k^{(j)}-\hat{q}_{k-m_k+j}\hat{\zeta}_k^{(j)}\|_2 = \hat{\kappa}\mathcal{O}(\|x_{k-m_k}-x^*\|_2^2). 
 \end{equation*}
 Then, \eqref{lemma:diff:q_k_j} holds for $\ell = j$, thus completing the induction. 
 
 Since \eqref{lemma:diff:q_k_j} holds for $j=m_k-1$, and $q_k=q_k^{m_k-1}$, we know
 $\|q_k\|_2=\mathcal{O}(\|x_{k-m_k}-x^*\|_2)$ and $\|q_k-\hat{q}_k\|_2 = \hat{\kappa}\mathcal{O}(\|x_{k-m_k}-x^*\|_2^2)$. 
 If $m_k = 1$, then $p_k = \Delta x_{k-1}$. So $\|p_k\|_2 \leq \|x_k-x^*\|_2 + \|x_{k-1}-x^*\|_2 = \mathcal{O}(\|x_{k-m_k}-x^*\|_2)$ and 
 $\|p_k-\hat{p}_k\|_2 \leq \|x_k-\hat{x}_k\|_2 + \|x_{k-1}-\hat{x}_{k-1}\|_2 = \hat{\kappa}\mathcal{O}(\|x_{k-m_k}-x^*\|_2^2)$. 
 Consider $m_k\geq 2$. Since
 \begin{equation*}
  \|p_k\|_2 = \|\Delta x_{k-1}-P_{k-1}\zeta_k\|_2
   \leq \|x_k-x^*\|_2+\|x_{k-1}-x^*\|_2+ \sum_{j=1}^{m_k-1}\|p_{k-m_k+j} \zeta_k^{(j)}\|_2, 
 \end{equation*}
 it follows that $\|p_k\|_2 = \mathcal{O}(\|x_{k-m_k}-x^*\|_2)$. Also, similar to \eqref{ineq:diff:q_zeta}, we have 
 \begin{align*}
  \|p_{k-m_k+j}\zeta_k^{(j)}-\hat{p}_{k-m_k+j}\hat{\zeta}_k^{(j)}\|_2 = \hat{\kappa}\mathcal{O}(\|x_{k-m_k}-x^*\|_2^2), 
 \end{align*}
 which further yields 
 \begin{equation*}
  \|P_{k-1}\zeta_k - \hat{P}_{k-1}\hat{\zeta}_k\|_2
  \leq \sum_{j=1}^{m_k-1}\|p_{k-m_k+j}\zeta_k^{(j)}-\hat{p}_{k-m_k+j}\hat{\zeta}_k^{(j)}\|_2 = \hat{\kappa}\mathcal{O}(\|x_{k-m_k}-x^*\|_2^2). 
 \end{equation*}
 Then, with
 $
  \|p_k-\hat{p}_k\|_2 \leq \|x_k-\hat{x}_k\|_2+\|x_{k-1}-\hat{x}_{k-1}\|_2+
  		\|P_{k-1}\zeta_k - \hat{P}_{k-1}\hat{\zeta}_k\|_2  		
 $, we obtain $\|p_k-\hat{p}_k\|_2 = \hat{\kappa}\mathcal{O}(\|x_{k-m_k}-x^*\|_2^2)$. Hence, \eqref{lemma:p_k} and \eqref{lemma:diff:p_k} hold. 
 
  Now, we prove \eqref{lemma:diff:Gamma_k}, following a similar way of proving \eqref{lemma:diff:zeta_k}. The concerned auxiliary relation is
  \begin{align}
   \|r_k^j\|_2 = \mathcal{O}(\|x_{k-m_k}-x^*\|_2), \quad \|r_k^j-\hat{r}_k^j\|_2 = \hat{\kappa}\mathcal{O}(\|x_{k-m_k}-x^*\|_2^2),  \label{lemma:diff:r_k_j}
  \end{align}
  for $j=0,\dots,m_k$. We still conduct the proof by induction.
  
  For $j=0$, \eqref{lemma:diff:Gamma_k} holds due to $\Gamma_k^{(0)} = \hat{\Gamma}_k^{(0)} = 0$. Since $r_k^0=r_k, \hat{r}_k^0=\hat{r}_k$, we have
  $\|r_k^0\|_2\leq \eta \|r_{k-m_k}\|_2 \leq \eta L\|x_{k-m_k}-x^*\|_2$, and $\|r_k^0-\hat{r}_k^0\|_2 = \hat{\kappa}\mathcal{O}(\|x_{k-m_k}-x^*\|_2^2)$. 
  
  Suppose that $j\geq 1$, and \eqref{lemma:diff:Gamma_k} and \eqref{lemma:diff:r_k_j} hold for $\ell = 0,\dots,j-1$. Consider the $j$-th step  in \eqref{eq:r_k_type2}. With \eqref{ineq:vq_lower}, we have
  \begin{equation}
   |\Gamma_k^{(j)}|\leq \frac{\|v_{k-m_k+j}\|_2\|r_k^{j-1}\|_2}{\underline{\kappa}\|x_{k-m_k}-x^*\|_2^2} = \frac{\mathcal{O}(\|x_{k-m_k}-x^*\|_2^2)}{\underline{\kappa}\|x_{k-m_k}-x^*\|_2^2} = \mathcal{O}(1).   \label{ineq:Gamma_k_j}
  \end{equation}
  Next, if $v_{k-m_k+j}^\mathrm{T}r_k^{j-1} \neq 0$, then
  \begin{align}
   |\Gamma_k^{(j)}-\hat{\Gamma}_k^{(j)}| = |\Gamma_k^{(j)}|\cdot |a_1(1-b_1)+b_1|
      \leq |\Gamma_k^{(j)}|\cdot (|a_1|+|b_1|+|a_1b_1|),   \label{ineq:diff:Gamma_k}
  \end{align}
  where $a_1:=1-\frac{\hat{v}_{k-m_k+j}^\mathrm{T}\hat{r}_k^{j-1}}{v_{k-m_k+j}^\mathrm{T}r_k^{j-1}}$ and $b_1:=1-\frac{v_{k-m_k+j}^\mathrm{T}q_{k-m_k+j}}{\hat{v}_{k-m_k+j}^\mathrm{T}\hat{q}_{k-m_k+j}}$. With \eqref{lemma:diff:r_k_j}, \eqref{lemma:diff:p_k}, and \eqref{lemma:p_k}, it follows that
  \begin{align}
   &|v_{k-m_k+j}^\mathrm{T}r_k^{j-1}-\hat{v}_{k-m_k+j}^\mathrm{T}\hat{r}_k^{j-1}|
    \leq |v_{k-m_k+j}^\mathrm{T}(r_k^{j-1}-\hat{r}_k^{j-1})|    \nonumber \\
   & ~~ + |(v_{k-m_k+j}-\hat{v}_{k-m_k+j})^\mathrm{T}r_k^{j-1}|
     + |(v_{k-m_k+j}-\hat{v}_{k-m_k+j})^\mathrm{T}(\hat{r}_k^{j-1}-r_k^{j-1})|  \nonumber \\
   & = \hat{\kappa}\mathcal{O}(\|x_{k-m_k}-x^*\|_2^3). \label{ineq:diff:v_r}
  \end{align}
  Then with \eqref{ineq:diff:v_r} and \eqref{ineq:vq_lower}, we obtain
  \begin{align}
   |\Gamma_k^{(j)}|\cdot |a_1| = \left|\frac{v_{k-m_k+j}^\mathrm{T}r_k^{j-1}-\hat{v}_{k-m_k+j}^\mathrm{T}\hat{r}_k^{j-1}}{v_{k-m_k+j}^\mathrm{T}q_{k-m_k+j}}\right| \leq \hat{\kappa}\mathcal{O}(\|x_{k-m_k}-x^*\|_2).  \label{ineq:Gamma_k_a}
  \end{align}
  For the bound of $|b_1|$, note that we have obtained \eqref{ineq:b} and also have already proved \eqref{lemma:p_k} and \eqref{lemma:diff:p_k} for the $k$-th iteration. Thus, $|b_1| = \hat{\kappa}\mathcal{O}(\|x_{k-m_k}-x^*\|_2)$, which together with \eqref{ineq:Gamma_k_a}, \eqref{ineq:Gamma_k_j}, and \eqref{ineq:diff:Gamma_k} yields 
  $|\Gamma_k^{(j)}-\hat{\Gamma}_k^{(j)}| = \hat{\kappa}\mathcal{O}(\|x_{k-m_k}-x^*\|_2)$. 
  On the other side, 
  if $v_{k-m_k+j}^\mathrm{T}r_k^{j-1} = 0$, then $\Gamma_k^{(j)} = 0$. Hence
  \begin{align*}
   |\Gamma_k^{(j)}-\hat{\Gamma}_k^{(j)}|= \left| \frac{\hat{v}_{k-m_k+j}^\mathrm{T}\hat{r}_k^{j-1}}{\hat{v}_{k-m_k+j}^\mathrm{T}\hat{q}_{k-m_k+j}} \right| \leq \frac{\hat{\kappa}\mathcal{O}(\|x_{k-m_k}-x^*\|_2^3)}{\frac{1}{2}\underline{\kappa}\|x_{k-m_k}-x^*\|_2^2} = \hat{\kappa}\mathcal{O}(\|x_{k-m_k}-x^*\|_2). 
  \end{align*}
 Therefore \eqref{lemma:diff:Gamma_k} holds for $\ell = j$. Next, we obtain 
 \begin{align*}
  \|r_k^j\|_2 \leq \|r_k^{j-1}\|_2+ \|q_{k-m_k+j}\|_2 |\Gamma_k^{(j)}| = \mathcal{O}(\|x_{k-m_k}-x^*\|_2) 
 \end{align*}   
  due to \eqref{lemma:diff:r_k_j}, \eqref{lemma:p_k}, \eqref{ineq:Gamma_k_j}, and $j\leq m_k\leq m$. By \eqref{lemma:diff:p_k}, \eqref{lemma:diff:Gamma_k}, and \eqref{lemma:p_k}, we have 
  \begin{align}
   & \|q_{k-m_k+j}\Gamma_k^{(j)}-\hat{q}_{k-m_k+j}\hat{\Gamma}_k^{(j)}\|_2 
       \leq \|(q_{k-m_k+j}-\hat{q}_{k-m_k+j})\Gamma_k^{(j)}\|_2    \nonumber \\
    & \quad  +\|(\hat{q}_{k-m_k+j}-q_{k-m_k+j})(\Gamma_k^{(j)}-\hat{\Gamma}_k^{(j)})\|_2 
       + \|q_{k-m_k+j}(\Gamma_k^{(j)}-\hat{\Gamma}_k^{(j)})\|_2   \nonumber \\
    & \quad = \hat{\kappa}\mathcal{O}(\|x_{k-m_k}-x^*\|_2^2),  \label{ineq:diff:q_Gamma}
  \end{align}
  which yields that
  \begin{equation*}
   \|r_k^j-\hat{r}_k^j\|_2 \leq \|r_k^{j-1}-\hat{r}_k^{j-1}\|_2
   	   + \|q_{k-m_k+j}\Gamma_k^{(j)}-\hat{q}_{k-m_k+j}\hat{\Gamma}_k^{(j)}\|_2
   	  = \hat{\kappa}\mathcal{O}(\|x_{k-m_k}-x^*\|_2^2). 
  \end{equation*}
  Then, \eqref{lemma:diff:r_k_j} holds for $\ell = j$, thus completing the induction. 
  
  Since \eqref{lemma:diff:r_k_j} holds for $j=m_k$, and $\bar{r}_k=r_k^{m_k}$, we obtain $\|\bar{r}_k\|_2 = \mathcal{O}(\|x_{k-m_k}-x^*\|_2)$ and $\|\bar{r}_k-\bar{\hat{r}}_k\|_2 = \hat{\kappa}\mathcal{O}(\|x_{k-m_k}-x^*\|_2^2)$. Moreover, similar to \eqref{ineq:diff:q_Gamma}, we have 
  \begin{equation*}
   \|p_{k-m_k+j}\Gamma_k^{(j)}-\hat{p}_{k-m_k+j}\hat{\Gamma}_k^{(j)}\|_2
     =\hat{\kappa}\mathcal{O}(\|x_{k-m_k}-x^*\|_2^2), 
  \end{equation*}
  which further yields 
 \begin{equation*}
  \|P_k\Gamma_k-\hat{P}_k\hat{\Gamma}_k\|_2 
   \leq \sum_{j=1}^{m_k}\|p_{k-m_k+j}\Gamma_k^{(j)}-\hat{p}_{k-m_k+j}\hat{\Gamma}_k^{(j)}\|_2 = \hat{\kappa}\mathcal{O}(\|x_{k-m_k}-x^*\|_2^2). 
 \end{equation*}
 Then, from $\|\bar{x}_k-\bar{\hat{x}}_k\|_2\leq \|x_k-\hat{x}_k\|_2 + \|P_k\Gamma_k-\hat{P}_k\hat{\Gamma}_k\|_2$, we obtain
 $\|\bar{x}_k-\bar{\hat{x}}_k\|_2 = \hat{\kappa}\mathcal{O}(\|x_{k-m_k}-x^*\|_2^2)$. Hence, \eqref{lemma:diff:bar_x_k} holds. 

 Finally,  since $x_{k+1} = \bar{x}_{k}+\beta_{k}\bar{r}_{k}$, it follows that
 \begin{equation*}
  \|x_{k+1}-\hat{x}_{k+1}\|_2 = \|(\bar{x}_{k}-\bar{\hat{x}}_{k})+\beta_{k}(\bar{r}_{k}-\bar{\hat{r}}_{k})\|_2 = \hat{\kappa}\mathcal{O}(\|x_{k-m_k}-x^*\|_2^2),
 \end{equation*}
 where the second equality is due to \eqref{lemma:diff:bar_x_k} and the fact that $\beta_{k}$ is bounded.  
 
 As a result, we complete the induction. Thus, \eqref{lemma:diff:r_k} and \eqref{lemma:diff:x_k} are proved. 
\end{proof}

\subsection{Proof of Theorem~\ref{them:am}}  \label{subsec:proof:them:am}
Let $\lambda_{\min}(\cdot)$ and $\lambda_{\max}(\cdot)$ denote the smallest eigenvalue and the largest eigenvalue of a real symmetric matrix. 
We first give a lemma.
\begin{lemma}   \label{lemma:theta_k}
 Suppose that $A\in\mathbb{R}^{d\times d}$ is positive definite with $\lambda_{\min}(\mathcal{S}(A))\geq \mu$ and $\|A\|_2\leq L$, where $\mu, L > 0$.  Then for a constant $\theta\in\bigl[\bigl(1-\frac{\mu^2}{L^2}\bigr)^{1/2},1\bigr)$, there exist positive constants $\beta,\beta'$ such that when $\beta_k\in[\beta,\beta']$, the inequality $\|I-\beta_kA\|_2 \leq \theta$ holds. If $\theta = \bigl(1-\frac{\mu^2}{L^2}\bigr)^{1/2}$, then $\|I-\beta_kA\|_2 \leq \theta$ when $\beta_k = \mu/L^2$. 
\end{lemma}
\begin{proof}
 Since 
 $ (I-\beta_kA)^\mathrm{T}(I-\beta_kA) = I-\beta_k(A+A^\mathrm{T})+\beta_k^2 A^\mathrm{T}A$, 
 it follows from Weyl's inequalities \cite[Theorem~III.2.1]{bhatia2013matrix} that 
 \begin{align}
  \|I-\beta_kA\|_2^2 &\leq \lambda_{\max}\bigl(I-\beta_k(A+A^\mathrm{T})\bigr)+ \lambda_{\max}(\beta_k^2A^\mathrm{T}A)    \nonumber \\
  &\leq 1-\beta_k\lambda_{\min}(A+A^\mathrm{T})+\beta_k^2\|A\|_2^2   \nonumber \\
  &\leq 1-2\beta_k\mu+\beta_k^2L^2. 
 \end{align}
 Thus, 
 to ensure $\|I-\beta_kA\|_2\leq \theta$, it suffices to require that 
 \begin{align}
  1-2\beta_k\mu+\beta_k^2L^2\leq \theta^2.   \label{ineq:theta}
 \end{align}
 Since $\theta\in\bigl[\bigl(1-\frac{\mu^2}{L^2}\bigr)^{1/2},1\bigr)$, solving \eqref{ineq:theta} yields that $\beta_k\in [\beta,\beta']$, where
 \begin{equation}
  \beta = \frac{\mu-\bigl( \mu^2-L^2(1-\theta^2) \bigr)^{1/2}}{L^2}, \quad
  \beta' = \frac{\mu+\bigl( \mu^2-L^2(1-\theta^2) \bigr)^{1/2}}{L^2}.
 \end{equation}
 If $\theta = \bigl(1-\frac{\mu^2}{L^2}\bigr)^{1/2}$, then $\beta_k = \mu/L^2$. 
\end{proof}

 Now, we give the proof of Theorem~\ref{them:am}. 
\begin{proof}
 Let the notations be the same as those in the proof of \Cref{lemma:diff}. 
 
  1. For the Type-I method, let $x_k^\mathrm{A}$ and $r_k^\mathrm{A}$ denote the $m_k$-th iterate and residual of Arnoldi's method applied to solve $\hat{h}(x) = 0$, with the starting point $x_{k-m_k}$. 
  Due to \cref{prop:linear}, we have 
 $ \bar{\hat{x}}_k = x_k^\mathrm{A}$. Then, according to the known convergence of Arnoldi's method \cite[Corollary 2.1 and Proposition 4.1]{saad1981krylov}, 
 \begin{equation}
  \|\bar{\hat{x}}_k-x^*\|_2 =  \|x_k^\mathrm{A}-x^*\|_2 \leq \sqrt{1+\gamma_k^2\kappa_k^2}\min_{\mathop{}_{p(0)=1}^{p\in \mathcal{P}_{m_k}}}\|p(A)(x_{k-m_k}-x^*)\|_2.   \label{them:x_k:arnoldi}
 \end{equation}
 Since $\hat{x}_{k+1} = \bar{\hat{x}}_k+\beta_k\bar{\hat{r}}_k$, it follows that
 $\hat{x}_{k+1}-x^* = (I-\beta_kA)(\bar{\hat{x}}_k-x^*)$. Hence $\|\hat{x}_{k+1}-x^*\|_2 \leq \theta_k \|\bar{\hat{x}}_k-x^*\|_2$, which along with \eqref{them:x_k:arnoldi} and \cref{lemma:diff} yields \eqref{them:type1}. 
 
 Let $\sigma_{\min}(\cdot)$ denote the smallest singular value. Choose  $U_k\in\mathbb{R}^{d\times m_k}$ that satisfies ${\rm range}(U_k) = \mathcal{K}_{m_k}(A,r_{k-m_k})$ and $U_k^\mathrm{T}U_k = I$. Then $\pi_k = U_kU_k^\mathrm{T}$. 
 Since $\pi_k$ and $I-\pi_k$ are orthogonal projectors, it follows that $\gamma_k = \|\pi_kA(I-\pi_k)\|_2 \leq \|A\|_2 \leq L$. 
 For the restriction $A_k\vert_{\mathcal{K}_{m_k}(A,r_{k-m_k})}$, we have
 \begin{align*}
  & \quad \sigma_{\min}(A_k\vert_{\mathcal{K}_{m_k}(A,r_{k-m_k})})  \\
  & = \min_{\mathop{}_{\|y\|_2=1}^{y\in\mathbb{R}^{m_k}}}\|A_kU_ky\|_2
   = \min_{\mathop{}_{\|y\|_2=1}^{y\in\mathbb{R}^{m_k}}}\|U_kU_k^\mathrm{T}AU_ky\|_2 = \sigma_{\min}(U_k^\mathrm{T}AU_k). 
 \end{align*}
  Since $\sigma_{\min}(U_k^\mathrm{T}AU_k) \geq \lambda_{\min}(\mathcal{S}(U_k^\mathrm{T}AU_k)) = \lambda_{\min}(U_k^\mathrm{T}\mathcal{S}(A)U_k)\geq \lambda_{\min}(\mathcal{S}(A))\geq \mu$, where the first inequality is due to Fan-Hoffman theorem \cite[Proposition III.5.1]{bhatia2013matrix}, and the second inequality is due to \cite[Corollary~III.1.5]{bhatia2013matrix}, it follows that
 $\kappa_k = \|(A_k\vert_{\mathcal{K}_{m_k}(A,r_{k-m_k})})^{-1}\|_2 = \frac{1}{\sigma_{\min}(A_k\vert_{\mathcal{K}_{m_k}(A,r_{k-m_k})})}\leq 1/\mu$. 
 
 2. For the Type-II method, let $x_k^\mathrm{G}$ and $r_k^\mathrm{G}$ denote the $m_k$-th iterate and residual of GMRES applied to solve $\hat{h}(x) = 0$, with the starting point $x_{k-m_k}$.  We have $\bar{\hat{x}}_k = x_k^\mathrm{G}$ due to \cref{prop:linear}. It follows from the property of GMRES \cite{saad1986gmres} that
 \begin{equation}
  \|\bar{\hat{r}}_k\|_2 = \|r_k^\mathrm{G}\|_2 = \min_{\mathop{}_{p(0)=1}^{p\in \mathcal{P}_{m_k}}}\|p(A)\hat{r}_{k-m_k}\|_2
	\leq \min_{\mathop{}_{p(0)=1}^{p\in \mathcal{P}_{m_k}}}\|p(A)r_{k-m_k}\|_2 + \hat{\kappa}\mathcal{O}(\|x_{k-m_k}-x^*\|_2^2), \label{them:x_k:gmres}
 \end{equation}
 where the inequality is due to \eqref{lemma:diff:r_k} and $\|p(A)\|_2 \leq 1$ when  $p(A) = (I-\frac{\mu}{L^2}A)^{m_k}$ (see \Cref{lemma:theta_k}). 
 Since $\hat{x}_{k+1} = \bar{\hat{x}}_k+\beta_k\bar{\hat{r}}_k$, it follows that
 $\hat{r}_{k+1} = (I-\beta_kA)\bar{\hat{r}}_k$. Hence $\|\hat{r}_{k+1}\|_2 \leq \theta_k \|\bar{\hat{r}}_k\|_2$, which along with \eqref{them:x_k:gmres},  \cref{lemma:diff}, and $\theta_k \leq 1+\beta' L$ yields \eqref{them:type2}. 
 
  If $\theta_j \leq \theta<1$ (ensured by \Cref{lemma:theta_k}) for $j= 0,\dots,\max\{k-1,0\}$, then 
  \begin{equation}
   \|r_j\|_2 \leq \|r_0\|_2, ~ j=0,\dots,k,   \label{them:type2:r_j}
  \end{equation}
    when $\|x_0-x^*\|_2$ is sufficiently small. We prove it by induction. For $k=0$, \eqref{them:type2:r_j} is clear. Suppose that $k\geq 0$, and as an inductive hypothesis, \eqref{them:type2:r_j} holds for $k$. We establish the result for $k+1$. 
   Since
  \begin{align*}
   \|A(\hat{x}_{k+1}-x^*)\|_2 = \|\hat{r}_{k+1}\|_2 
   		\leq \theta_k \min_{\mathop{}_{p(0)=1}^{p\in \mathcal{P}_{m_k}}}\|p(A)\hat{r}_{k-m_k}\|_2 \leq \theta_k \|A(x_{k-m_k}-x^*)\|_2,
  \end{align*}
  it follows that 
  \begin{align*}
   \|A(x_{k+1}-x^*)\|_2 &\leq \|A(\hat{x}_{k+1}-x^*)\|_2 + \|A(x_{k+1}-\hat{x}_{k+1})\|_2  \\
                    & \leq \theta\|A(x_{k-m_k}-x^*)\|_2+L\|x_{k+1}-\hat{x}_{k+1}\|_2.
  \end{align*}
  From \Cref{lemma:diff}, $\|x_{k+1}-\hat{x}_{k+1}\|_2 = \hat{\kappa}\mathcal{O}(\|x_{k-m_k}-x^*\|_2^2)$. So there exists a constant $c>0$ such that $\|x_{k+1}-\hat{x}_{k+1}\|_2 \leq \hat{\kappa}c \|A(x_{k-m_k}-x^*)\|_2^2 $.  Hence, if $x_{k-m_k}$ is chosen such that $\|A(x_{k-m_k}-x^*)\|_2 \leq \frac{1-\theta}{2L\hat{\kappa}c}$, it yields $ \|A(x_{k+1}-x^*)\|_2 \leq \frac{1+\theta}{2}\|A(x_{k-m_k}-x^*)\|_2$. Thus, $\|x_{k+1}-x^*\|_2\leq \frac{1+\theta}{2}\frac{L}{\mu}\|x_{k-m_k}-x^*\|_2$, which indicates $x_{k+1}\in\mathcal{B}_{\hat{\rho}}(x^*)$ 
  for  $\|x_{k-m_k}-x^*\|_2\leq \frac{2\mu\hat{\rho}}{L(1+\theta)}$. Then
   due to \eqref{assum:jacobian_continue}, we have  
  $\|r_{k+1}\|_2 \leq \|A(x_{k+1}-x^*)\|_2+\frac{1}{2}\hat{\kappa}\|x_{k+1}-x^*\|_2^2$. Therefore, 
  \begin{align*}
   \|r_{k+1}\|_2 &\leq \frac{1+\theta}{2} \|A(x_{k-m_k}-x^*)\|_2 + \frac{1}{2}\hat{\kappa}\|x_{k+1}-x^*\|_2^2    \\
      & \leq \frac{1+\theta}{2}\bigl(\|r_{k-m_k}\|_2+\frac{\hat{\kappa}}{2}\|x_{k-m_k}-x^*\|_2^2\bigr)+\frac{1}{2}\hat{\kappa}\left(\frac{1+\theta}{2}\frac{L}{\mu}\|x_{k-m_k}-x^*\|_2\right)^2 \\
      & \leq \theta'\|r_{k-m_k}\|_2 + \hat{\kappa} c' \|r_{k-m_k}\|_2^2,
  \end{align*}
 where $\theta':=\frac{1+\theta}{2}$, $c'>0$ is a constant, and the last inequality is due to \eqref{ineq:r_k}. So by choosing  $\|x_{k-m_k}-x^*\|_2\leq \frac{1-\theta'}{2\hat{\kappa}c'L}$, it follows that $\|r_{k-m_k}\|_2\leq L\|x_{k-m_k}-x^*\|_2 \leq \frac{1-\theta'}{2\hat{\kappa}c'}$. Then 
  $\|r_{k+1}\|_2 \leq \frac{1+\theta'}{2}\|r_{k-m_k}\|_2 < \|r_{k-m_k}\|_2 \leq \|r_0\|_2$. Since $\|x_{k-m_k}-x^*\|_2 \leq \frac{1}{\mu}\|r_{k-m_k}\|_2 \leq \frac{1}{\mu}\|r_0\|_2 \leq \frac{L}{\mu}\|x_0-x^*\|_2$, the requirement that $\|x_{k-m_k}-x^*\|_2 \leq \rho$  for some constant $\rho>0$ can be induced from $\|x_0-x^*\|_2\leq \frac{\mu\rho}{L}$. Hence, we complete the induction. 
  Then \eqref{them:type2} holds if $\theta_j\leq \theta$ $(j\geq 0)$ and $x_0$ is sufficiently close to $x^*$. 
 
 3. If $m_k = d$, then the Process~II obtains the exact solution of $\hat{h}(x) = 0$, i.e. $\hat{x}_{k+1} = x^*$. Therefore $\|x_{k+1}-x^*\|_2 = \hat{\kappa}\mathcal{O}(\|x_{k-m_k}-x^*\|_2^2)$. 
 \end{proof}

\subsection{Proof of Lemma~\ref{lemma:h}}   \label{subsec:proof:lemma:h}
\begin{proof}
  Since $h'(x) = I-g'(x)$, it follows that for every $x,y\in\mathcal{B}_{\hat{\rho}}(x^*)$, 
  \begin{equation*}
   \|h'(x)-h'(y)\|_2 = \|g'(x)-g'(y)\|_2 \leq \hat{\kappa}\|x-y\|_2,
  \end{equation*}
 which implies that $h(x)$ is Lipschitz continuously differentiable in $\mathcal{B}_{\hat{\rho}}(x^*)$ and the Lipschitz constant of $h'(x)$ is $\hat{\kappa}$. 
 
 Due to $\|I-h'(x)\|_2 = \|g'(x)\|_2 \leq \kappa < 1$, we have
 $\|h'(x)\|_2 \leq \|I\|_2+\|I-h'(x)\|_2 \leq 1+\kappa$, and
 \begin{equation*}
  \frac{1}{\sigma_{\min}(h'(x))}=\|h'(x)^{-1}\|_2 = \|\left(I-g'(x)\right)^{-1}\|_2 \leq \frac{1}{1-\|g'(x)\|_2} \leq \frac{1}{1-\kappa}, 
 \end{equation*}
 where $\sigma_{\min}(\cdot)$ denotes the smallest singular value. Thus, $\sigma_{\min}(h'(x)) \geq 1-\kappa$. The \eqref{assum:jacobian} holds for $\mu=1-\kappa$ and $L=1+\kappa$. 
 Note that 
 \begin{align}
  \left\Vert I-\mathcal{S}(h'(x)) \right\Vert_2
    \leq \frac{1}{2}(\Vert I-h'(x)\Vert_2+\Vert I-h'(x)^\mathrm{T}\Vert_2)
     = \kappa.   \label{ineq:symmetric_part}
 \end{align}
 Let $\lambda$ be an arbitrary eigenvalue of $\mathcal{S}(h'(x))$. Since $\mathcal{S}(h'(x))$ is symmetric, it follows from \eqref{ineq:symmetric_part} that $\vert 1-\lambda \vert \leq \kappa$, which yields $0<1-\kappa\leq \lambda \leq 1+\kappa$. Thus \eqref{assum:hermite_part} also holds for $\mu=1-\kappa$
 and $L=1+\kappa$. 
 
 Therefore,  Assumption~\ref{assum:h} is satisfied. 
 \end{proof}

\section{Proofs of Section~\ref{sec:mix_params}}

 \subsection{Proof of Proposition~\ref{prop:linear:hessenberg}}  \label{proof:prop:linear:hessenberg}
\begin{proof}
  Since $v_j^\mathrm{T}q_j \neq 0$ for $j=k-m_k+1,\dots,k$, the procedures \eqref{eq:q_k_type2} and \eqref{eq:r_k_type2} are well defined. 
  First, by construction, we have
  \begin{align}
   p_{k+1} &= \Delta x_{k} - P_k\zeta_{k+1} = -P_k\Gamma_k +\beta_k\bar{r}_k-P_k\zeta_{k+1} = \beta_k\bar{r}_k-P_k\phi_k  \nonumber \\
   &= \beta_k(r_k-Q_k\Gamma_k) - P_k\phi_k  \nonumber \\
   &= \beta_k(I-\beta_{k-1}A)\bar{r}_{k-1}-\beta_kQ_k\Gamma_k-P_k\phi_k   \nonumber \\
   &= \beta_k(I-\beta_{k-1}A)\frac{p_k+P_{k-1}\phi_{k-1}}{\beta_{k-1}}+\beta_k AP_k\Gamma_k-P_k\phi_k.    \label{eq:p_k_recursive}
  \end{align}
 Here, for brevity, we define $P_k = \mathbf{0}\in\mathbb{R}^d$, $\phi_k = 0$,  $\Gamma_{k+1}^{[0]} = 0$, $\zeta_{k+1} = 0$, if $m_k = 0$. Correspondingly, $\bar{x}_k = x_k$, $\bar{r}_k = r_k$, if $m_k = 0$.  We prove \eqref{eq:AP_k} by induction. 
 
 If $m_k=1$, it follows from \eqref{eq:p_k_recursive} that 
 \begin{align*}
  p_{k+1} = \beta_k(I-\beta_{k-1}A)\frac{p_k}{\beta_{k-1}}+\beta_kAp_k\Gamma_k-p_k\phi_k.
 \end{align*}
 It follows that $Ap_k = \frac{1}{1-\Gamma_k}\left(\frac{1}{\beta_{k-1}}-\frac{1}{\beta_k}\phi_k\right)p_k-\frac{1}{(1-\Gamma_k)\beta_k}p_{k+1}$, namely  \eqref{eq:AP_k}. 
 
 For $m_k\geq 2$, the inductive hypothesis is $AP_{k-1}=P_k\bar{H}_{k-1}$. With  \eqref{eq:p_k_recursive}, we have
 \begin{align}
  p_{k+1} &= \frac{\beta_k}{\beta_{k-1}}(p_k+P_{k-1}\phi_{k-1})-\beta_kA(p_k+P_{k-1}\phi_{k-1}) + \beta_kAP_k\Gamma_k-P_k\phi_k   \nonumber \\  
  & = P_k\left( \frac{\beta_k}{\beta_{k-1}}\begin{pmatrix}
   \phi_{k-1}  \\
   1
  \end{pmatrix} -\beta_k\bar{H}_{k-1}\bigl(\phi_{k-1}-\Gamma_k^{[m_k-1]}\bigr) -\phi_k \right)   \nonumber  \\
  & \quad - \beta_kAp_k\bigl(1-\Gamma_k^{(m_k)}\bigr).   \nonumber 
 \end{align}
 Hence, by rearrangement, we obtain \eqref{eq:AP_k}, thus completing the induction. 
 
 Suppose that $m_k\geq 1$. We prove $1-\Gamma_k^{(m_k)}\neq 0$ by contradiction. If $\Gamma_k^{(m_k)} = 1$, then $\bar{r}_k = r_k-Q_k\Gamma_k = r_k-Q_{k-1}\Gamma_k^{[m_k-1]}-q_k = r_k-Q_{k-1}\Gamma_k^{[m_k-1]}-(\Delta r_{k-1}-Q_{k-1}\zeta_k) = r_{k-1}-Q_{k-1}\bigl(\Gamma_k^{[m_k-1]}-\zeta_k\bigr) \perp {\rm range}(V_k)$. Hence $\bar{r}_k = \bar{r}_{k-1}$ due to $\bar{r}_k \perp {\rm range}(V_{k-1})$. 
  For the Type-I method, $v_k = p_k = \beta_{k-1}\bar{r}_{k-1}-P_{k-1}\phi_{k-1}$, so $0 = \bar{r}_{k-1}^\mathrm{T}p_k = \beta_{k-1}\bar{r}_{k-1}^\mathrm{T}\bar{r}_{k-1}$, which indicates $\bar{r}_{k-1} = 0$. For the Type-II method, $v_k = q_k = -Ap_k = -\beta_{k-1}A\bar{r}_{k-1}-Q_{k-1}\phi_{k-1}$, so $0 = \bar{r}_{k-1}^\mathrm{T}q_k = -\beta_{k-1}\bar{r}_{k-1}^\mathrm{T}A\bar{r}_{k-1}$, which indicates $\bar{r}_{k-1} = 0$ since $A$ is positive definite. However, $\bar{r}_{k-1} = 0$ yields that $r_k = (I-\beta_{k-1}A)\bar{r}_{k-1} = 0$, which is impossible because the algorithm has not found the exact solution. As a result, $1-\Gamma_k^{(m_k)}\neq 0$. Thus \eqref{eq:H_k} and \eqref{eq:h_k_2} are well defined. 
 \end{proof}

\subsection{Proof of Lemma~\ref{lemma:diff:Hessenberg}}  \label{subsec:proof:lemma:diff:Hessenberg}
\begin{proof}
 From \eqref{lemma:diff:Gamma_k} in the proof of \cref{lemma:diff} and the assumption $|1-\Gamma_k^{(m_k)}| \geq \tau_0$, with sufficiently small $\|x_{0}-x^*\|_2$, we can ensure 
 $|\Gamma_k^{(m_k)}-\hat{\Gamma}_k^{(m_k)}| = \hat{\kappa}\mathcal{O}(\|x_{k-m_k}-x^*\|_2)$ and $|1-\hat{\Gamma}_k^{(m_k)}|\geq \frac{1}{2}\tau_0$. Thus 
 \begin{equation}
  \left\vert\frac{1}{1-\Gamma_k^{(m_k)}} - \frac{1}{1-\hat{\Gamma}_k^{(m_k)}}\right\vert
  =\frac{\vert\Gamma_k^{(m_k)}-\hat{\Gamma}_k^{(m_k)}\vert}{\vert(1-\Gamma_k^{(m_k)})(1-\hat{\Gamma}_k^{(m_k)})\vert} = \hat{\kappa}\mathcal{O}(\|x_{k-m_k}-x^*\|_2).  \label{lemma:diff:inv_Gamma}
 \end{equation}
 
 We prove \eqref{lemma:H_k} 
 by induction. The same as $h_k, h_k^{(m_k+1)}, H_k, \bar{H}_k, \phi_k$ in Process~I, the notations $\hat{h}_k, \hat{h}_k^{(m_k+1)}, \hat{H}_k, \bar{\hat{H}}_k, \hat{\phi}_k$ are defined for Process~II, correspondingly. 
 
 If $m_k = 1$, then 
  \begin{align}
   \left\vert h_k - \hat{h}_k\right\vert &= \left\vert\frac{1}{1-\Gamma_k}\left(\frac{1}{\beta_{k-1}}-\frac{1}{\beta_k}\phi_k\right)-\frac{1}{1-\hat{\Gamma}_k}\left( \frac{1}{\beta_{k-1}}-\frac{1}{\beta_k}\hat{\phi}_k \right)\right\vert   \nonumber \\
   & \leq \left\vert \frac{1}{1-\Gamma_k}\cdot \frac{\hat{\phi}_k-\phi_k}{\beta_k} \right\vert + \left\vert \left( \frac{1}{1-\Gamma_k}-\frac{1}{1-\hat{\Gamma}_k}\right)\cdot \left( \frac{1}{\beta_{k-1}}-\frac{1}{\beta_k}\hat{\phi}_k\right) \right\vert   \nonumber \\
   & = \hat{\kappa}\mathcal{O}(\|x_{k-m_k}-x^*\|_2),
  \end{align}
  because of \eqref{lemma:diff:inv_Gamma}, and \eqref{lemma:diff:zeta_k}, \eqref{lemma:diff:Gamma_k} in the proof of \cref{lemma:diff}. Also, $\|H_k\|_2 = \vert h_k\vert = \mathcal{O}(1)$. 
  
  Suppose that $m_k\geq 2$, and as an inductive hypothesis, $\|H_{k-1}-\hat{H}_{k-1}\|_2 = \hat{\kappa}\mathcal{O}(\|x_{k-m_k}-x^*\|_2)$, $\|H_{k-1}\|_2 = \mathcal{O}(1)$. First, due to \eqref{lemma:diff:inv_Gamma}, we have
  \begin{equation*}
   \left\vert h_{k-1}^{(m_k)}- \hat{h}_{k-1}^{(m_k)}\right\vert 
   = \frac{1}{\beta_{k-1}}\left\vert \frac{1}{1-\Gamma_{k-1}^{(m_k-1)}} - \frac{1}{1-\hat{\Gamma}_{k-1}^{(m_k-1)}} \right\vert = \hat{\kappa}\mathcal{O}(\|x_{k-m_k}-x^*\|_2).  
  \end{equation*}
  Also, $\vert h_{k-1}^{(m_k)}\vert \leq \frac{1}{\beta\tau_0}$, and $m_k\leq m$. 
 Thus for $\bar{H}_{k-1}$ and $\bar{\hat{H}}_{k-1}$, we have that 
 \begin{equation} 
 \|\bar{H}_{k-1}-\bar{\hat{H}}_{k-1}\|_2 = \hat{\kappa}\mathcal{O}(\|x_{k-m_k}-x^*\|_2), \quad \|\bar{H}_{k-1}\|_2 = \mathcal{O}(1).  \label{lemma:diff:H_k_bar}
 \end{equation}
 As a result, 
 \begin{align}
  &\quad \left\Vert\bar{H}_{k-1}(\phi_{k-1}-\Gamma_k^{[m_k-1]})-\bar{\hat{H}}_{k-1}(\hat{\phi}_{k-1}-\hat{\Gamma}_k^{[m_k-1]})\right\Vert_2  \nonumber \\
  & \leq \left\Vert \bar{H}_{k-1}\left(\phi_{k-1}-\Gamma_k^{[m_k-1]}-
  (\hat{\phi}_{k-1}-\hat{\Gamma}_k^{[m_k-1]})\right) \right\Vert_2   \nonumber \\ 
  & \quad + \left\Vert \left( \bar{H}_{k-1}-\bar{\hat{H}}_{k-1}\right)\left( \hat{\phi}_{k-1}-\hat{\Gamma}_k^{[m_k-1]}\right) \right\Vert_2   \nonumber \\
  & \leq \hat{\kappa}\mathcal{O}(\|x_{k-m_k}-x^*\|_2),   \nonumber 
 \end{align}
 and $\|\bar{H}_{k-1}(\phi_{k-1}-\Gamma_k^{[m_k-1]})\|_2 = \mathcal{O}(1)$. 
 Besides, 
 \begin{align}
  & \left\Vert\frac{1}{\beta_{k-1}}\begin{pmatrix}
   \phi_{k-1} \\
   1
  \end{pmatrix}-\frac{1}{\beta_k}\phi_k  
  - \left( 
   \frac{1}{\beta_{k-1}}\begin{pmatrix}
   \hat{\phi}_{k-1} \\
   1
  \end{pmatrix}-\frac{1}{\beta_k}\hat{\phi}_k    
  \right)
  \right\Vert_2 = \hat{\kappa}\mathcal{O}(\|x_{k-m_k}-x^*\|_2),   \nonumber  \\
  & \left\Vert\frac{1}{\beta_{k-1}}\begin{pmatrix}
   \phi_{k-1} \\
   1
  \end{pmatrix}-\frac{1}{\beta_k}\phi_k \right\Vert_2 = \mathcal{O}(1).   \nonumber 
 \end{align}
 Therefore, $\Vert(1-\Gamma_k^{(m_k)})h_k -(1-\hat{\Gamma}_k^{(m_k)})\hat{h}_k\Vert_2 = \hat{\kappa}\mathcal{O}(\|x_{k-m_k}-x^*\|_2)$, $\|(1-\Gamma_k^{(m_k)})h_k\|_2 = \mathcal{O}(1)$. Hence,
 \begin{align*}
  &\|h_k-\hat{h}_k\|_2 \leq \left\Vert \frac{1}{1-\Gamma_k^{(m_k)}}\left( (1-\Gamma_k^{(m_k)})h_k -(1-\hat{\Gamma}_k^{(m_k)})\hat{h}_k \right) \right\Vert_2  \nonumber \\
  & \qquad + \left\Vert\left( \frac{1}{1-\Gamma_k^{(m_k)}}-\frac{1}{1-\hat{\Gamma}_k^{(m_k)}} \right)\cdot \left( 1-\hat{\Gamma}_k^{(m_k)}\right)\hat{h}_k \right\Vert_2 = \hat{\kappa}\mathcal{O}(\|x_{k-m_k}-x^*\|_2), 
 \end{align*}
 and $\|h_k\|_2 = \mathcal{O}(1)$, which together with \eqref{lemma:diff:H_k_bar} and $m_k\leq m$  implies that \eqref{lemma:H_k} holds. Thus we complete the induction. 
 \end{proof}

\section{Proofs of Section~\ref{sec:short-term}}

 \subsection{Proof of Theorem~\ref{them:short_term_convergence}}  \label{proof:them:short_term_convergence}
\begin{proof}
 Consider the two processes defined in \cref{definition:two_processes}. Here, we replace the restarted AM method by the restarted ST-AM method.
  Note that the restarted ST-AM is obtained from the restarted AM by setting $\zeta_k^{(j)} = 0$ for $j\leq k-3$, and 
   $\Gamma_k^{(j)} = 0$ for $ j \leq k-2$. Similar to \cref{lemma:diff}, it can be proved that 
  \begin{equation}
  \|r_k - \hat{r}_k\|_2 = \hat{\kappa}\cdot\mathcal{O}(\|x_{k-m_k}-x^*\|_2^2),   ~~
   \| x_{k+1} - \hat{x}_{k+1}\|_2 = \hat{\kappa}\cdot\mathcal{O}(\|x_{k-m_k}-x^*\|_2^2), 
   \label{them:diff:x_k}
  \end{equation}
 provided that there exists a constant $\eta_0 > 0$ such that 
 \begin{equation}
  \|r_j\|_2 \leq \eta_0 \|r_0\|_2,  ~ j=0,\dots,k,   \label{them:short-term:r_j}
 \end{equation}
  and $x_0\in \mathcal{B}_{\hat{\rho}}(x^*)$ is sufficiently close to $x^*$. 
 
 Since $\theta_k = \|I-\beta_kA\|_2 \leq \theta$, there are positive constants $\beta, \beta'$ such that $\beta\leq \beta_k\leq \beta'$. In fact, 
 by choosing $\beta_k \in \bigl[\frac{1-\theta}{\mu},\frac{1+\theta}{L}\bigr]$, we can ensure $\theta_k \leq \max\{|1-\beta_kL|,|1-\beta_k\mu| \}\leq \theta<1$. 
 We give the proof of the Type-I method here.  
 
 For the restarted Type-I ST-AM method, if $x_0$ is sufficiently close to $x^*$, then
 \begin{equation}
  \|x_j-x^*\|_A \leq \|x_0-x^*\|_A, ~ j=0,\dots,k.   \label{them:short-term:x_A}
 \end{equation}
 We prove \eqref{them:short-term:r_j} and \eqref{them:short-term:x_A} hold for the Type-I method by induction. 
 For $k=0$, \eqref{them:short-term:r_j} and \eqref{them:short-term:x_A} hold. Suppose that for $k\geq 0$, the results hold  for $k$. We establish the results for $k+1$. Let $x_k^\mathrm{A}$ and $r_k^\mathrm{A}$ denote the $m_k$-th iterate and residual of Arnoldi's method applied to solve $\hat{h}(x)=0$, with the starting point $x_{k-m_k}$. Due to \cref{prop:spd_linear} and \cref{prop:linear},  we have  $\bar{\hat{x}}_k = x_k^\mathrm{A}$. 
 Hence 
 \begin{align}
 \|\hat{x}_{k+1} - x^*\|_A &\leq \theta_{k}\|\bar{\hat{x}}_{k}-x^*\|_A
  =  \theta_{k}\|x_{k}^\mathrm{A}-x^*\|_A    \nonumber   \\
  & = \theta_{k}\min_{\mathop{}_{p(0)=1}^{p\in \mathcal{P}_{m_{k}}}}\|p(A)(x_{k-m_{k}}-x^*)\|_A \leq \theta_{k}\|x_{k-m_{k}}-x^*\|_A.		\label{them:hat_x_k_type1}
 \end{align}
 Here, we use the fact that $\|I-\beta_kA\|_A = \|I-\beta_kA\|_2$. 
  From \eqref{them:diff:x_k}, it follows that $\|x_{k+1}-\hat{x}_{k+1}\|_A = \hat{\kappa}\mathcal{O}(\|x_{k-m_k}-x^*\|_A^2)$. 
  Then, there is a constant $c_1>0$ such that $ \|x_{k+1}-\hat{x}_{k+1}\|_A \leq \hat{\kappa}c_1\|x_{k-m_k}-x^*\|_A^2$.  With \eqref{them:hat_x_k_type1}, we have  
 \begin{align}
  \|x_{k+1}-x^*\|_A \leq \theta \|x_{k-m_k}-x^*\|_A+ \hat{\kappa}c_1\|x_{k-m_k}-x^*\|_A^2.
 \end{align}
 Then $\|x_{k+1}-x^*\|_A \leq \frac{1+\theta}{2}\|x_{k-m_k}-x^*\|_A$ provided $\|x_{k-m_k}-x^*\|_A\leq \frac{1-\theta}{2\hat{\kappa}c_1}$, which can be satisfied  by choosing $\|x_0-x^*\|_2 \leq \frac{1-\theta}{2\sqrt{L}\hat{\kappa}c_1}$, since by the inductive hypothesis, $\|x_{k-m_k}-x^*\|_A \leq \|x_0-x^*\|_A \leq \sqrt{L}\|x_0-x^*\|_2$. 
 Thus, $\|x_{k+1}-x^*\|_A < \|x_{k-m_k}-x^*\|_A\leq \|x_0-x^*\|_A$, namely \eqref{them:short-term:x_A} for $k+1$. Also, $\|x_{k+1}-x^*\|_2 \leq \frac{1}{\sqrt{\mu}}\|x_{k+1}-x^*\|_A \leq \frac{1}{\sqrt{\mu}}\|x_0-x^*\|_A \leq \frac{\sqrt{L}}{\sqrt{\mu}}\|x_0-x^*\|_2$. So we can impose $\|x_0-x^*\|_2\leq \frac{\sqrt{\mu}\hat{\rho}}{\sqrt{L}}$ to ensure $x_{k+1}\in\mathcal{B}_{\hat{\rho}}(x^*)$, which further yields that
 $ \|r_{k+1}\|_2 \leq L\|x_{k+1}-x^*\|_2\leq \frac{L\sqrt{L}}{\sqrt{\mu}}\|x_0-x^*\|_2\leq \frac{L\sqrt{L}}{\mu\sqrt{\mu}}\|r_0\|_2$, 
 namely \eqref{them:short-term:r_j} for $k+1$, and $\eta_0 = \frac{L\sqrt{L}}{\mu\sqrt{\mu}}$. Hence, we complete the induction. 
 
 Since $A$ is SPD, we can use the Chebyshev polynomial to obtain
  \begin{align}
  \min_{\mathop{}_{p(0)=1}^{p\in \mathcal{P}_{m_{k}}}}\|p(A)\|_2
   \leq \min_{\mathop{}_{p(0)=1}^{p\in \mathcal{P}_{m_{k}}}}\max_{\lambda \in [\mu,L]}|p(\lambda)|   
   \leq 2\left(\frac{\sqrt{L/\mu}-1}{\sqrt{L/\mu}+1}\right)^{m_k},    \label{them:chebyshev}
 \end{align}
 which is a classical result \cite[Section~6.11.3]{saad2003iterative}. Note that $\|p(A)(x_{k-m_k}-x^*)\|_A \leq \|p(A)\|_A\|x_{k-m_k}-x^*\|_A = \|p(A)\|_2\|x_{k-m_k}-x^*\|_A$. 
 Thus, by choosing $x_0$ sufficiently close to $x^*$,  \eqref{them:short_term_type1} holds as a result of \eqref{them:hat_x_k_type1}, \eqref{them:chebyshev}, and \eqref{them:diff:x_k}. 
 
 For the Type-II method, since $\theta_k\leq \theta < 1$, the bound \eqref{them:type2} can be established following the similar approach to proving  \Cref{them:am}. With \eqref{them:chebyshev}, the bound  \eqref{them:short_term_type2} holds. 
 \end{proof}

\subsection{Proof of Theorem~\ref{them:short-term:eigen_estimate}}  \label{proof:them:short-term:eigen_estimate}
\begin{proof}
 The same as $t_k^{(m_k+1)}, T_k $ in Process~I, the notations $\hat{t}_k^{(m_k+1)}, \hat{T}_k$ are defined for Process~II, correspondingly.
  In this case, the tridiagonal matrix $\hat{T}_k$ can be diagonalized. Let $A:=h'(x^*)$. Then 
  $A\hat{Q}_k = \hat{Q}_k\hat{T}_k+\hat{t}_k^{(m_k+1)}\hat{q}_{k+1}e_{m_k}^\mathrm{T}$.  Hence
  \begin{align*}
   \hat{V}_k^\mathrm{T}A\hat{Q}_k = \hat{V}_k^\mathrm{T}\hat{Q}_k\hat{T}_k,
  \end{align*}
 due to $\hat{V}_k^\mathrm{T}\hat{q}_{k+1} = 0$. Thus $\hat{T}_k = (\hat{V}_k^\mathrm{T}\hat{Q}_k)^{-1}\hat{V}_k^\mathrm{T}A\hat{Q}_k$. Here,  $\hat{V}_k^\mathrm{T}\hat{Q}_k$ and $\hat{V}_k^\mathrm{T}A\hat{Q}_k$ are  symmetric for both types of ST-AM methods. Define $\hat{W}_k=-\hat{V}_k^\mathrm{T}\hat{Q}_k$ for the Type-I method, and $\hat{W}_k = \hat{V}_k^\mathrm{T}\hat{Q}_k$ for the Type-II method. 
 Then
 \begin{equation}
  \hat{W}_k^{1/2}\hat{T}_k\hat{W}_k^{-1/2}  
   = \mp\hat{W}_k^{-1/2}(\hat{V}_k^\mathrm{T}A\hat{Q}_k)\hat{W}_k^{-1/2},  \label{eq:T_k_similar}
 \end{equation}
 where the sign is ``$-$'' for the Type-I method, and ``$+$'' for the Type-II method. 
 The right side in \eqref{eq:T_k_similar} is symmetric, so there exists an  orthonormal matrix $\hat{U}_k\in \mathbb{R}^{m_k\times m_k}$ such that 
 \begin{align}
  \hat{T}_k = \hat{W}_k^{-1/2}\hat{U}_k^\mathrm{T}\hat{D}_k\hat{U}_k\hat{W}_k^{1/2},
 \end{align}
 where $\hat{D}_k$ is a diagonal matrix formed by the eigenvalues of $\hat{T}_k$. Also, 
 similar to the proof of \Cref{lemma:diff}, 
 the relations \eqref{ineq:vq_lower_2}, \eqref{lemma:p_k}, and \eqref{lemma:diff:p_k} also hold for the ST-AM methods.  Note that $\hat{V}_k^\mathrm{T}\hat{Q}_k$ is diagonal. We have 
 \begin{align*}
  \|\hat{V}_k^\mathrm{T}\hat{Q}_k\|_2\|(\hat{V}_k^\mathrm{T}\hat{Q}_k)^{-1}\|_2 
   = \frac{\max_{k-m_k+1\leq i\leq k}\lbrace\vert \hat{v}_i^\mathrm{T}\hat{q}_i \vert\rbrace}{\min_{k-m_k+1\leq j\leq k}\lbrace\vert \hat{v}_j^\mathrm{T}\hat{q}_j \vert\rbrace} = \mathcal{O}(1). 
 \end{align*}
 Thus $\|\hat{W}_k^{1/2}\|_2 \|\hat{W}_k^{-1/2}\|_2 = \mathcal{O}(1)$. Also, 
  similar to \cref{lemma:diff:Hessenberg}, we have $\|T_k-\hat{T}_k\|_2 = \hat{\kappa}\mathcal{O}(\|x_{k-m_k}-x^*\|_2)$. 
 Hence, the result \eqref{them:diff:short_term_lambda} follows from Bauer-Fike theorem. 
 \end{proof}

\section{Additional experimental results}
\label{appendix:additional_expr}
  
 \subsection{Solving linear systems} 
  To verify the theoretical properties of the AM and ST-AM methods for solving linear systems,  we considered 
  solving 
 \begin{align}   \label{eq:linear_system}
  Ax = b,
 \end{align}
 where $A\in\mathbb{R}^{d\times d}, b\in\mathbb{R}^d$, the residual is defined as  $r_k = b-Ax_k$ at $x_k$. 
   The fixed-point iteration is the Richardson's iteration $x_{k+1} = x_k+\beta r_k$, where $\beta$ was chosen to ensure linear convergence.  For the restarted AM and restarted ST-AM,  the restarting conditions were disabled since \eqref{eq:linear_system} is linear.    
   AM-I and AM-II used \eqref{eq:beta_k} with $\beta_0 = 1$ to choose $\beta_k$; ST-AM-I and ST-AM-II used \eqref{eq:symmetric:beta_k} with $\beta_0 = 1$  to choose $\beta_k$.

 \subsubsection{Nonsymmetric linear system} The matrix $A\in\mathbb{R}^{100\times 100}$ was randomly generated from Gaussian distribution and was further modified by making all the eigenvalues have positive real parts. 
 
 The results are shown in \Cref{fig:linear:rk_bar} and \Cref{fig:linear:eigs}. 
  The convergence behaviours of $\|\bar{r}_k\|_2/\|r_0\|_2$ and $\|r_k\|_2/\|r_0\|_2$  verify \cref{prop:modified_seqs}, \cref{prop:linear} and \cref{them:am}.  
  The eigenvalue estimates well approximate the exact eigenvalues of $A$, which justifies the adaptive mixing strategy.
 
 \begin{figure}[ht]
\centering 
\subfigure{
\includegraphics[width=0.44\textwidth]{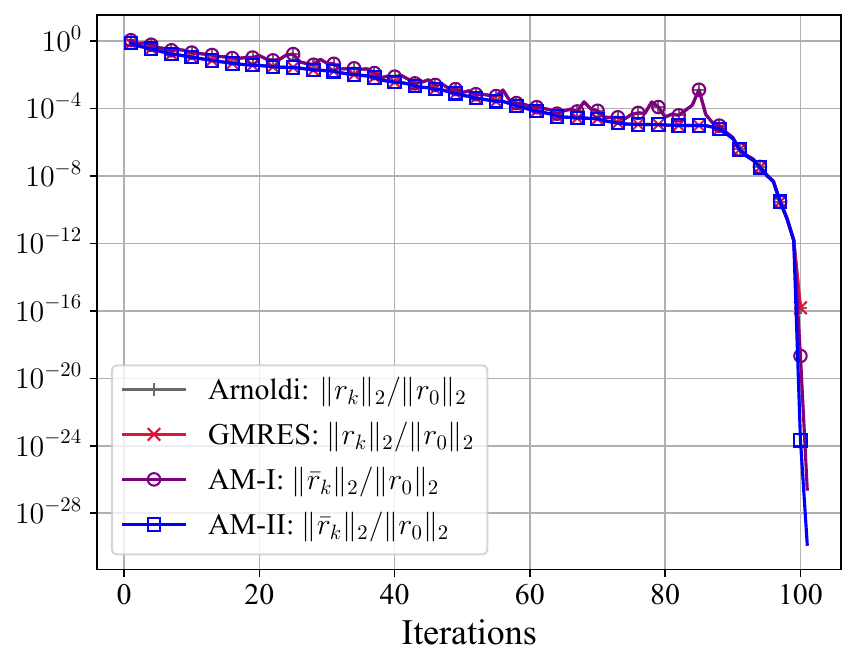}
}
\subfigure{
\includegraphics[width=0.46\textwidth]{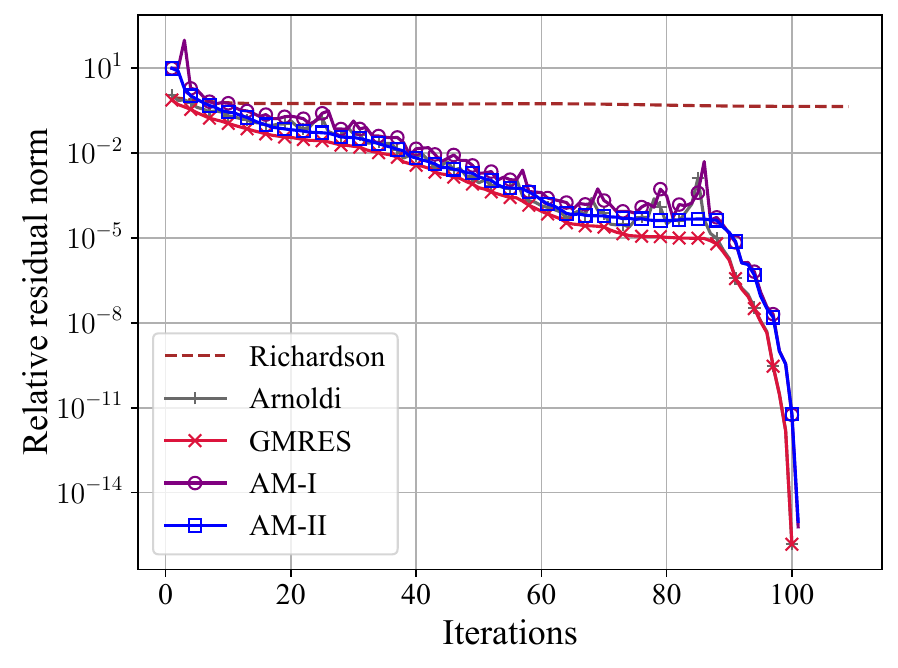}
}
\caption{ Results of solving the nonsymmetric linear system. Left:  
$ \|r_k\|_2/\|r_0\|_2$ of Arnoldi's method and GMRES, and  $ \|\bar{r}_k\|_2/\|r_0\|_2 $ of AM-I and AM-II.  Right: $ \|r_k\|_2/\|r_0\|_2$ of each method.}
\label{fig:linear:rk_bar}
\end{figure}

\begin{figure}[ht]
\centering 
\subfigure{
\includegraphics[width=0.46\textwidth]{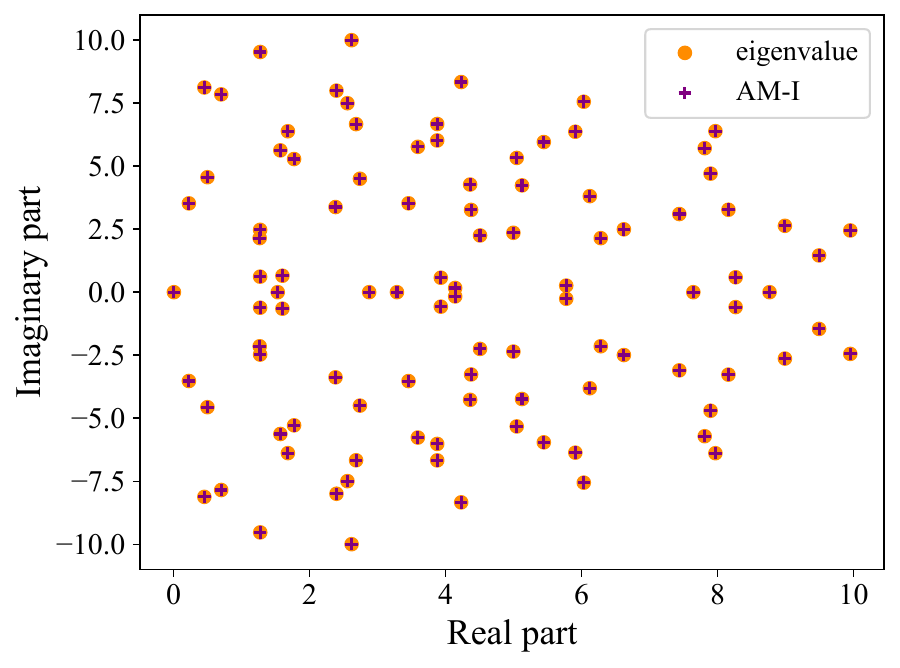}
}
\subfigure{
\includegraphics[width=0.46\textwidth]{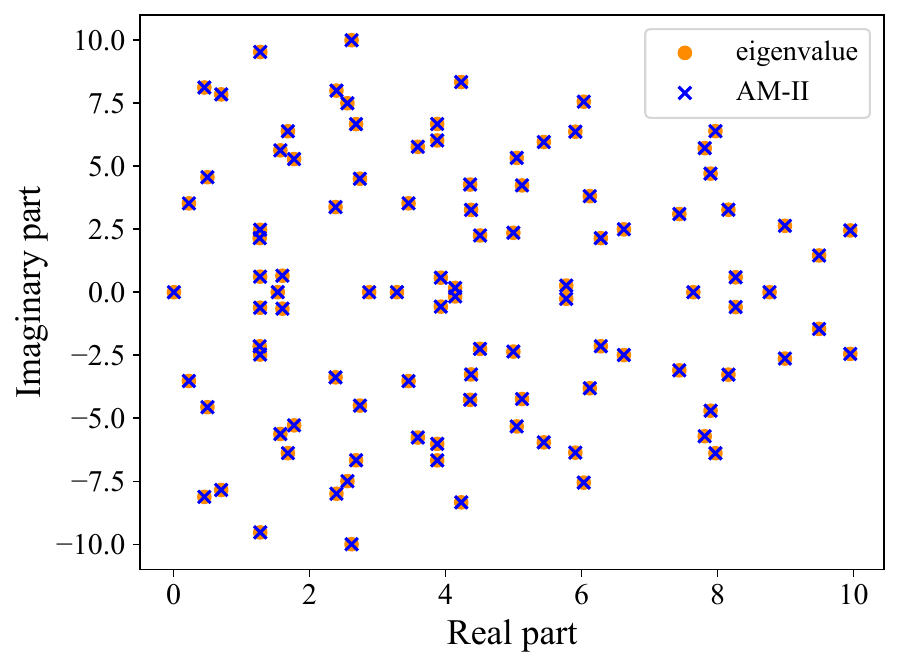}
}
\caption{ Results of solving the nonsymmetric linear system.  The eigenvalues of $A$; the eigenvalue estimates from AM-I and AM-II at the last iteration.}
\label{fig:linear:eigs}
\end{figure}  
  


 \subsubsection{SPD linear system}  We first generated a matrix $B\in\mathbb{R}^{100\times 100}$ from Gaussian distribution, then chose $A=B^\mathrm{T}B$.  
 In this case,  the conjugate gradient (CG) method \cite{hestenes1952methods} and the conjugate residual (CR) method \cite[Algorithm~6.20]{saad2003iterative}, which have short-term recurrences, are equivalent to Arnoldi's method and GMRES, respectively. 
  
  The results in \Cref{fig:spd_linear:rk_bar} verify the properties of the ST-AM methods.  We see the intermediate residuals $\lbrace \bar{r}_k\rbrace$ of ST-AM-I/ST-AM-II match the residuals  $\lbrace r_k\rbrace$ of the CG/CR method during the first 30 iterations. However, the equivalence cannot exactly hold in the later iterations due to the loss of global orthogonality in finite arithmetic.  Nonetheless, the convergence of ST-AM-I/ST-AM-II is comparable to that of CG/CR. 
  \Cref{fig:spd_linear:eigs} shows that the eigenvalue estimates from ST-AM well approximate the exact eigenvalues of $A$. 
 
 \begin{figure}[ht]
\centering 
\subfigure{
\includegraphics[width=0.44\textwidth]{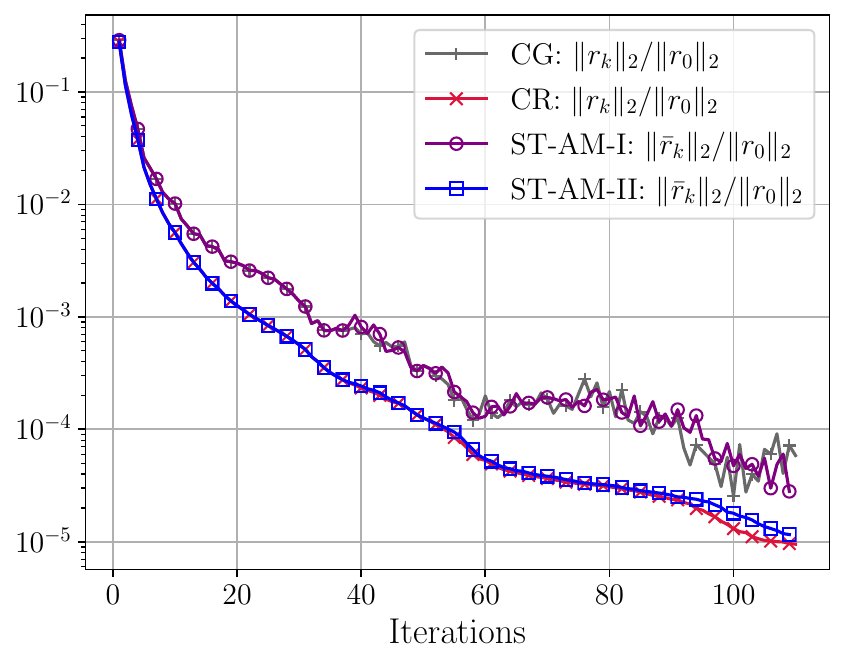}
}
\subfigure{
\includegraphics[width=0.46\textwidth]{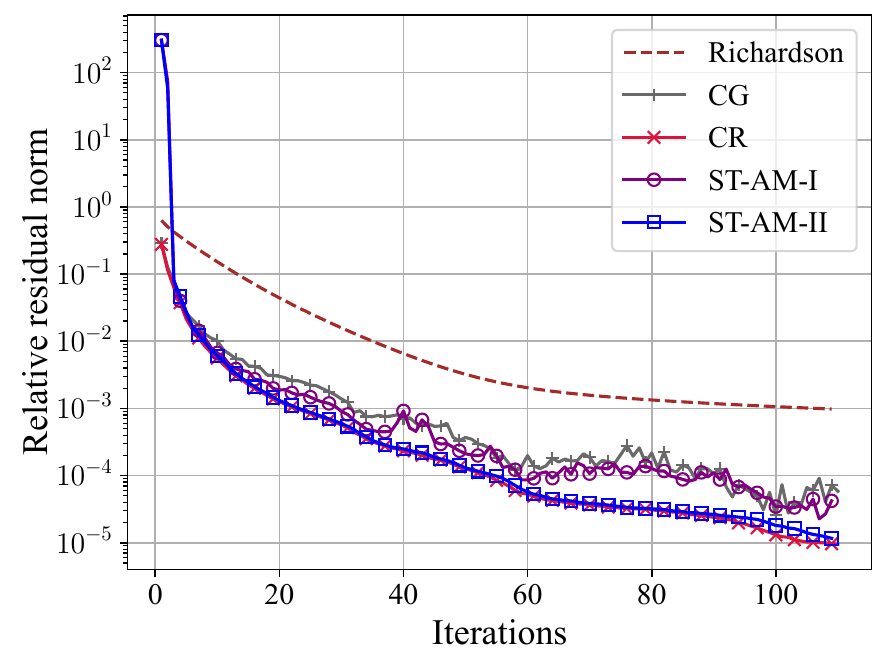}
}
\caption{Results of solving the SPD linear system. Left: $ \|r_k\|_2/\|r_0\|_2$ of CG and CR, and $ \|\bar{r}_k\|_2/\|r_0\|_2 $ of ST-AM-I and ST-AM-II.  Right: $ \|r_k\|_2/\|r_0\|_2$ of each method. }
\label{fig:spd_linear:rk_bar}
\end{figure}  

\begin{figure}[ht]
\centering 
\subfigure{
\includegraphics[width=0.46\textwidth]{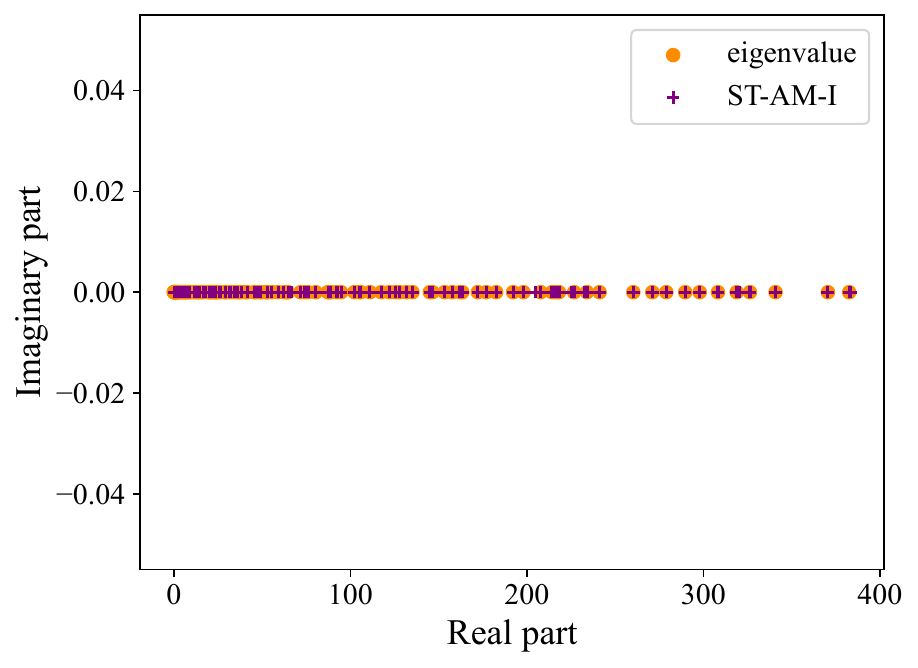}
}
\subfigure{
\includegraphics[width=0.46\textwidth]{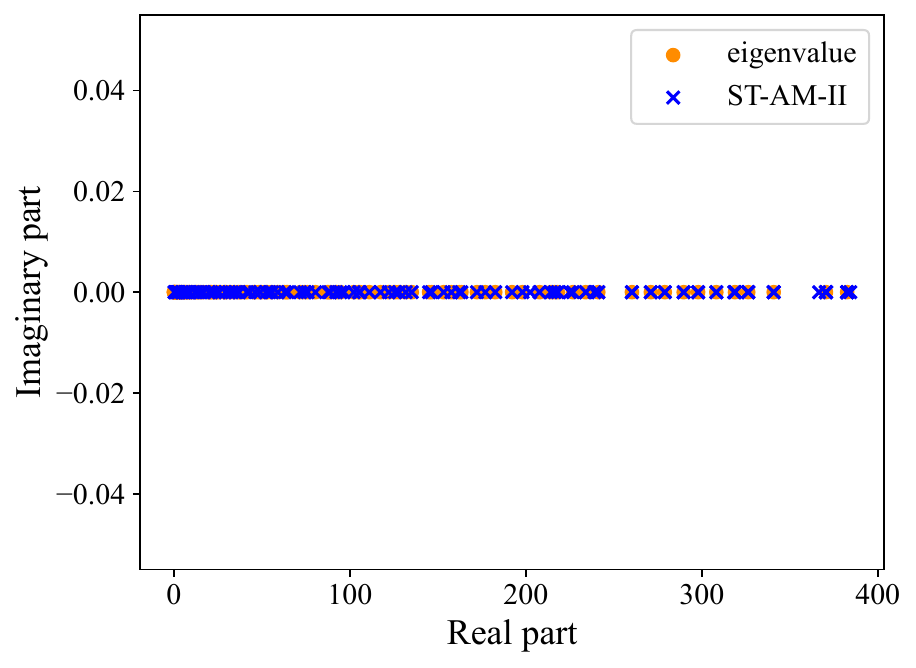}
}
\caption{Results of solving the SPD linear system.  The eigenvalues of $A$;  the eigenvalue estimates from ST-AM-I and ST-AM-II at the last iteration. }
\label{fig:spd_linear:eigs}
\end{figure}

 
	

\subsection{Additional results of solving the modified Bratu problems}
  We provide details about the eigenvalue estimates and show the effect of $\beta_k$ on the convergence. 
 
 \subsubsection{Nonsymmetric Jacobian}
  To verify \Cref{them:eigen_estimate}, 
  we compared the eigenvalue estimates with the Ritz values of $F'(U^*)$ where $ F(U^*) = 0$. The Ritz values were obtained from the $k$-step Arnoldi's method  \cite{saad1980variations} (denoted by Arnoldi($k$)). 
  \Cref{fig:bratu:nonsymmetric_eig} indicates that the extreme Ritz values are well approximated, which accounts for the proper choices of $\beta_k$. 
  
  \begin{figure}[ht]
\centering 
\subfigure{
\includegraphics[width=0.46\textwidth]{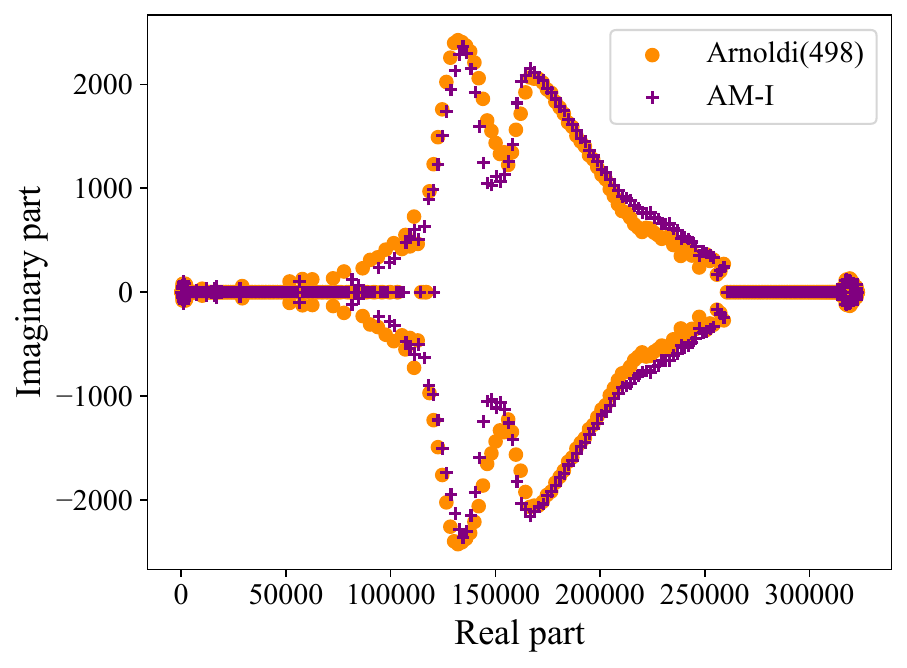}
}
\subfigure{
\includegraphics[width=0.46\textwidth]{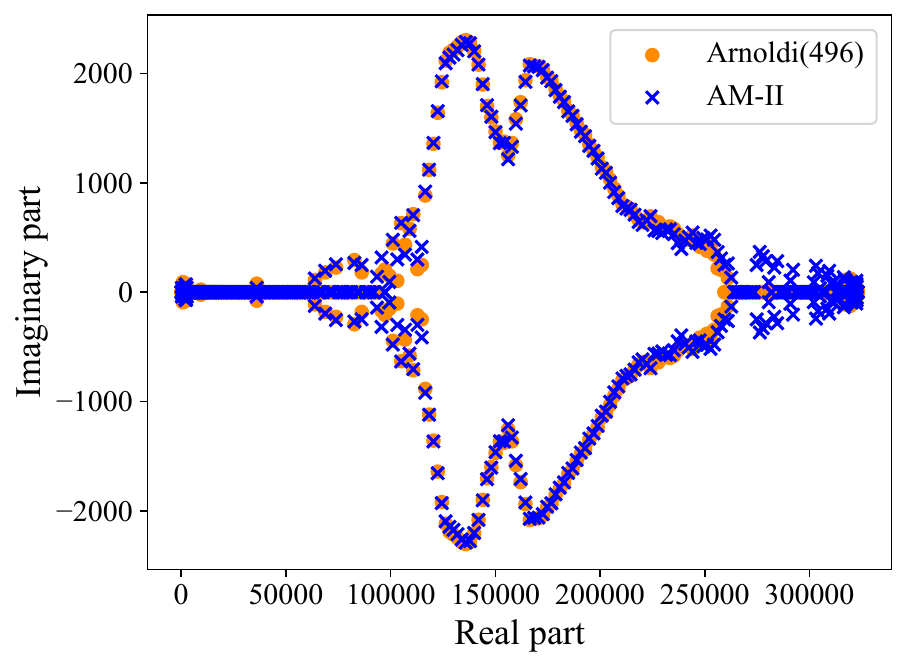}
}
\caption{The modified Bratu problem with $\alpha = 20$. The Ritz values from Arnoldi($k$) and the eigenvalue estimates from AM-I/AM-II at the last iteration. }
\label{fig:bratu:nonsymmetric_eig}
\end{figure}   
  
  We also tested AM-I and AM-II with fixed $\beta_k$. 
  \Cref{fig:bratu:nonsymmetric_betak} shows that the choice of $\beta_k$ can largely affect the convergence behaviours of both methods, and the adaptive mixing strategy performs well. It is worth noting that for the Picard iteration, choosing $\beta$ from $\{10^{-5},\dots,10^{-2} \}$ causes  divergence, which  suggests $\theta_k > 1$ in \eqref{them:type1} and \eqref{them:type2} when $\beta_k \in \{10^{-5},\dots,10^{-2} \} $. Nevertheless, the residual norms of the restarted AM methods can still converge since the minimization problems in \eqref{them:type1} and \eqref{them:type2} dominate the convergence when $m_k$ is large.

 \begin{figure}[ht]
\centering 
\subfigure{
\includegraphics[width=0.46\textwidth]{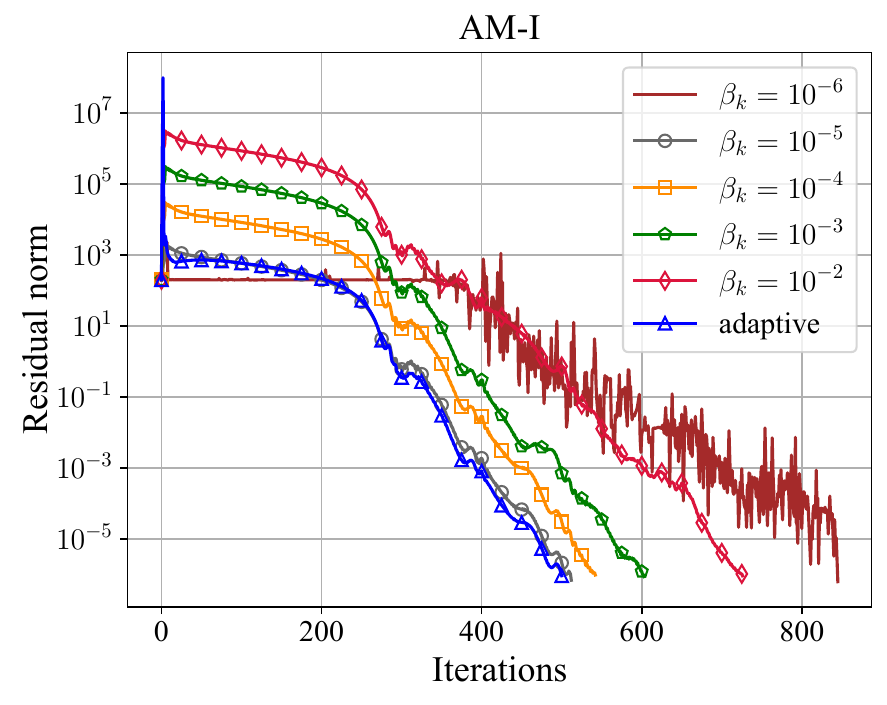}
}
\subfigure{
\includegraphics[width=0.46\textwidth]{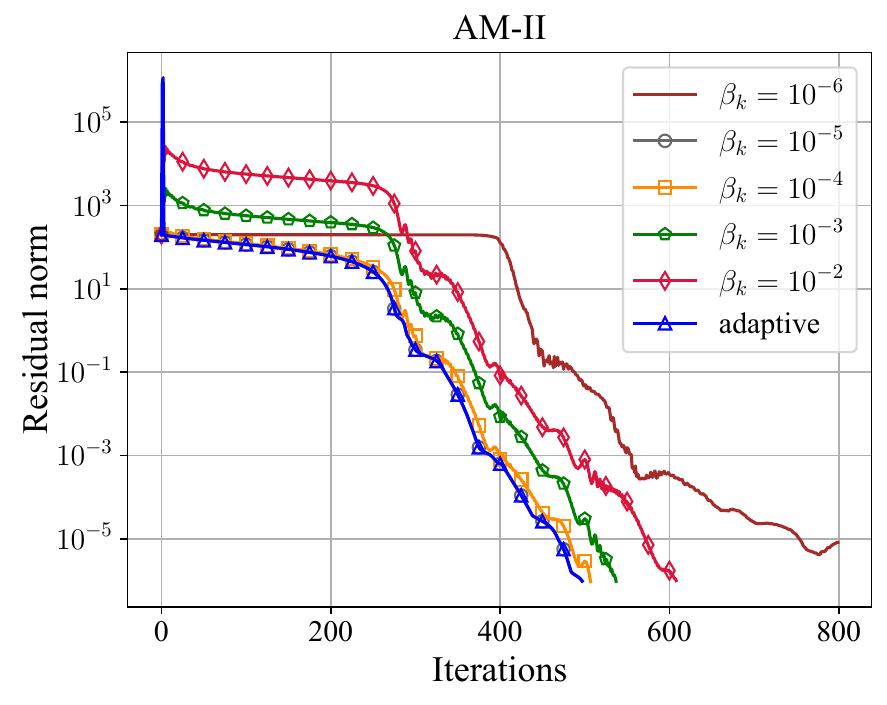}
}
\caption{The modified Bratu problem with $\alpha = 20$. Left: AM-I with different choices of $\beta_k$.  Right: AM-II with different choices of $\beta_k$. ``adaptive'' means using the adaptive mixing strategy.}
\label{fig:bratu:nonsymmetric_betak}
\end{figure} 

  \subsubsection{Symmetric Jacobian}
  \Cref{fig:bratu:symmetric:eigens} shows the eigenvalue estimates computed by ST-AM-I/ST-AM-II and the Ritz values of $F'(U^*)$ computed by the $k$-step symmetric Lanczos method \cite{golub2013matrix} (denoted by Lanczos($k$)), where $F(U^*) = 0$.  It is observed that the eigenvalue estimates  well approximate the Ritz values, which verifies \Cref{them:short-term:eigen_estimate}. 
  
 \begin{figure}[ht]
\centering 
\includegraphics[width=0.46\textwidth]{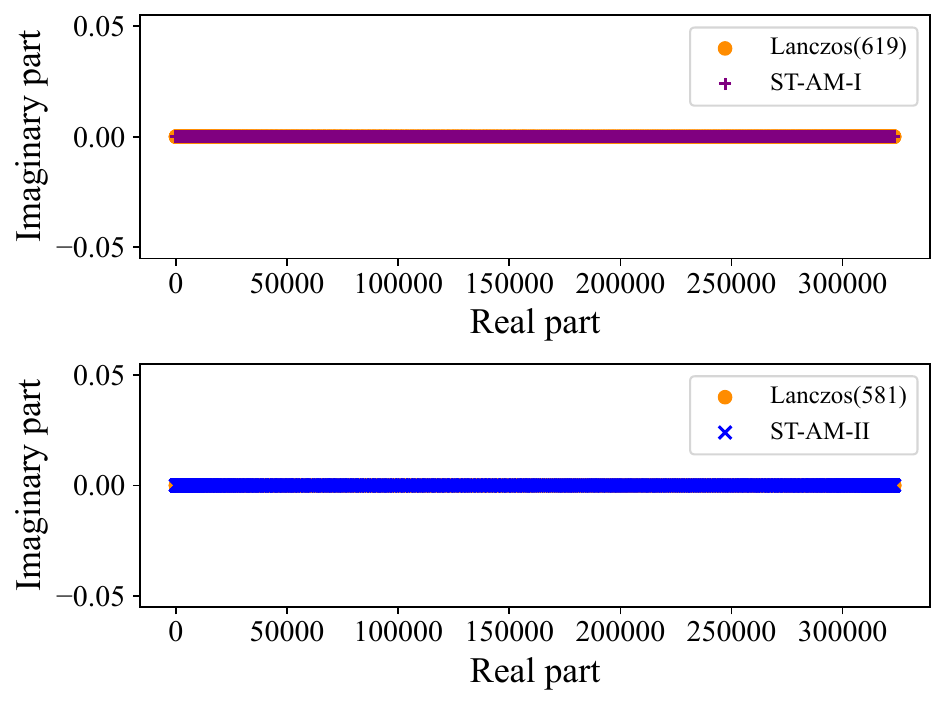}
\caption{The modified Bratu problem with $\alpha = 0$.  The Ritz values from Lanczos($m$)  and the eigenvalue estimates from ST-AM-I/ST-AM-II at the last iteration.  }
\label{fig:bratu:symmetric:eigens}
\end{figure} 


\subsection{Additional results of solving the Chandrasekhar H-equation}
  \Cref{fig:H_eq:eig} shows the Ritz values of $F'(h^*)$ and the eigenvalue estimates, where $h^*$ is the solution. 
  We computed the Ritz values of $F'(h^*)$ by applying 500 steps of Arnoldi's method to $G'(h^*)$. It is observed in \Cref{fig:H_eq:eig} that most Ritz values are nearly $1$, which accounts for the efficiency of the simple Picard iteration for solving this problem. Since the eigenvalues form 3 clusters, we also computed three eigenvalue estimates by AM-I/AM-II ($\eta=\infty, m=100, \tau=10^{-15},$ and $\beta_k \equiv 1$). We find the eigenvalue estimates still roughly match the Ritz values in the cases $\omega=0.5$ and $\omega=0.99$. For $\omega=1$, the Jacobian $F'(h^*)$ is singular, so the error in estimating the eigenvalue zero is large. 

 \begin{figure}[ht]
\centering 
\subfigure{
\includegraphics[width=0.31\textwidth]{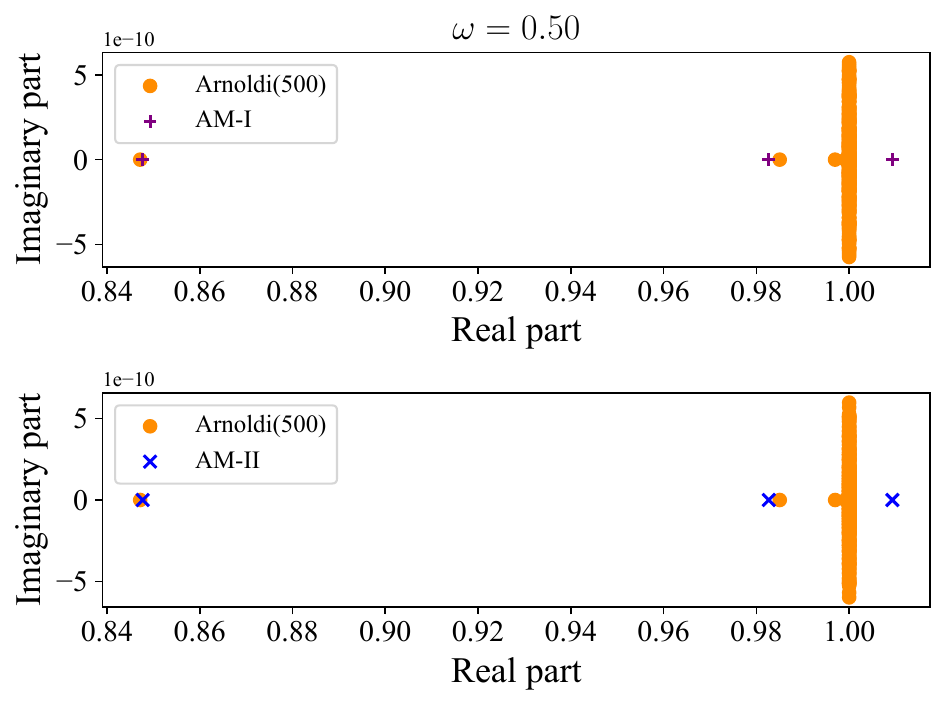}
}
\subfigure{
\includegraphics[width=0.31\textwidth]{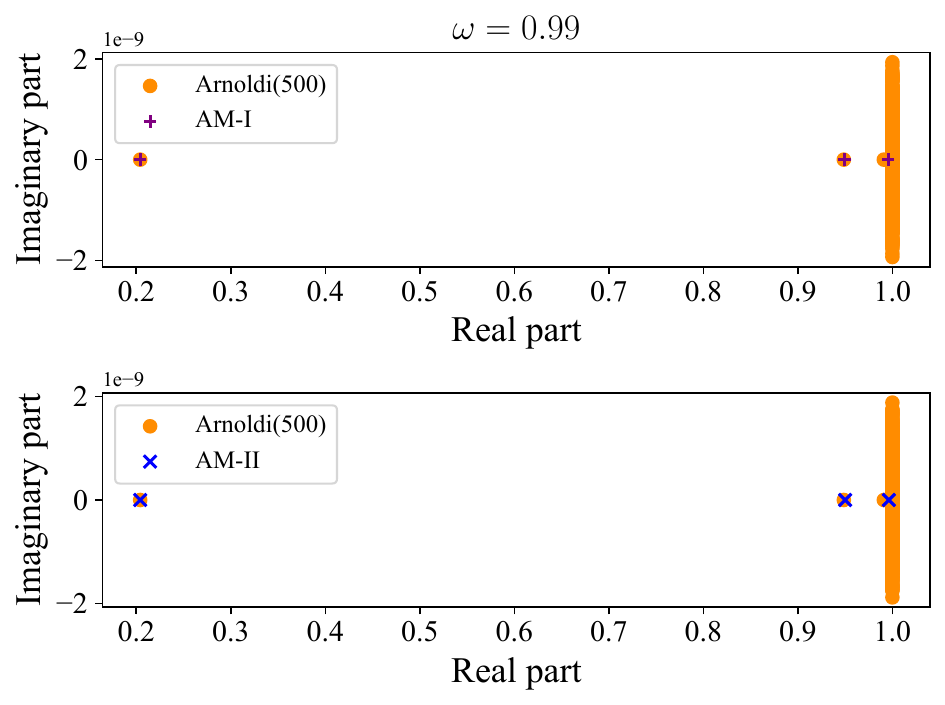}
}
\subfigure{
\includegraphics[width=0.31\textwidth]{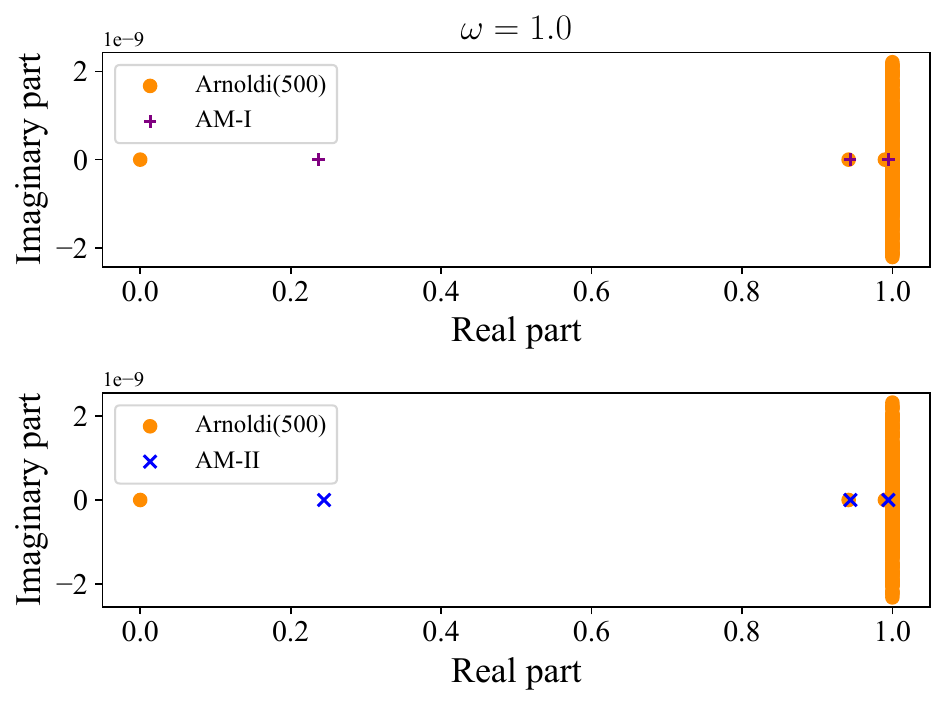}
}
\caption{ The Chandrasekhar H-equation with $\omega = 0.50, 0.99,$ and $1.0$. The Ritz values of $F'(h^*)$ computed by Arnoldi's method and the eigenvalue estimates computed by AM-I/AM-II. }
\label{fig:H_eq:eig}
\end{figure}

 \subsection{Additional results of solving the regularized logistic regression}
  \Cref{fig:logistic_regression_eig} shows the eigenvalue estimates computed by the eigenvalue estimation procedure of ST-AM-I/ST-AM-II at the last iteration. The comparison with the Ritz values of $\nabla^2 f(x^*)$ indicates that the extreme eigenvalues are well approximated. 

 \begin{figure}[ht]
\centering 
\includegraphics[width=0.46\textwidth]{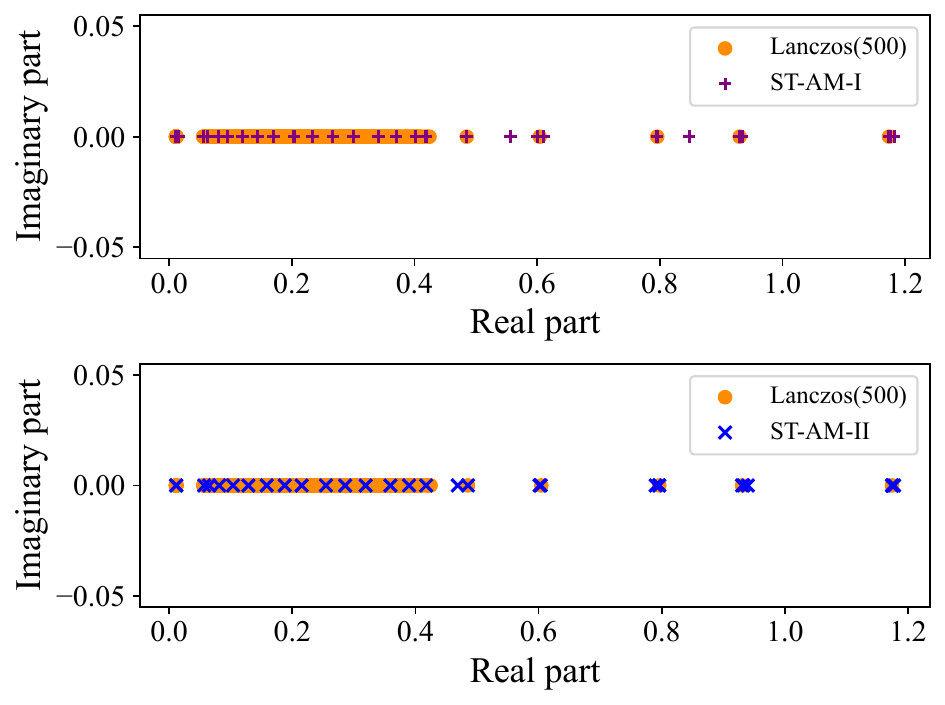}
\caption{The regularized logistic regression with $w=0.01$. The Ritz values from Lanczos method and the eigenvalue estimates from ST-AM-I/ST-AM-II at the last iteration.  }
\label{fig:logistic_regression_eig}
\end{figure}

\bibliographystyle{plainnat}
\bibliography{references3} 

\end{document}